\documentclass[11pt]{amsart} 
\usepackage{xspace,amssymb,amsmath,amscd,amsthm,bm, braket,comment,enumerate,,epsfig,etoolbox, hyperref,  float, mathrsfs, mathtools, mdframed, multicol,  tikz,  tikz, ,ytableau, rotating}
\usetikzlibrary{calc, tikzmark}
\hypersetup{colorlinks = true, citecolor = blue, linkcolor = blue}
\usepackage[capitalise, nameinlink]{cleveref}
\ytableausetup{nobaseline}
\newlength{\defbaselineskip}
\theoremstyle{plain}
\newtheorem{theorem}{Theorem}
\newtheorem{lemma}[theorem]{Lemma}
\newtheorem{prop}[theorem]{Proposition}
\newtheorem{corollary}[theorem]{Corollary}

\theoremstyle{definition}

\newtheorem{example}[theorem]{Example}
\newtheorem{remark}[theorem]{Remark}

\newcommand{\boks}[1]{\ytableausetup{boxsize = #1 cm}}

\newcommand{\y}[1]{\ydiagram{#1}}
\newcommand{\yt}[1]{\ytableaushort{#1}}

\newcommand{\blk}{\underline{\hspace*{1em}}}

\newcommand{\x}{\mathbf{x}}

\newcommand{\Q}{\mathbb{Q}}

\renewcommand{\S}{\mathfrak{S}}

\newcommand{\Par}{\mathrm{Par}}
\newcommand{\Com}{\mathrm{Com}}
\newcommand{\PCom}{\mathrm{PCom}}
\newcommand{\sqf}{\mathrm{sqf}}
\newcommand{\dya}{\mathrm{dyad}}
\newcommand{\Typ}{\mathrm{Typ}}
\newcommand{\mcM}{\mathcal{M}}
\newcommand{\acts}{\circlearrowright}
\newcommand{\rot}[1]{\text{\begin{rotate}{90}$#1$\end{rotate}}}
\DeclareMathOperator{\pt}{\mathcal{PT}}
\renewcommand{\ast}{^*}
\DeclareMathOperator{\sbt}{SBT}
\DeclareMathOperator{\wbt}{WBT}
\DeclareMathOperator{\rbt}{RBT}
\DeclareMathOperator{\psbt}{PSBT}
\DeclareMathOperator{\pwbt}{PWBT}
\DeclareMathOperator{\prbt}{PRBT}
\newcommand{\mcI}{\sigma}
\newcommand{\T}{\mathcal{T}}
\newcommand{\U}{\mathcal{U}}
\renewcommand{\L}{\mathcal{L}}
\DeclareMathOperator{\cycC}{cycC}
\DeclareMathOperator{\cycP}{cycP}
\DeclareMathOperator{\dg}{dg}
\DeclareMathOperator{\sgn}{sgn}
\DeclareMathOperator{\psgn}{psgn}
\DeclareMathOperator{\sort}{sort}
\DeclareMathOperator{\psort}{psort}
\DeclareMathOperator{\Sym}{Sym}
\DeclareMathOperator{\sym}{Sym}
\DeclareMathOperator{\PSym}{PSym}
\DeclareMathOperator{\psym}{PSym}
\DeclareMathOperator{\wt}{wt}
\DeclareMathOperator{\CS}{CS}

\let\bbmatrix\bordermatrix
\patchcmd{\bbmatrix}{8.75}{8.75}{}{}
\patchcmd{\bbmatrix}{\left(}{\left[}{}{}
\patchcmd{\bbmatrix}{\right)}{\right]}{}{}

\newlength{\cellsize}
\newcommand\tableau[1]{
\vcenter{
\let\\=\cr
\baselineskip=-16000pt
\lineskiplimit=16000pt
\lineskip=0pt
\halign{&\tableaucell{##}\cr#1\crcr}}}

\newcommand{\tableaucell}[1]{{%
\def \arg{#1}\def \void{}%
\ifx \void \arg
\vbox to \cellsize{\vfil \hrule width \cellsize height 0pt}%
\else
\unitlength=\cellsize
\begin{picture}(1,1)
\put(0,0){\makebox(1,1){$#1$}}
\put(0,0){\line(1,0){1}}
\put(0,1){\line(1,0){1}}
\put(0,0){\line(0,1){1}}
\put(1,0){\line(0,1){1}}
\end{picture}%
\fi}}

\ytableausetup{aligntableaux = center, nobaseline}
\begin{document}

\title[Transition Matrices between $H$, $E$, $E^+$, $P$]{Transition Matrices between Plethystic Bases of Polysymmetric Functions via Bijective Methods}

\author{Aditya Khanna}
\address{Dept. of Mathematics, Virginia Tech, Blacksburg, VA 24061-0123}
\email{adityakhanna@vt.edu}
 
\begin{abstract}
Many identities involving symmetric functions can be proved through bijective manipulations of tableaux. In this paper, we prove identities and expansions involving polysymmetric functions through bijections and sign-reversing involutions.
In their paper titled ``\textit{Polysymmetric functions and motivic measures of configuration spaces}", Asvin G and Andrew O'Desky introduced the algebra of polysymmetric functions (PSym) which can be defined as the tensor product of copies of the symmetric functions algebra (Sym) where the $i$th tensor factor is scaled by $i$. On one hand, we can obtain bases of this algebra by taking tensor products of the bases of Sym. On the other hand, the Asvin G and Andrew O'Desky paper introduces non-pure tensor bases families $H$, $E$, $E^+$, and $P$ that we call \textit{plethystic bases}. In this paper, we present combinatorial interpretations of the entries of the transition matrices between all twelve pairs of distinct plethystic bases. We also provide new interpretations for six OEIS sequences that turn up in this context.
\end{abstract}

\maketitle

\noindent\textbf{Keywords:} symmetric functions; polysymmetric functions;
 transition matrices; brick tabloids; sign-reversing involutions.
 
\noindent\textbf{2020 MSC Subject Classifications:}  
 05A19; 05A17; 15A09; 05E05. 

\section{Introduction}
The algebra $\Sym$ of symmetric functions is a celebrated and well-studied object in combinatorics. The algebraic flavor of symmetric functions is often complemented with the combinatorics of partitions.
In their paper \cite{polysymm}, Asvin G and Andrew O'Desky introduce a generalization of symmetric functions called \textit{polysymmetric functions} and a generalization of partitions called \textit{types}. The algebra of polysymmetric functions denoted $\PSym$ has several analogous properties to $Sym$ which are explored in \cite{polysymm}. We first recall the notions related to partitions and symmetric functions and then review types and polysymmetric functions by analogy. 
\subsection{Review of Symmetric Functions} For this paper, we assume basic familiarity with symmetric functions and their combinatorics; some excellent introductions to these topics are \cite{macd, loehr-comb, sagan, stanley}. 
We briefly recall some terminology and introduce notation that will be relevant throughout the paper. A \textit{composition of $n$} is a tuple $\alpha = (\alpha_1, \alpha_2,\ldots, \alpha_k)$ of positive integers such that $\sum_{i=1}^k \alpha_i = n$. We call $n$ the \textit{size of $\alpha$} and denote it by $|\alpha|$. The number of entries in $\alpha$ is called the \textit{length of $\alpha$} and is denoted by $\ell(\alpha)$. The composition $\alpha = (3,4,3)$ has size $|\alpha| = 10$ and length $\ell(\alpha) = 3$.

We denote the set of all compositions of size $n$ by $\Com(n)$ and define $\Com(0) = \{\varnothing\}$ where $\varnothing$ denotes the empty composition. Furthermore, $\Com(n) = \varnothing$ whenever $n$ is not a non-negative integer. A \textit{composition} is an element of $\Com:=\bigcup_{n\geq 0} \Com(n)$.
We visualize a composition by constructing left-justified rows consisting of $\alpha_i$ boxes in the $i$th row from the top and denote this visualization by $\dg(\alpha)$. If $\alpha = (3,4,3,1)$, then \boks{0.2}$\dg(\alpha) = \y{3,4,3,1}$.
A \textit{partition of $n$} is a composition of $n$ where the entries occur in weakly decreasing order, that is, $\alpha_1 \geq \alpha_2 \geq \ldots \geq \alpha_k$. We denote the set of partitions of $n$ by $\Par(n)$ and the set of all partitions by $\Par$. Define $\sort: \Com(n) \to \Par(n)$ by arranging the entries of a composition in weakly decreasing order. For instance, $\sort((3,4,3,1,2)) = (4,3,3,2,1)$.

Consider a family of indeterminates $x_1, x_2, \ldots$ and let $f(x_1, x_2, \ldots) \in\Q[[x_1, x_2, \ldots]]$ be a formal power series with rational coefficients. We denote by $\S_n$ the symmetric group on $n$ letters and by $\S_\infty$ the group of permutations that permute finitely many natural numbers. For $\sigma\in \S_\infty$, define the action $\sigma\cdot f$ by sending each $x_i$ to $x_{\sigma(i)}$ for all $i \in \{1,2, \ldots\}$. We say that $f$ is a \textit{symmetric function of degree $n$} if $f$ is homogeneous of degree $n$ and $\sigma\cdot f = f$ for all $\sigma \in \S_\infty$. The function \[f(x_1, x_2, \ldots) = x_1^2x_2x_3 + x_1x_2^2x_3 + x_1x_2x_3^2 + x_1^2 x_2x_4 + \ldots \] is a symmetric function of degree 4 where the rest of the terms are obtained by permuting the indices in all possible ways.
The set of symmetric functions of degree $n$ is denoted by $\sym(n)$ and we define $\sym := \bigoplus_{n\geq 0} \sym(n)$. The set $\sym$ can be assigned a $\mathbb{Q}$-algebra structure as the addition, multiplication and $\Q$-scaling of symmetric functions is still symmetric. Define $h_n$ to be the sum over all monomials in $x_1, x_2, \ldots$ of degree $n$; $e_n$ to be the sum of monomials in $x_1, x_2, \ldots$ of degree $n$ such that each indeterminate $x_i$ has exponent at most 1; and $p_n$ to be the sum $x_1^n + x_2^n + \ldots$. For instance, we have
 \begin{center}
 $h_3 = x_1^3 + x_1^2x_2 + x_1x_2x_3 + x_2^3 + x_2^2x_1 + \ldots$\\
 $e_3 =  x_1x_2x_3 +  x_2x_3x_4 +  x_1x_2x_4 +  x_1x_3x_4 + \ldots$\\
 $p_3 = x_1^3 + x_2^3 + x_3^3 + x_4^3 + \ldots$
 \end{center}
 
 It is well-known (see \cite[Sec I.2]{macd} and \cite[Thm. 9.75, 9.78]{loehr-comb}) that each of $\{h_n\}_{n\geq 1}$, $\{e_n\}_{n\geq 1}$ and $\{h_n\}_{n\geq 1}$ is an algebraically independent system of generators for the $\mathbb{Q}$-algebra $\sym$. For $\lambda\in \Par$, define $f_\lambda = f_{\lambda_1}f_{\lambda_2}\ldots f_{\ell(\lambda)}$, and it can be shown \cite[Thm. 9.66, 9.71, 9.79]{loehr-comb} that $\{h_\lambda\}_{\lambda\in \Par(n)}$, $\{e_\lambda\}_{\lambda\in \Par(n)}$ and $\{p_\lambda\}_{\lambda\in \Par(n)}$ are all linear bases of $\sym(n)$. Two other bases of note are the \textit{monomial basis}, $\{m_\lambda\}_{\lambda\in \Par(n)}$, and the Schur basis, $\{s_\lambda\}_{\lambda\in \Par(n)}$, whose definitions can be found in \cite{macd}. Let $f$ and $g$ be two bases of $\Sym$. The coefficients of $g_\lambda$ in the $g$-expansion of $f_\mu$ can be recorded in column $\mu$ and row $\mu$ of the \textit{transition matrix from $f$ to $g$} denoted $\mcM(f,g)$. The entries of $\mcM(f,g)$ for $f,g\in \{h,e,p\}$ can be computed using some beautiful combinatorics \cite{eg-rem} involving objects called \textit{brick tabloids}, which we review in Section \ref{ssec:bricktabs}. Transition matrices for other pairs of bases in $\Sym$ have been studied in \cite{BnRemmel}. Transition matrices arising from other generalizations of $\Sym$ such as $\text{NSym}$ (algebra of non-commuative symmetric functions) or $\text{QSym}$ (algebra of quasisymmetric functions) have been studied in \cite{HuangHecke, SchurLift, noncomm} combinatorially.
 
In this paper, we extend the concept of a brick tabloid and provide combinatorial interpretations for the transition matrix entries between the bases of {polysymmetric functions}.
\subsection{Polycompositions and types} A \textit{polycomposition of $n\geq 0$} is an ordered list $\delta = (\alpha^{(1)}, \alpha^{(2)}, \ldots)$ of (possibly empty) compositions such that $|\delta|:=\sum_{i=1}^\infty i|\alpha^{(i)}| = n$. We denote $\delta$ using the formal expression (not to be confused with exponentiation) \[\delta = \left(\alpha^{(1)}_1, \alpha^{(1)}_2, \ldots\right)^1 \left(\alpha^{(2)}_1, \alpha^{(2)}_2, \ldots\right)^2 \left(\alpha^{(3)}_1, \alpha^{(3)}_2, \ldots\right)^3\ldots.\] 
Here, each $\alpha^{(i)}_j$ for $ i,j\geq 1$ is called a \textit{degree} while the superscripts are called \textit{multiplicities}. For instance, $\delta = (3,1,2,2)^1 (1,2,1)^2 (1,5)^4$ is a polycomposition of  $1\cdot 8 + 2\cdot 4 + 4\cdot 6 = 40$. For a positive integer $r$, denote by $\delta^r$ the polycomposition
\[\left(\alpha^{(1)}_1, \alpha^{(1)}_2, \ldots\right)^{r} \left(\alpha^{(2)}_1, \alpha^{(2)}_2, \ldots\right)^{2r} \left(\alpha^{(3)}_1, \alpha^{(3)}_2, \ldots\right)^{3r}\ldots.\] 
We use the notation $\delta|^i$ to denote the composition formed by the degrees with multiplicity $i$, that is, $\delta|^i = \alpha^{(i)}$. The \textit{length $\ell(\delta)$} of a polycomposition $\delta$ is $\ell(\delta):=\sum\limits_{i\geq 1} \ell(\delta|^i)$. For our previous example $\ell(\delta) = 4 + 3 + 2 = 9$.
Denote the set of polycompositions of $n$ by $\PCom(n)$ and the set of all polycompositions by $\PCom$.
A \textit{type} \footnote{We prefer the verbiage \textit{polypartitions} in analogy with the ``poly-" prefix for symmetric functions but for consistency with the literature we will stick to using the term ``types".} $\tau$ of $n\geq 0$ is a polycomposition such that $\tau|^i\in \Par$ for all $i\geq 1$. Denote the set of types of $n$ by $\Typ(n)$ and the set of all types by $\Typ$. The map $\psort: \PCom(n) \to \Typ(n)$ is defined by $\psort(\delta) =  \tau$ where $\tau|^i = \sort(\delta|^i)$ for all $i\geq 1$. For our previous example, $\psort(\delta) = (3,2,2,1)^1 (2,1,1)^2 (5,1)^4$.
 
A \textit{block} is a polycomposition with exactly one degree and one multiplicity. A block with degree $d$ and multiplicity $r$ is denoted by $d^r$, where we omit the parentheses. Through this perspective, a polycomposition is a sequence of blocks with weakly increasing multiplicities, and we write $d^r\in \delta$ when a block $d^r$ appears in the polycomposition $\delta$. Similarly, a type can be thought of as a multiset of blocks. Continuing with our previous example, we can write $\delta$ in block form as $\delta = 3^11^12^12^11^22^21^21^45^4$. Also, this means that $\ell(\delta)$ is the number of blocks in $\delta$ counted with repetitions. Whenever a polycomposition or a type is written as a sequence of blocks $d_1^{r_1}\ldots d_k^{r_k}$, we will assume $r_1\leq r_2\leq \ldots \leq r_k$.

The \textit{tensor diagram}\footnote{Note that in \cite{KLpsym}, each $d^r\in \tau$ corresponds to a row of length $r$ in the $d$th tensor factor of the tensor diagram of a type $\tau$. In that paper, the notation for a tensor diagram is $\dg(\tau)$ and as our interpretation of a block is different, we use a different notation.} of a polycomposition $\delta$, denoted $\dg^\otimes(\delta)$, is a formal tensor product of $\dg(\delta|^{i})$ for $i\geq 1$. The tensor diagram for our running example $\delta = (3,1,2,2)^1 (1,2,1)^2 (1,5)^4$ is
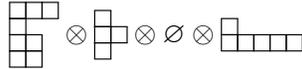
\[
\y{3,1,2,2} \otimes \y{1,2,1} \otimes \varnothing \otimes \y{1,5}
\]
\subsection{Polysymmetric functions}
Consider the set of doubly-indexed indeterminates $\{x_{i,j}\}_{i,j\geq 1}$ where each $x_{i,j}$ has degree $i$. For a monomial $f$, define $\exp_{i,j}(f)$ to be the exponent of $x_{i,j}$ in $f$. We say \textit{$f$ is a monomial of degree $d$} if $\deg(f):=\sum_{i,j\geq 1} i\exp_{i,j}(f) = d$. For instance, the monomial $x_{2,1}^4 x_{2,4}^3 x_{5,1}^1$ has degree $2\cdot 4 + 2\cdot 3 + 5\cdot 1 = 19$.

A \textit{polysymmetric function} $F$ is a formal power series of bounded degree in $x_{i*} = \{x_{i,j}\}_{i,j\geq 1}$ which is a symmetric function in each variable set $\{x_{ij}\}_{j\geq 1}$. The graded $\Q$-algebra of polysymmetric functions is denoted by $\psym$ and the sub-algebra spanned by polynomials of degree $n$ (the $n$th grading) by $\psym(n)$. The linear bases of the $\Q$-vector space $\psym(n)$ are indexed by types of $n$ \cite[Thm. 3.1]{polysymm}. One way to construct a linear basis of $\psym$ is by taking a tensor product of copies of a symmetric function basis $\{f_\lambda:\lambda\in \Par\}$ which produces a \textit{pure-tensor basis} of $\psym$. For instance, $h_{(3,1)}\otimes h_{(2,1)} \otimes 1\otimes h_{(2,2)} = h_{(3,1)}(x_{1*}) h_{(2,1)}(x_{2*}) h_{(2,2)}(x_{4*})$ is a basis element in the $h^\otimes$ basis of $\psym(26)$.
Let $\sym^{(i)}$ be an isomorphic copy of $\sym$ where each variable has degree scaled by $i$. We can express $\psym$ as $\bigotimes_{i\geq 1} \sym^{(i)}$.
In \cite{polysymm}, the authors define four families of non-pure-tensor bases $\{H_\tau: \tau \in \Typ\}$, $\{E_\tau: \tau \in \Typ\}$, $\{E^+_\tau: \tau \in \Typ\}$ and $\{P_\tau: \tau \in \Typ\}$. These are generalizations of the symmetric base and in our exposition, we call $H$, $E$, $E^+$ and $P$ \textit{plethystic bases}. We now review their definitions.

Here and later, we use the term ``monomial" to mean a monomial in the indeterminates $\{x_{i,j}: i,j\geq 1\}$, and we may omit the comma between the indices for convenience. We call a monomial \textit{square-free} if $\exp_{i,j}(f)\leq 1$ for all $i,j\geq 1$. Define $\sgn(f) = \prod_{i,j\geq 1}(-1)^{\exp_{i,j}(f)}$  where the total exponent of $-1$ is the number of indeterminates in the monomial $f$. Define $H_d = \sum f$ where the sum is over monomials $f$ with degree $d$. Define $E^+_d = \sum f$ and $E_d = \sum \sgn(f) f$, where both sums are over square-free monomials of degree $d$. Define $P_d = \sum\limits_{k\mid d}\sum\limits_{j\geq 1} kx_{kj}^{d/k}$ where the outer sum is indexed by positive divisors of $d$. 

Let $x_{**}$ denote the variable set $\{x_{i,j}\}_{i,j\geq 1}$. For a block $d^r$ and $F \in \{H, E, E^+, P\}$, define $F_{d^r} = F_d(x^{r}_{**})$, that is, $F_{d^r}$ is obtained by replacing all indeterminates $x_{ij}$ in $F_d$ by $x_{i,j}^r$. This action is called the \textit{Addams operation} and explains the nomenclature of plethystic bases as $f(x_*^r) = f[p_r](x_{*})$ for $f\in \sym$. Let $\delta = d_1^{r_1} d_2^{r_2} \ldots d_k^{r_k}$ be a polycomposition expressed as a sequence of blocks with weakly increasing multiplicities. Define $F_{\delta} = \prod_{i = 1}^k F_{d_i^{r_i}}$ which is a product over all blocks $d^r$ that appear in $\delta$ counted with multiplicity.
\begin{example}
To illustrate the above definitions, \begin{align*}
H_3 &= x_{31} + x_{21}x_{11} + x_{11}^3
 + x_{11}^2x_{12} + x_{11}x_{12}x_{13}+\ldots\\
E^+_3 &= x_{31} + x_{21}x_{11} + x_{11}x_{12}x_{13} + \ldots\\
E_3 &= -x_{31} + x_{21}x_{11} - x_{11}x_{12}x_{13} + \ldots\\
P_{6} &= 6x_{61} + 3x_{31}^2 + 2x_{21}^3 + x_{11}^6  + \ldots.
\end{align*}
To obtain the rest of the terms, we permute the second index of each of the leading monomials above.
To go from $H_3$ to $H_7$, we raise each indeterminate to the power 7 to obtain
\[
H_{3^7} = x_{31}^7 + x_{21}^7x_{11}^7 + x_{11}^{21}
 + x_{11}^{14}x_{12}^7 + x_{11}^7x_{12}^7x_{13}^7+ \ldots.
\]
An example of $H_\delta$ for a general $\delta$ is $H_{(3,2)^1(2,1)^3} = H_3H_2H_{2^3}H_{1^3}$.
 \end{example}
The combinatorial interpretations for the expansions of the plethystic bases $H$, $E$, $E^+$, and $P$ in terms of the pure-tensor bases formed by $m$, $s$ and $p$ can be found in \cite{KLpsym}. In this paper, we provide combinatorial interpretations of the expansion coefficients between all 12 pairs of distinct plethystic bases. For $F,G\in \{H,P, E, E^+\}$, it is sufficient to find the $G$-expansion of $F_d$, as that can be extended to find the $G$-expansions of $F_{d^r}$ and $F_\tau$ for type $\tau$, as follows. For $d\geq 0$, the $G$-expansion of $F_d$ indexed by the spanning set of polycompositions of $d$ is combinatorially more tractable than the $G$-expansion indexed by types of $d$. Once we have found $F_d = \sum_{\delta \in \PCom(d)} c_{\delta} G_{\delta}$ for certain coefficients $c_\delta \in \Q$, then we can collect the terms indexed by $\delta$ with $\psort(\delta) = \tau$ and obtain the $G$-expansion of $F_d$ indexed by types. Furthermore, $F_{d^r} = \sum_{\delta \in \PCom(d)} c_{\delta} G_{\delta^r}$ and for $\sigma\in \Typ(n)$, we find $F_\sigma = \prod_{d^r\in \sigma} F_{d^r} = \sum_{\tau\in \Typ(n)} c_{\tau}^\sigma G_\tau$. The coefficients $c_{\tau}^\sigma$ can be interpreted as a sum over signed, weighted tilings of the boxes of $\dg^\otimes(\sigma)$ according to rules dictated by the specific choices of $F$ and $G$. We discuss these combinatorial interpretations in detail in Section \ref{sec:polybrick}.
\begin{remark}
Polysymmetric functions were first introduced in \cite{polysymm} in a geometric context while studying the cohomology of the variety of geometrically irreducible hypersurfaces of degree $d$ in projective $n$-space. The representation theory of the uniform block permutation (UBP) algebra was independently studied in \cite{alg-ubp}, where polysymmetric functions appear when studying the Frobenius characteristic map for the UBP algebra. Just as partitions of $n$ index the representations of $\S_n$, the representations of the UBP algebra, $\mathcal{U}_n$, are indexed by types of size $n$ which we call $V_\tau$ for some $\tau \in \Typ$. The symmetric function associated to the character of the restriction of $V_\tau$ to $\Sym(k)$ is another presentation of the pure-tensor basis arising from the Schur functions. The reader may refer to \cite{alg-ubp} for a discussion on the representation theory of UBP algebras and refer to Remark 5 in \cite{KLpsym} for a brief description of the connection to $\psym$.
\end{remark}

\subsection{Structure of this Paper}
In Section \ref{sec:comb-models}, we introduce \textit{bar tableaux}, which serve as combinatorial models for the monomial expansions of certain polysymmetric functions.
In subsequent sections, for each pair $F,G\in \{H,E,E^+, P\}$, we present a $G$-expansion (indexed by polycompositions) of $F_d$ for $d\geq 0$. We prove these expansions using bijections and sign-reversing involutions on bar tableaux. 
In Section \ref{sec:HEP}, we find the expansions between $H$, $E$, and $P$ where the only multiplicity that appears in the indexing polycompositions is 1. We also prove some recursion results involving these three bases. In Section \ref{sec:E+in}, we present $H$, $E$, and $P$ expansions of $E^+$ where the expansions are indexed by polycompositions. The polycompositions that index the terms of the $H$, $E$ and $P$ expansions have multiplicities at most 2. We also prove a formula for $E^+_d$ in terms of $H$ and $E$, which can be used to find the $H$, $E$, and $P$ expansions of $E^+$ by appealing to the expansions found in Section \ref{sec:HEP}. In Section \ref{sec:inE+}, we present $E^+$-expansions of $H$, $E$, and $P$, and the polycompositions that index the terms in the expansion have multiplicities which are powers of 2. Assuming familiarity with the combinatorial models introduced in Section \ref{sec:comb-models}, it is possible to understand the proofs of $G$-expansions of $F_d$ for $d\geq 0$ independent of each other, except in some rarer cases where a map builds upon a map defined in a previous proof. In Section \ref{sec:polybrick}, we use the $G$-expansion of $F_d$ to find the $G$-expansion of $F_\sigma$ for $\sigma\in \Typ$ in terms of new objects called \textit{polybrick tabloids}. In Section \ref{sec:oeis}, we discuss the relationship of our results to six OEIS \cite{oeis} entries. In particular, the number of types which index the non-zero terms in the $G$-expansion of $F_d$ can be counted using these OEIS entries:
\begin{itemize}
\item $E$-expansion of $E^+$: \texttt{A024786}
\item $H$-expansion of $E^+$: \texttt{A025065}
\item $P$-expansion of $E^+$: \texttt{A002513}
\item $E^+$-expansion of $H$: \texttt{A018819} 
\item $E^+$-expansion of $E$: \texttt{A092119}
\item $E^+$-expansion of $P$: \texttt{A305841}
\end{itemize}
\subsection{Summary of results}
We present a brief summary of our main definitions and propositions for easy reference. Recall that $\PCom(n)$ is the set of polycompositions of size $n$. We define some notable subsets of $\PCom(n)$:
\begin{table}[H]
\renewcommand{\arraystretch}{1.2}
\begin{tabular}{|c|p{10cm}|p{3cm}|}
\hline
\textbf{Notation} & \textbf{Description} (set of $\delta\in \PCom(n)$ such that...) & \textbf{Example}\\
\hline
$\PCom_\sqf(n)$ & { $\delta$ has the unique multiplicity 1}.& $(3,1,2,1)^1$\\
\hline
 $\PCom_P(n)$ & $\delta$ is of the form $\alpha^1\beta^2$ with either $\alpha, \beta$ are possibly empty. & $(3,1)^1(1,2,1)^2$\\
\hline
$\PCom_E(n)$ & $\delta$ is of the form $\alpha^1(b)^2$ with $b\in \mathbb{Z}_{\geq 0}$. & $(3,1)^1(4)^2$\\
\hline
$\PCom_H(n)$ & $\delta$ is of the form $(a)^1\beta^2$ with $a\in \mathbb{Z}_{\geq 0}$. & $(2)^1(1,2,2)^2$\\
\hline
$\PCom_\dya(n)$ & all multiplicities in $\delta$ are powers of 2. & $(1,1)^1(1)^2(3,1)^8$\\
\hline
$\PCom'_\dya(n)$ & all multiplicities in $\delta$ are powers of 2 and $\ell(\delta|^i) \leq 1$ for all $i\geq 1$. & $(1)^1(3)^2(2)^4(1)^8$\\
\hline
$\PCom\ast_\dya(n)$ & $\delta$ has a unique multiplicity and that must be a power of 2. & $(1,2,2,1)^8$\\
\hline
\end{tabular}
\end{table}

Let $\delta = d_1^{r_1}d_2^{r_2}\ldots d_k^{r_k}$ with $r_1\leq r_2\leq \ldots \leq r_k$.
Recall that $\ell(\delta) = k$ is the number of blocks in $\delta$. Define $L(\delta) = d_kr_k$ to be the size of the last block.

Let $\lambda$ be a partition. For $i\geq 1$, define $m_i(\lambda)$ to be the number of times part $i$ appears in $\lambda$. Then define $z_\lambda = \prod_i i^{m_i(\lambda)} m_i(\lambda)!$ and define the polysymmetric analog for a type $\tau$ by $z^\otimes_\tau = \prod_{d\geq 1} z_{\tau|^d}$.

For a positive integer $d$, we have the following expansions:
\begin{itemize}
 \item $H_d = \sum\limits_{\delta \in \PCom_{\sqf}(d)} (-1)^{\ell(\delta)}E_\delta$ (Proposition \ref{prop:H-E,E-H})
 \item $E_d = \sum\limits_{\delta \in \PCom_{\sqf}(d)} (-1)^{\ell(\delta)}H_\delta$ (Proposition \ref{prop:H-E,E-H})
 \item $P_d = \sum\limits_{\delta \in \PCom_{\sqf}(d)} (-1)^{\ell(\delta)-1} L(\delta) H_\delta$ (Proposition \ref{prop:P-H,P-E})
\item $P_d = \sum\limits_{\delta  \in \PCom_{\sqf}(d)} (-1)^{\ell(\delta)} L(\delta) E_\delta$ (Proposition \ref{prop:P-H,P-E})
\item $H_d = \sum\limits_{\delta \in \PCom_{\sqf}(d)} \dfrac{P_\delta}{Z_\delta}$ (Proposition \ref{prop:H-P})
\item $E_d = \sum\limits_{\delta \in \PCom_{\sqf}(d)} (-1)^{\ell(\delta)}\dfrac{P_\delta}{Z_\delta}$ (Proposition \ref{prop:E-P})
\item  $E^+_d = \sum\limits_{\delta = \alpha^1(b)^2 \in \PCom_E(d)} (-1)^{\ell(\alpha)} E_{\delta}$ (Proposition \ref{prop:E+-HEP})
\item $   E^+_d = \sum\limits_{\delta = (a)^1\beta^2 \in \PCom_H(d)} (-1)^{\ell(\beta)} H_{\delta}$ (Proposition \ref{prop:E+-HEP})
\item $E^+_d  = \sum\limits_{\delta = \alpha^1 \beta^2 \in \PCom_P(d)} (-1)^{\ell(\beta)} \dfrac{1}{Z_\alpha Z_\beta}P_{\delta}$ (Proposition \ref{prop:E+-HEP})
\item  $H_d = \sum\limits_{\delta\in \PCom_\dya'(d)} E^+_\delta$ (Proposition \ref{prop:HinE+})
\item $E_d = \sum\limits_{\delta\in \PCom_\dya(d)} (-1)^{\ell(\delta)}E_\delta^+$ (Proposition \ref{prop:EinE+})
\item $P_d=\sum\limits_{\delta\in \PCom_\dya\ast(d)} (-1)^{\ell(\delta)-1} L(\delta) E_\delta^+$ (Proposition \ref{prop:PinE+})
\end{itemize}

\begin{remark}
A preprint by David Martinez \cite{dm-psym} uploaded at the same time as this preprint independently finds the expansions between the plethystic bases. Their proofs use generating function methods while ours are entirely combinatorial. 
\end{remark}

\section{Combinatorial models for monomial expansions of the plethystic bases}\label{sec:comb-models}

In this section, we provide combinatorial descriptions for the monomials that appear in the bases $H$, $E^+$, $E$ and $P$, and introduce notions that we will use throughout the paper. We will manipulate these combinatorial descriptions to prove the expansions between the bases in later sections.
\boks{0.42}
\subsection{Bar tableaux} Hereafter, a \textit{bar} refers to a row within a partition diagram where all cells in the row are labeled with the same natural number. The number of cells in the bar is called its \textit{length} or \textit{size}. We call two bars \textit{identical} if they have the same length (number of cells) and the same label.\boks{0.35} For instance, $\yt{44}$ and $\yt{44}$ are identical bars in the example below.
 For a partition $\lambda$, a \textit{weak bar tableau (WBT) of shape $\lambda$} is a labeling of the cells of $\dg(\lambda)$ with natural numbers such that each row is a bar and the labels increase weakly within parts of the same size. We denote the set of weak bar tableaux of shape $\lambda$ by $\wbt(\lambda)$ and define the set of weak bar tableaux of size $d$ by $\wbt[d] = \bigcup_{\lambda \in \Par(d)} \wbt(\lambda)$. \boks{0.42}The following tableau is an element of $\wbt((4,4,4,3,3,2,2))\subset \wbt[22]$:
 \boks{0.42}
 \[
 T = \yt{2222,2222,2222,111,222,44,44}
 \]
 
With each bar of length $i$ and label $j$, we associate the variable $x_{ij}$. For any $T\in \wbt[d]$, define $|T|$ to be the number of cells in $T$ and define $\ell(T)$ to be the number of rows in $T$. With each $T\in \wbt[d]$, we associate the monomial $\x_T = \prod_{i,j\geq 1} x_{ij}^{b_{ij}(T)}$ where $b_{ij}(T)$ counts the number of bars of length $i$ and label $j$ in $T$. The monomial for $T$ shown above is $\x_T = x_{42}^3 x_{31}x_{32}x_{24}^2$.
The subset $\sbt(\lambda)\subset \wbt(\lambda)$ is the set of \textit{strict bar tableaux (SBT) of shape $\lambda$} such that any $T\in \sbt(\lambda)$ has labels that increase strictly within parts of the same size. Also define the set of \textit{strict bar tableaux of size $d$} by $\sbt[d] = \bigcup_{\lambda \in \Par(d)} \sbt(\lambda)$.  With each $T\in \sbt(\lambda)$, we associate the sign $\sgn(T) = (-1)^{\ell(T)}$. The following is an example of an SBT $T$ of size 14 with the associated monomial $\x_T = x_{42}x_{43}x_{25}x_{27}x_{11}x_{13}$ and $\sgn(T) = (-1)^6 = 1$:
 \[
 T = \yt{2222,3333,55,77,1,3}
 \]
 A \textit{rectangular bar tableau (RBT) of size $d$} is a WBT where all bars are identical. We denote the set of RBTs of size $d$ by $\rbt[d]$. A \textit{marked RBT of size $d$} is an RBT of size $d$ where one of the cells in the top row is marked (which we signify with an asterisk ($\ast$)). The set of marked RBTs of size $d$ is denoted by $\rbt\ast[d]$. For example,
 \[
 \rbt^*[4] = \left\{\yt{{a^*},a,a,a}, \yt{{a\ast}a,aa}, \yt{a{a\ast},aa}, \yt{{a\ast}aaa}, \yt{a{a\ast}aa}, \yt{aa{a\ast}a}, \yt{aaa{a\ast}}: a\geq 1\right\}.
 \]

\begin{lemma}\label{lem:btinterps} 
For $d\geq 0$, 
\begin{enumerate}
\item $H_d = \sum\limits_{T\in \wbt[d]} \x_T$.
\item $E^+_d = \sum\limits_{T\in \sbt[d]} \x_T$.
\item $E_d = \sum\limits_{T\in \sbt[d]} \sgn(T)\x_T$.
\item $P_d = \sum\limits_{T\in \rbt\ast[d]} \x_T$.
\end{enumerate}
\end{lemma}
\begin{proof}
When $d = 0$, the sets indexing all four sums on the right hand side are $\{\varnothing\}$, and we obtain $H_0 = E_0 = E^+_0 = P_0 = \x_{\varnothing} = 1$ which agrees with the definitions of the bases. We now assume $d>0$.
To prove (1), it suffices to show that the map $T\mapsto \x_T$ from $\wbt[d]$ to the set of monomials (in $x_{ij}$) of degree $d$ is a bijection. We start with a WBT $T$. The monomial $\x_T = \prod_{i,j} x_{ij}^{b_{ij}}$ has degree $\sum_{i,j} i b_{ij}$ which is exactly the number of cells in $T$, that is, $d$. This shows that each WBT $T$ of size $d$ has a unique degree $d$ monomial associated with it. To prove the bijection, we start with a monomial $\prod\limits_{i,j\geq 1}x_{ij}^{b_{ij}}$ of degree $d$ and construct the associated WBT that contains $b_{ij}$ bars of length $i$ and label $j$ such that the sizes of bars weakly decrease as we go down and the labels within bars of the same size  weakly increase. The proof of (2) and (3) proceeds the same as the proof of (1), but here the strict increase in labels corresponds to each $x_{ij}$ appearing at most once, thus making the monomial square-free. For (4), we notice that an RBT of size $d$, in which all bars are identical, must have $d/k$ bars of length $k$ for some divisor $k$ of $d$. If the bars have label $j$, then this RBT contributes the monomial $x_{kj}^{d/k}$. The coefficient $k$ of $x_{kj}^{d/k}$ in $P_d$ corresponds to the $k$ different marked RBTs generated by marking the $k$ cells in the top row.
\end{proof}

\subsection{Poly bar tableaux} So far, we have only provided combinatorial models for $F_\delta$ where $\delta = (d)^1$ is a composition with only one block with multiplicity 1. We now discuss the combinatorial model for $F_\delta$ where $\delta\in \PCom$.

Let $\delta$ be a polycomposition of size $n$. A \textit{polyWBT of shape $\delta$} is an ordered array of WBTs $\T = (T_{ij})^{i\geq 1}_{ 1\leq j \leq \ell(\delta|^i)}$ where $T_{ij}\in \wbt[(\delta|^i)_j]$. To visualize this, we place WBTs in horizontally stacked layers with indices increasing downwards, where the indices are values of $r$ for which $\delta|^r$ is non-empty. Then in layer $r$ we place an ordered tuple of WBTs $(T_{r1}, T_{r2}, \ldots, T_{r\ell(\delta|^r)})$. We denote the set of polyWBTs of shape $\delta$ by $\pwbt(\delta)$. The following is an element of $\pwbt((3,1,2,2)^1 (1,2,1)^2 (1,5)^4)$:
\ytableausetup{aligntableaux = top}
\begin{align*}
1\quad & \yt{11,3} \quad \yt{2} \quad \yt{3,3} \quad \yt{22}\\
2\quad & \yt{3} \quad \yt{11} \quad \yt{3}\\
4\quad & \yt{1}\quad \yt{22,33,1}
\end{align*}
If the only layer is 1, then we omit the indexing for the layers.

\ytableausetup{aligntableaux = center}
For $\T\in \pwbt(\delta)$, we define the monomial $\x_{\T} = \prod\limits_{i\geq 1}\prod\limits_{1\leq j \leq \ell(\delta|^i)} \x_{T_{i,j}}^i$. The above example has the associated monomial \[(x_{21}x_{13})(x_{12})(x_{13}^2)(x_{22})(x_{13})^2(x_{21})^2(x_{13})^2(x_{11})^4(x_{22}x_{23}x_{11})^4 = x_{21}^3x_{22}^5x_{23}^4 x_{11}^8 x_{12}x_{13}^7.\] 
Define the set $\psbt(\delta)$ of \textit{polySBTs of shape $\delta$} and the set $\prbt(\delta)$ of \textit{polyRBTs of shape $\delta$} similarly. The number of tableaux that appear in $\T$ is denoted by $\ell(\T)$ which is equal to $\ell(\delta)$. Let $B(\T) = \sum_{i,j\geq 1} \ell(T_{i,j})$ be the total number of bars (with repetitions) that appear in $\T$. For $\T\in \psbt(\delta)$, define $\psgn(\T) := \prod\limits_{\substack{i\geq 1\\ 1\leq j \leq \ell(\delta|^i)}}\sgn(T_{ij})=(-1)^{B(\T)}$. A third measure of size is $|\T|$, the total number of cells in $\T$, which is not the same as $|\delta|$.
\begin{lemma} \label{lem:pbtinterps}
For a polycomposition $\delta$,
\begin{enumerate}
\item $H_\delta = \sum\limits_{\T\in \pwbt(\delta)} \x_{\T}$.
\item $E^+_\delta = \sum\limits_{\T\in \psbt(\delta)} \x_{\T}$.
\item $E_\delta = \sum\limits_{\T\in \psbt(\delta)} \psgn(\T)\x_{\T}$.
\end{enumerate}
\end{lemma}
\begin{proof}
We prove (1), and the other statements can be proved similarly. We find that $H_{d^r} = \sum_{T\in \wbt[d]} \x_T^r$ as $H_{d^r}(x_{**}) = H_d(x_{**}^r)$. Let $\delta =(d_1^{r_1}, d_2^{r_2}, \ldots, d_k^{r_k})$ be a polycomposition. For each $i$, we choose $T_i \in \wbt[d_i]$ and we construct a PWBT $\T$ by placing each $T_i$ in layer $r_i$ such that each $T_j$ is placed to the right of $T_k$ whenever $j > k$ in the same layer. We obtain $\x_{\T} = \prod_i^k \x_{T_i}^{r_i}$ for this choice. This shows that the monomials in the expansion of $H_\delta$ can be found in the sum $\sum_{\T\in \pwbt(\delta)} \x_{\T}$ . Conversely, if we have a PWBT $\T = (T_{i,j})_{i,j \geq 1}$ where $T_{i,j}$ is the (possibly empty) $j$th PWBT from the left contained in the $i$th layer, then we can write $\x_{\T} = \prod_{i,j\geq 1} \x_{T_{i,j}}^i$. In this case, each $x_{T_{i,j}}^i$ occurs as a monomial in $H_{|T_{i,j}|^i}$, and thus $\x_{\T}$ occurs as a monomial in $H_\delta$.
\end{proof}

We order all the occurrences of bars in a PWBT using a \textit{scanning order} wherein we visit the layers 1, 2, 3 and so on in order; going from left to right through tableaux within each layer, and going from the top bar to the bottom bar within each tableau. The bars listed in the scanning order in the above example are
\[
\yt{11}\to \yt{3}\to \yt{2}\to \yt{3}\to \yt{3}\to \yt{22} \to \yt{3} \to \yt{11} \to \yt{3} \to \yt{1} \to \yt{22} \to \yt{33} \to \yt{1}.
\]
For bar tableaux $T$ and $T'$ in some PWBT $\T$, we say that \textit{$T'$ occurs after $T$} in the scanning order if the topmost bar in $T'$ appears after the lowermost bar in $T$ in the scanning order. In our example, $\yt{11}$ (second WBT in layer 2) appears after the tableaux $\yt{11,3}$ (first WBT in layer 1).

\subsection{Marked poly bar tableaux} A \textit{marked polyWBT (marked PWBT)} $\T\ast$ is constructed from a PWBT $\T$ by marking one cell in the last tableau in the scanning order. Suppose $T_1, T_2, \ldots T_k, T_{k+1}$ are WBTs ordered according to their occurrence in the scanning order for $\T$ and let $T\ast$ be obtained from $ T_{k+1}$ by marking one cell. We write $\T\ast = (T_1, \ldots, T_k, T^*)$.
We denote the set of marked PWBTs of shape $\delta$ by $\pwbt\ast(\delta)$. The following is an element of $\pwbt\ast((2,6,3,6)^4)$:
\[
\T\ast = 4\quad \yt{1,3}\quad \yt{22,33,1,1} \quad \yt{44,55}\quad \yt{1{1\ast}1,22,1}.
\]
We associate the monomial $\x_{\T}$ with $\T\ast$ which is the same monomial as the one for the corresponding unmarked PWBT.
Similarly, we define the notion of a \textit{marked polySBT of shape $\delta$} and denote the corresponding set by $\psbt\ast(\delta)$.  If the marked tableau $T\ast$ occurs in layer $r$, then we define $\wt\ast(\T) = r$. The set  $\prbt\ast(\delta)$ of \textit{marked rectangular bar tableaux of shape $\delta$}, is defined slightly differently; a marked PRBT $\T\in \prbt\ast(\delta)$ is a tuple of marked RBTs, that is, it is a tuple of rectangular bar tableaux where each tableau has one cell marked in its top row and the size of the $j$th marked RBT from left in layer $i$ is $(\delta|^i)_j$. The following is an element of $\prbt\ast((2,6)^1(4,1,4)^2)$:

\begin{align*}
    1 \quad &\yt{{4\ast},4}\quad \yt{4{4\ast},44,44}\\
    2\quad  & \yt{11{1\ast}1} \quad \yt{{2\ast}} \quad \yt{2{2\ast},22}
\end{align*}

 For $\delta\in \PCom(n)$, let $r$ be the largest multiplicity of $\delta$ and $\alpha:= \delta|^r$. Then define $L(\delta) = r\alpha_{\ell(\alpha)}$. The quantity $L(\delta)$ is the size of the last block that appears in $\delta$. If $\delta = (2,1)^2(5,1,4)^5$, then $L(\delta) = 4\cdot 5  = 20$.
 \begin{lemma}\label{lem:markedinterps}
 For a polycomposition $\delta$,
 \begin{enumerate}
 \item $L(\delta)H_\delta = \sum\limits_{\T\in \pwbt\ast(\delta)} \wt\ast(\T)\x_{\T}$.
\item $L(\delta)E^+_\delta = \sum\limits_{\T\in \psbt\ast(\delta)} \wt\ast(\T)\x_{\T}$.
\item $L(\delta)E_\delta = \sum\limits_{\T\in \psbt\ast(\delta)} \psgn(\T)\wt\ast(\T)\x_{\T}$
 \item $P_\delta = \sum\limits_{\T\in \prbt\ast(\delta)} \x_{\T}$.
 \end{enumerate}
 \end{lemma}
 \begin{proof}
 Let $r$ be the largest multiplicity of $\delta$ and $\delta|^r = \alpha$. Then, the marked tableau $T\ast$ must be in layer $r$ and thus $\wt\ast(\T) = r$. In (1), (2), and (3), there are exactly $|T\ast| = \alpha_{\ell(\alpha)}$ ways of choosing the marked cell. Multiplying this by $\wt\ast(\T)$ accounts for the quantity $L(\delta)$ on the left hand side. For (4), the proof proceeds similarly to the proof of Lemma \ref{lem:pbtinterps}.
 \end{proof}

\section{Transition Matrices Between $H$, $E$, and $P$}\label{sec:HEP}

In this section, we prove the $H$, $E$, and $P$ expansions of $H_d$, $E_d$, and $P_d$ for $d\geq 0$. In Section \ref{ssec:HEP-rec}, we prove recursions relating $H_d$, $E_d$ and $P_d$ that are reminiscent of the recursions for Sym (Remark \ref{rem:sym-rec}). The bijections and sign-reversing involutions in Section \ref{ssec:HEP-rec} serve as a warm up for the proofs of propositions in later sections. In Section \ref{ssec:H-E}, we prove the expansions between $H$ and $E$ using the {stack-or-slash} operation. In Section \ref{ssec:P-HE}, marked polybar tableaux make their appearance. In Sections \ref{ssec:H-P} and \ref{ssec:E-P}, we study the cycles of permutations and utilize a technique described in \cite[Sec 7.2]{KLlocal}.
\subsection{Recursions among $H$, $E$, and $P$}\label{ssec:HEP-rec}
We first prove bijectively a formula stated without proof in \cite{polysymm} which relates $H_d$ and $E_d$. 
\begin{prop}[\cite{polysymm}, Remark 10]\label{prop:HE-rec}
For $d\geq 0$, we have $\sum_{k = 0}^d H_kE_{d -k} = \begin{cases} 0& \text{if }d > 0\\ 1 & \text{if } d = 0\end{cases}$.
\end{prop}
\begin{proof}
For $d = 0$, the left side is $\x_{\varnothing}\x_{\varnothing} =1$. Let $d>0$.
By using the results of Lemma \ref{lem:btinterps}, we only have to show
\[
\sum\limits_{k= 0}^d \sum\limits_{\substack{T\in \wbt[k]\\ U\in \sbt[d-k]}}\sgn(U)\x_T \x_U = 0.
\]
We define an involution $\psi$ on the set $\bigcup_{k=0}^d\wbt[k]\times \sbt[d-k]$ which acts on $(T,U)$ as follows: 

\begin{enumerate}
\item As $d>0$, at least one of $T$ or $U$ must be non-empty. If $U = \varnothing$ but $T\neq \varnothing$, then define $T'$ by removing the top row of $T$ and let $U'$ be the top row of $T$. If $T=\varnothing$ but $U\neq \varnothing$, the let $T'$ be the top row of $U$ and $U'$ be obtained from $U$ by removing the top row.
\item If the top row of $T$ is {strictly} larger in length than the top row of $U$, or has the same length but a {strictly} smaller label, then obtain $T'$ from $T$ by removing the top row and obtain $U'$ by inserting the top row of $T$ above the top row of $U$. 
\item If the top row of $T$ is strictly smaller in length than the top row of $U$, or has the same length but a {weakly} larger label, then obtain $T'$ by inserting the top row of $U$ above the top row of $T$ and define $U'$ by removing the top row from $U$.
\end{enumerate}
Let $\psi(T,U) = (T',U')$. Then $\x_{T}\x_U = - \x_{T'}\x_{U'}$ which allows us to cancel pairs of monomials. To see that $\psi$ is an involution consider the output $(T',U')$ obtained from $(T, U)$ in case (2). The top row of $U'$ must be the top row of $T$. As $T$ is a WBT, the top row of $T'$ must have a smaller length, or the same length but a weakly larger label, than the top row of $U'$. This means $\psi(T',U')$ is handled by case (3) and returns $(\T,\U)$. The other cases can be verified similarly.
 The following example illustrates the action of $\psi$: 
\[
\left(\:\yt{222,222,33,5}, \yt{333,1} \:\right) \xmapsto{\psi} \left(\:\yt{222,33,5}, \yt{222,333,1}\:\right)\qedhere
\]
\end{proof}
\begin{prop}
For $d\geq 0$, $dH_d = \sum_{i=1}^d H_{d-i}P_i$.
\end{prop}
\begin{proof}
When $d = 0$, the sum on the right hand side is empty and thus both sides of the equality are 0. Let $d>0$ and define $\wbt\ast[d]$ to be the set of weak bar tableaux with $d$ cells where we mark one cell.
Using the results of Lemma \ref{lem:btinterps}, we have to show
\[
\sum\limits_{T\in \wbt\ast[d]} \x_{T} = \sum\limits_{i=1}^d \sum\limits_{\substack{U\in \wbt[d-i]\\ V\in \rbt\ast[i]}} \x_U\x_V.
\]
 We describe a bijection $\varphi$ that maps an element $(U,V)$ of the set $\bigcup_{i=1}^d \wbt[d-i]\times \rbt\ast[i]$ to an element of $\wbt\ast[d]$ such that $\x_U \x_V = \x_{\varphi(U,V)}$. If $U$ contains a bar identical to the bars in $V$, then obtain $\varphi(U,V)$ by inserting $V$ below the lowest bar in $U$ identical to the bars in $V$, while preserving the location of the marked cell. If $U$ contains no bar identical to $V$ then obtain $\varphi(U,V)$ by inserting $V$ in a unique place in $V$ such that $\varphi(U,V)$ is a WBT. We now describe $\varphi^{-1}$. Let $T\in \wbt\ast[d]$ and let $B$ the marked bar in $T$. Define $\varphi^{-1}(T) = (U',V')$ such that $V'$ is the marked RBT formed by the bars identical to $B$ lying weakly below $B$, while $U'$ is obtained from $T$ by removing $V'$. This bijection is illustrated in the examples below:
\[
 \left( \yt{111,222,3,3},\yt{22{2\ast},222}\right) \xmapsto{\varphi}\yt{111,222,22{2\ast},222,3,3} \qquad  \left( \yt{111,3,3},\yt{4{4\ast},44}\right) \xmapsto{\varphi}\yt{111,4{4\ast},44,3,3}\qedhere
\]
\end{proof}

\begin{prop}
For $d\geq 0$, $dE_d = -\sum_{i=1}^d E_{d-i}P_i$.
\end{prop}
\begin{proof}
When $d = 0$, the sum on the right hand side is empty and thus both sides of the equality are 0. Let $d>0$, and define $\sbt\ast[d]$ to be the set of strict bar tableaux with $d$ cells where one cell is marked.
Using the results of Lemma \ref{lem:btinterps}, we rewrite the above in terms of monomials as
\[
\sum\limits_{T\in \sbt\ast[d]}(-1)^{\ell(T)} \x_T = \sum\limits_{i=1}^d \sum\limits_{\substack{U\in \sbt[d-i]\\ V\in \rbt\ast[i]}} (-1)^{\ell(U)+1} \x_U \x_V.
\]
We define a sign-reversing involution $\rho$ on the set $\bigcup_{i=1}^d \sbt[d-i]\times \rbt\ast[i]$ such that the fixed point set is in bijection with $\sbt\ast[d]$. For a fixed point $(U_0, V_0)$, let the marked SBT in bijection with $\rho(U_0, V_0)$ be $Z$, then we have $(-1)^{\ell(U_0)+1} \x_{U_0}\x_{V_0} = (-1)^{\ell(Z)}\x_{Z}$. Denote $\rho(U,V)$ by $(U',V')$. As $i> 0$, the marked RBT $V$ is non-empty. Let $B$ be the bottom bar of $V$. If $U$ contains a bar $B'$ identical to $B$, then define $U'$ by removing $B'$ from $U$, and define $V'$ by inserting $B'$ in $V$ below $B$. On the other hand, if $U=\varnothing$ or if $U$ contains no bars identical to $B$ and $\ell(V) > 1$, obtain $U'$ by inserting $B$ in $U$ such that $U'$ is an SBT and $V'$ by removing $B$ from $V$. In both these cases, we have $(-1)^{\ell(U')}\x_{U'}\x_{V'} = - (-1)^{\ell(U)+1}\x_{U}\x_{V}$. The following example illustrates the action of the involution $\psi$:

\[
\left(\yt{111,222,33},\yt{2{2\ast}2}\right) \xmapsto{\psi} \left(\yt{111,33},\yt{2{2\ast}2,222}\right) 
\]

In the remaining case, we have that $V$ is a single marked bar and $U$ does not contain a bar identical to $V$. Construct the corresponding $T\in \sbt\ast[d]$ by inserting $V$ in $U$ in the unique location such that $U$ is an SBT and we have $\x_{T} = \x_U \x_V$. As $\ell(T) =  \ell(U) + 1$, the monomials appear with the correct sign. The following example illustrates the correspondence:
\[
\left(\yt{222,333,44,5},\yt{2{2\ast}}\right) \leftrightarrow \yt{222,333,2{2\ast},44,5}
\]
\end{proof}
\begin{remark}\label{rem:sym-rec}
The recursions discussed above are quite similar to the corresponding recursions among symmetric polynomials. We have $\sum\limits_{i=0}^d (-1)^{i}h_{d-i}e_{i} =  \begin{cases} 0& \text{if }d > 0\\ 1 & \text{if } d = 0\end{cases}$ \cite[Thm 9.81]{loehr-comb}, $\sum_{i=1}^{d} h_{d-i} p_{i} = dh_d$ \cite[Thm 9.88]{loehr-comb}, and $\sum\limits_{i=1}^{d} (-1)^{i-1}e_{d-i} p_{i} = de_d$ \cite[Thm 9.89]{loehr-comb}. The combinatorial proofs of these recursions can be found in these citations.

\end{remark}

\subsection{$H$ and $E$}\label{ssec:H-E}
We prove the expansion in this section using a sign-reversing involution which we call the {stack-or-slash operation}. 

\boks{0.35}
We first describe the \textit{weak stack-or-slash operation}. For bars $B$ and $B'$, we call $(B,B')$ a \textit{pair} in $\T$ if $B'$ appears immediately after $B$ in the scanning order, and $B$ and $B'$ occur in the same layer. A pair $(B,B')$ called an \textit{identical pair} if $B$ and $B'$ are identical (same size and same label). We can extend this definition to say that $(T,T')$ is a \textit{pair of tableaux} if $T'$ occurs immediately after $T$ in the same layer of $\T$.
We consider pairs because when they occur in $\T$ in a certain arrangement, we can rearrange them such that the newly obtained object $\T'$ has exactly one tableau more or less than $\T$. We will see that this leads to a sign-reversing involution on polyWBTs. We say that
\begin{enumerate}
\item the pair of bars $(B,B')$ satisfies the \textit{decreasing parts condition} if $B$ contains strictly more cells than $B'$.
\item the pair of bars $(B,B')$ satisfies the \textit{weakly increasing labels condition} if $B$ and $B'$ have the same number of cells but the label of $B'$ is weakly greater than the label of $B$.
\end{enumerate}
We say a pair $(B,B')$ is a \textit{(weak) first instance}\footnote{We will drop the adjective ``weak'' when it is clear from context.} if $B$ is the first bar in the scanning order such that $(B,B')$ satisfies (1) or (2).  Let a PWBT $\T$ contain the bars $B$ and $B'$ such that $(B,B')$ is the first instance.
Define the output $\T'$ of the \textit{weak stack-or-slash} operation on $\T$ as follows.
\begin{itemize}
\item \textbf{Slash}: If $B$ and $B'$ occur within the same WBT $T$ in $\T$, then let $T'$ be the WBT containing the top rows of $T$ up to and including the bar $B$. Let $T''$ be the WBT formed by the rows of $T$ below and including $B'$. Let $\T'$ be obtained from $\T$ by removing $T$ and placing in its position the pair $(T',T'')$. In the following example, the pair $(B = \yt{111},B' = \yt{22})$ in layer 1 of $\T$ is the first instance and it satisfies the decreasing parts condition.
\[
\T = \left( \yt{222} \quad \yt{111,22,22} \quad \yt{33}\right) \to \T' = \left( \yt{222} \quad \yt{111} \quad \yt{22,22} \quad \yt{33}\right)
\]
\item \textbf{Stack}: If $B$ and $B'$ are in different tableaux, say $T$ and $T'$ respectively, then it must be the case that $B$ is the bottom row of $T$ and $B'$ is the top row of $T'$. Define $T''$ to be the WBT whose top rows form the subtableau $T$ while the rest of the  rows form the subtableau $T'$. Obtain $\T'$ from $\T$ by removing the pair $(T,T')$ and replacing it with $T''$. In the following example, we apply the weak stack-or-slash operation. Then $(B = \yt{33} ,B' = \yt{33})$ (in layer 1 of $\T$) is the first instance and satisfies the weakly increasing labels condition. 
\[
\T = \left(\yt{33} \quad \yt{33,44} \quad \yt{55,2}\right) \to \T' = \left(\yt{33,33,44} \quad \yt{55,2}\right).
\]

\end{itemize}

We define the \textit{strict stack-or-slash operation} acting on PSBTs as a slight modification of the weak stack-or-slash operation. We say that
\begin{enumerate}
\item[(3)] the pair $(B,B')$ satisfies the \textit{strictly increasing labels condition} if $B$ and $B'$ have the same number of cells but the label of $B'$ is strictly greater than the label of $B$.
\end{enumerate}

The pair $(B,B')$ is a \textit{(strict) first instance} if $B$ is the first bar in the scanning order that satisfies the conditions (1) and (3). The action on the bars remains the same as in the weak stack-or-slash operation.
For the example below, $(B = \yt{33},B' = \yt{44})$ is the strict first instance (in layer 1 ) and satisfies the strictly increasing labels condition. This yields the following output:
\[
\T = \yt{33} \quad \yt{33,44,2} \quad \yt{55,2} \to \T' = \yt{33}\quad \yt{33}\quad \yt{44,2} \quad \yt{55,2}
\]

Note that the weak stack-or-slash operation sends PWBTs to PWBTs and the strict variation sends PSBTs to PSBTs. One can also check that both variations of the stack-or-slash operation are involutions as they preserve the pair that is the first instance. For $\delta\in \PCom$, if $\T\in \pwbt(\delta)$ contains no first instance, then $\T$ is a fixed point of the operation. Explicitly, the fixed points under the weak stack-or-slash operation are PWBTs $\T = (T_1, \ldots, T_k)$ where each $T_i$ is a bar, the sizes of $T_i$ increase weakly as $i$ increases, and the labels between $T_i$s of the same size decrease strictly. Similarly, the fixed points under the strict stack-or-slash operation are PSBTs $\T = (T_1, \ldots, T_k)$ where each $T_i$ is a bar, the sizes of $T_i$ increase weakly with $i$, and the labels between $T_i$s of the same size decrease weakly.

In the following example, $\T$ is a fixed point under the weak stack-or-slash operation and $\U$ is a fixed point under the strict stack-or-slash operation.
\[
\T = \yt{4} \quad \yt{2} \quad \yt{22} \quad \yt{11} \quad \yt{7777} \quad \yt{33333}
\]
\[
\U = \yt{4} \quad \yt{2} \quad \yt{22} \quad \yt{22} \quad \yt{22} \quad \yt{3333}
\]

We call a polycomposition $\delta$ of $n$ \textit{square-free} if all multiplicities are equal to 1. Each square-free polycomposition $\delta$ can be written as $(\alpha)^1$ for some $\alpha\in \Com(n)$. We denote the set of square-free polycompositions of $n$ by $\PCom_{\sqf}(n)$.

\begin{prop}\label{prop:H-E,E-H}
 For $d\geq 0$, 
 \begin{enumerate}
 \item $H_d = \sum_{\delta \in \PCom_{\sqf}(d)} (-1)^{\ell(\delta)}E_\delta$.
 \item $E_d = \sum_{\delta \in \PCom_{\sqf}(d)} (-1)^{\ell(\delta)}H_\delta$.
 \end{enumerate}
\end{prop}

\begin{proof}
We first prove (1). By using Lemmas \ref{lem:btinterps} and \ref{lem:pbtinterps}, we have to show
\[
\sum\limits_{T\in \wbt[d]} \x_T= \sum\limits_{\delta\in \PCom_{\sqf(d)}} \sum\limits_{\T\in \psbt(\delta)}(-1)^{\ell(\T)} \psgn(\T) \x_{\T}.
\]
We define $\sigma:= \mcI_E^H$ to be a map on $\bigcup_{\delta\in \PCom_{\sqf}(d) } \psbt(\delta)$ such that $\mcI(\T)$ is the output of the strict stack-or-slash operation on $\T$. This makes $\mcI$ an involution on $\bigcup_{\delta \in \PCom_{\sqf}(n)} \psbt(\delta)$ such that $\mcI(\T)$ has one tableau more or less than $\T$. This means that for $\T$ that is not a fixed point of $\mcI$, we have $(-1)^{\ell(\T)}\psgn(\T)\x_{T} = -(-1)^{\ell(\mcI(\T))}\psgn(\sigma(\T))\x_{\mcI(\T)}$. The fixed points under $\mcI$ are PSBTs $\T = (B_1, \ldots, B_{\ell(\delta)})$ where $B_i$s are bars, the sizes of $B_i$s increase weakly with index $1\leq i \leq \ell(\delta)$ and the labels decrease weakly between bars of the same size. These fixed points are in bijection with WBTs of size $d$ where the $i$th row from top of the WBT corresponding to $\T$ is identical to the bar $B_{\ell(\T)-i+1}$. This can be informally seen as reading the bars of $\T$ from right to left and constructing the rows of the weak bar tableau from top to bottom. The monomials $\x_{\T}$ corresponding to the fixed points $\T$ appear on the right hand side with the sign 1 as the total number of bars in $\T$, $B(\T)$, is equal to $\ell(\T)$ and thus $(-1)^{\ell(\T)}\psgn(\T) = (-1)^{\ell(\T) + B(\T)} = 1$.

The proof of (2) requires us to show
\[
\sum\limits_{T\in \sbt[d]} \sgn(T)\x_T = \sum\limits_{\delta\in \PCom_{\sqf}(d)} \sum\limits_{\T\in \pwbt(\delta)}(-1)^{\ell(\T)} \x_{\T}.
\]
We proceed as in the proof of (1) by defining $\sigma' = \mcI_H^E$ to be the weak stack-or-slash operation. For PWBTs $\T$ that are not fixed points, $(-1)^{\ell(\T)}\x_{T} = -(-1)^{\ell(\mcI'\T))}\x_{\mcI'(\T)}$ holds which allows us to pair up and cancel monomials with opposite signs. The fixed points under this map are PWBTs $\T = (B_1, \ldots, B_{\ell(\delta)})$ where $B_i$s are bars, the sizes of $B_i$s increase weakly and the labels decrease strictly between bars of the same size. Each fixed PWBT $\T$ corresponds to an SBT $T$ of size $d$ by considering $B_{\ell(\T)-i+1}$ as the $i$th row from top of $T$. Similar to the construction in (1), reading the bars of the (fixed point) PWBT from right to left corresponds to reading the SBT from top to bottom. As the number of bars in the fixed point PWBT $\T$ is the same as the number of rows of the corresponding SBT $T$, we have $\ell(T) = B(\T)$, and thus the signs on both sides match.
\end{proof}
\begin{example}
\boks{0.36}
In the $E$-expansion of $H_{10}$, an example of $\sigma^H_E$ acting on PSBTs contributing to the monomial $x_{32}x_{33}x_{22}^2$ is
\[
\yt{222} \quad \yt{333,22} \quad \yt{22} \xmapsto{\sigma^H_E} \yt{222,333,22} \quad \yt{22}
\]
An example of a fixed PSBT under $\sigma^H_E$ mapping to its associated WBT for $d = 10$ is 
\[
\yt{3} \quad\yt{1} \quad \yt{22} \quad \yt{333} \quad \yt{333}  \mapsto \yt{333,333,22,1,3}
\]
\end{example}
\begin{remark}\label{rem:omega}
In the case of symmetric functions, we have $h_d = \sum_{\alpha\in \Com(d)} (-1)^{d-\ell(\alpha)} e_\alpha$ and $e_d = \sum_{\alpha\in \Com(d)} (-1)^{d-\ell(\alpha)} h_\alpha$ (see \cite[Prop 4.3]{gelfand}, \cite[Eq 17]{KLlocal}). This allows us to deduce that the algebra homomorphism $\omega$ on $\Sym$ which maps $h_d$ to $e_d$ is an involution. The analogous formulas in the polysymmetric case allows us to show that algebra homomorphism $\Omega$ on $\PSym$ which maps the algebraically independent generators $H_{d^r}$ to $E_{d^r}$ is an involution.  This map is defined in Proposition 3.2 in \cite{polysymm} but the proof is based on generating functions. By using the definition of $H_{d^r}$ and $E_{d^r}$, we get $H_{d^r} = \sum_{\delta \in \PCom_{\sqf}(n)} (-1)^{\ell(\delta)}E_{\delta^r}$ and $E_{d^r} = \sum_{\delta \in \PCom_{\sqf}(n)} (-1)^{\ell(\delta)}H_{\delta^r}$. We use the fact that $\PSym$ is a $\Q$-algebra generated by the algebraically independent set $\{H_{d^r}:d,r\geq 1\}$ \cite[Thm 3.1]{polysymm}. 

We have $\Omega(E_{d^r}) = \Omega(\sum_{\alpha\in \Com(d)}(-1)^{ \ell(\alpha)} H_{(\alpha)^r})$. As $\Omega$ is an algebra homomorphism, we can write $\Omega(E_{d^r}) = \sum_{\alpha\in \Com(d} (-1)^{\ell(\alpha)} E_{(\alpha)^r} = H_{d^r}$. This shows that $\Omega\circ \Omega$ is the identity on $\PSym$ and thus $\Omega$ is an algebra involution.
\end{remark}

\subsection{$P$ in $H$ and $E$}\label{ssec:P-HE}
We now present the $H$ and $E$ expansions of $P_d$ which do not utilize the stack-or-slash operation. In the stack-or-slash operation, the first instance decides which bar is affected by our involution. As we will see, the objects that appear in the $H$ and $E$ expansions of $P$ are marked bar polytableaux and the bar on which the involution acts is dictated by the marking.

\begin{prop}\label{prop:P-H,P-E}
For $d\geq 0$, 
\begin{enumerate}
\item $P_d = \sum_{\delta \in \PCom_{\sqf}(d)} (-1)^{\ell(\delta)-1} L(\delta) H_\delta$.
\item $P_d = \sum_{\delta  \in \PCom_{\sqf}(d)} (-1)^{\ell(\delta)} L(\delta) E_\delta$.
\end{enumerate}
\end{prop}

\begin{proof}[Proof of Proposition \ref{prop:P-H,P-E} (1)]
Using the expansions (1) and (4) from Lemma \ref{lem:markedinterps}, we have to show 
\[
\sum\limits_{\T\in \rbt\ast[d]} \x_{\T} = 
 \sum\limits_{\substack{\delta\in \PCom_{\sqf}(d)\\ \T\in \pwbt\ast(\delta)}} (-1)^{\ell(\T)-1} \wt\ast(\T)\x_{\T} =  \sum\limits_{\substack{\delta\in \PCom_{\sqf}(d)\\ \T\in \pwbt\ast(\delta)}} (-1)^{\ell(\T)-1}\x_{\T}. \] As the polycompositions are square-free, the only multiplicity that appears is 1. In particular, the marked bar appears in layer 1, and thus $\wt\ast(\T) = 1$. We define an involution $\mcI^P_H$ on the set $\bigcup_{\delta \in \PCom_{\sqf}(d)} \pwbt\ast(\delta)$ such that $\mcI^P_H(\T)$ has one diagram more or less than $\T$, or $\T$ is a fixed point of $\sigma^P_H$. Let $\T = (T_1, \ldots, T_k, T\ast)$ and let $B\ast$ in $T^*$ be the bar containing the marked cell.
 \boks{0.38}
 \begin{itemize}
 \item If not all bars in $T\ast$ are identical to $B\ast$, or if $T\ast$ has all identical bars but $B\ast$ is not in the top row, then define $T'$ to be the tableau formed by all bars in $\T\ast$ identical to $B\ast$ weakly below $B\ast$, and let $T$ be obtained from $T^*$ by removing $T'$. Note that $T'$ contains the marked cell in the same column as $B\ast$ in the top row. Define $\mcI^P_H(\T) = (T_1, \ldots, T_k, T, T')$. The following three examples illustrate this map:
 \[
        \text{Example 1: } \quad \yt{11,22,1} \quad \yt{111,33,{3^*}3, 33,44} \xmapsto{\mcI^P_H} \quad \yt{11,22,1} \quad \yt{111, 33,44} \quad \yt{{3^*}3,33}
        \]
         \[
        \text{Example 2: } \quad \yt{11,22,1} \quad \yt{{3^*}3, 33,44,1,1} \xmapsto{\mcI^P_H} \quad \yt{11,22,1} \quad \yt{44,1,1} \quad \yt{{3^*}3,33}
        \]
        \[
        \text{Example 3: }\quad  \yt{11,22,1} \quad \yt{222,222,2{2^*}2,222,222} \xmapsto{\mcI^P_H} \quad \yt{11,22,1} \quad \yt{222,222}\quad \yt{2{2^*}2,222,222} 
        \]
 \item Suppose all bars in $T\ast$ are identical to $B\ast$ and $B\ast$ is in the top row. If $\ell(\T) > 1$, then we obtain a WBT $T$ by inserting $T\ast$ in $T_k$ immediately below the lowest bar identical to $B\ast$. If such a lowest bar does not exist, we insert $T\ast$ in $T_k$ in a unique position such that $T$ is a WBT.
  In this case, define $\mcI^P_H(\T) = (T_1, \ldots, T_{k-1}, T)$. The examples for this case can be constructed by considering the reverse direction in the above three examples. 
 \end{itemize}
 The only remaining case is when all bars in $T\ast$ are identical to $B\ast$, $B\ast$ is in the top row and $T\ast$ is the only RBT in $\T$, that is, $\ell(\T) = 1$. In this case, define $\sigma^P_H(\T) = \T$.
 
For non-fixed points $\T$
the length $\ell(\T)$ changes by exactly 1, but we have $\x_{\sigma^P_H(T)} = \x_\T$ which allows us to cancel the monomials arising from $\T$ and $\sigma^P_H(\T)$. Each fixed point is a PWBT $\T = (T\ast)$ such that $T\ast$ is a bar tableau containing all identical bars with the marked cell in the top row. This is exactly the set $\rbt\ast[d]$. The sign of $\x_{\T}$ for fixed points $\T$ is $(-1)^{1-1} = 1$ as needed.
\end{proof}

The proof for the $E$-expansion for $P_d$ is similar in flavor, but the fixed points look quite different.
\begin{proof}[Proof of Proposition \ref{prop:P-H,P-E}(2)]\label{pf:PE}
Using expansions (3) and (4) from Lemma \ref{lem:markedinterps}, we need to show
\[
\sum\limits_{\T\in \rbt\ast[d]} \x_{\T} = 
 \sum\limits_{\substack{\delta\in \PCom_{\sqf}(d)\\ \T\in \psbt\ast(\delta)}} (-1)^{\ell(\T)} \wt\ast(\T)\psgn(\T)\x_{\T} =  \sum\limits_{\substack{\delta\in \PCom_{\sqf}(d)\\ \T\in \pwbt\ast(\delta)}} (-1)^{\ell(\T)}\psgn(\T)\x_{\T}. \]
We have $\wt\ast(\T) = 1$ as layers below 1 are empty. We define a sign-reversing involution $\mcI^P_E$ on the set $\bigcup_{\delta\in \PCom_{\sqf}(d)} \psbt\ast(\delta)$ acting on an element $\T = (T_1, \ldots, T_k, T_{k+1})$ where $T_{k+1}$ contains the bar $B\ast$ with a marked cell. Suppose $j\in \{1, \ldots, k+1\}$ be the smallest index such that either of the following conditions is satisfied: (i) $T_j$ has more than one row and contains a bar identical\footnote{Recall that $B$ is identical to $B'$ if $B$ and $B'$ have the same length and label. It does not matter if one or both of $B$ and $B'$ contain a  marked cell.} to $B\ast$, or (ii) $j>1$ and $T_j$ is a bar identical to $B\ast$ with $T_{j-1}$ not containing a bar identical to $B\ast$. If (i) is satisfied, then define $T'$ to be a bar identical to $B\ast$ and $T$ to be the SBT obtained from $T_j$ by removing $T'$. As $T$ is an SBT it does not contain a bar identical to $B\ast$. Let $\mcI^P_E(\T) = (T_1, \ldots, T_{j-1}, T, T', T_{j+1}, \ldots, T_{k}, T_{k+1})$. If (ii) is satisfied, then let $T$ be the SBT obtained by inserting $T_j$ in $T_{j-1}$, and define $\mcI^P_E(\T) = (T_1, \ldots, T_{j-2}, T, T_{j+1}, \ldots, T_k, T_{k+1})$. Such an insertion is always possible for $j>1$, as otherwise we must have a bar identical to $B\ast$ in $T_{j-1}$, which would satisfy case (i) and $j$ would not be the smallest index satisfying our conditions. If the bar containing the marked cell is inserted or removed in these operations, then we preserve the position of the marked cell. In Example (i) below, condition (i) is satisfied with $j = 2$, while in Example (ii), condition (ii) is satisfied with $j = 4$:
\[
        \text{Example (i): } \quad \yt{111} \quad \yt{111,33,44} \quad \yt{1{1\ast}1, 222} \xmapsto{\mcI^P_E} \quad \yt{111} \quad \yt{33,44} \quad \yt{111} \quad \yt{1{1\ast}1, 222}
        \]
        \[
        \text{Example (ii): }\quad  \yt{111} \quad \yt{111} \quad \yt{333} \quad \yt{1{1\ast}1} \xmapsto{\mcI^P_E}\quad  \yt{111} \quad \yt{111} \quad \yt{1{1\ast}1,333} 
        \]
 Note that these operations change the number of SBTs in $\T$, that is $\ell(\T)$ by 1, but the number of bars in $\T$ and $\mcI^P_E(\T)$ are the same which means $\psgn(\T) = \psgn(\mcI^P_E(\T))$. This implies $(-1)^{\ell(\T)}\psgn(\T)\x_{\T} = - (-1)^{\ell(\mcI^P_E(\T))} \psgn(\mcI^P_E(\T))\x_{\mcI^P_E(\T)}$, which allows us to pair these terms and cancel them. If the conditions (i) and (ii) do not hold for a $\T$, then we define $\sigma^P_E(\T) = \T$.
 For the marked PSBTs $\T = (T_1, \ldots, T_{k+1})$ that are fixed under $\sigma^P_E$, we must have that all $T_i$ for $1\leq i \leq k+1$ are bars identical to $B\ast$. We map these to marked RBTs of size $d$ where the $i$th row from top is $T_{\ell(\T) -i +1}$. As a fixed point $\T$ contains $\ell(\T)$ single bars, we have $\psgn(\T) = (-1)^{\ell(\T)}$. Multiplying this by the sign $(-1)^{\ell(\T)}$ gives 1 which is the coefficient on the left hand side. The following example shows the mapping of a fixed point marked PSBT to a marked RBT:
 \boks{0.35}
 \[
 \yt{111}\quad \yt{111}\quad \yt{11{1\ast}} \mapsto \yt{11{1\ast},111,111}
 \]
\end{proof}
\begin{corollary}
For $\delta \in \PCom$, we have $\Omega(P_\delta) = (-1)^{\ell(\delta)}P_\delta$.
\end{corollary}
\begin{proof}
From the expansions in Proposition \ref{prop:P-H,P-E}, we deduce $\Omega(P_d) = -P_d$ because $\Omega(H_\delta) =E_\delta$ as shown in Remark \ref{rem:omega}. We also have $P_{d^r} = \sum_{\delta \in \PCom_{\sqf}(d)} (-1)^{\ell(\delta)-1} L(\delta) H_{\delta^r}$, which yields $\Omega(P_{d^r}) = - P_{d^r}$.
By using the fact that $\Omega$ is an algebra isomorphism, we get $\Omega(P_\delta) = \prod_{d^r\in \delta} \Omega(P_{d^r}) = (-1)^{\ell(\delta)} P_\delta$.
\end{proof}

\subsection{$H$ in $P$}\label{ssec:H-P}
The proofs of the expansion in this section and the next section involve operations on cycles of permutations. Every permutation $\pi\in \S_n$ can be decomposed as a product of disjoint cycles $\pi = C_k C_{k-1} \ldots C_1$ with $\ell(C_i)$ being the number of entries in $C_i$. Any reordering of the cycles as well as any cyclic shift of the elements within each cycle preserves the permutation. If we write a cycle as $C = (c_1, c_2, \ldots, c_l)$, then the notation $C(i)$ stands for the entry $c_i$. We say a permutation is in \textit{decreasing cycle form} if $\ell(C_k) \geq \ell(C_{k-1}) \geq \ldots \geq \ell(C_1)$, each cycle begins with its minimum element, and the minimum elements of cycles of the same size are in decreasing order. Define $\cycP(\pi) = (\ell(C_k), \ell(C_{k-1}), \ldots, \ell(C_1))\in \Par(n)$. We say a permutation $\pi =C_k'C_{k-1}'\ldots C_1'$ is in \textit{canonical form}\footnote{In the literature, this form is related to the \textit{standard form} as mentioned in Example 13 of \cite{bergeron}. The standard form, however, places cycles with smaller entries to the left while we place the cycles with smaller entries to the right.} if the first entry in each cycle is the minimum entry in that cycle and we order the cycles such that the minimum element of $C'_{i}$ is smaller than the minimum element of $C'_{i+1}$ for all $1\leq i \leq k-1$. 
With a permutation $\pi = C_k'C_{k-1}'\ldots C_1'$ in canonical form, we associate the composition $\cycC(\pi) =(\ell(C_k'),\ell(C_{k-1}'),\ldots, \ell(C_1'))$. For instance, the permutation in decreasing cycle form $\pi = (183)(57)(26)(4)\in S_8$ has $\cycP(\pi) = (3,2,2,1)$ and its canonical form is $(57)(4)(26)(183)$ with $\cycC(\pi) = (2,1,2,3)$. For $\lambda\in \Par(n)$, define $z_\lambda = \prod\limits_{i\geq 1} i^{m_i(\lambda)}m_i(\lambda)!$ where $m_i(\lambda)$ is the number of times part $i$ appears in $\lambda$. It is well-known (see \cite[Thm. 7.115]{loehr-comb}) that $n!/z_\lambda$ counts the number of $\pi \in \S_n$ with $\cycP(\pi) = \lambda$. For a composition $\alpha = (\alpha_1, \ldots, \alpha_l)\in \Com(n)$, define 
\[
Z_\alpha = (\alpha_1)(\alpha_1 + \alpha_2)(\alpha_1 + \alpha_2 + \alpha_3) \ldots (\alpha_1 + \alpha_2 + \ldots + \alpha_l).
\]
Let $K_\alpha$ be the set of $\pi\in \S_n$ with $\cycC(\pi) = \alpha$.
\begin{lemma}[Lemma 19 in \cite{KLlocal}]\label{lem:z-harmonic}\noindent
\begin{enumerate}
\item For a composition $\alpha$ of $n$, $|K_\alpha| = n!/Z_\alpha$. 
\item For a partition $\lambda$ of $n$,   $n!/z_\lambda = \sum_{\alpha: \sort(\alpha) = \lambda} |K_\alpha| $. In other words, $z_\lambda$ is the harmonic mean of $Z_\alpha$ where $\alpha$ ranges over compositions that sort to $\lambda$.
\end{enumerate}
\end{lemma}
\begin{proof}
To prove (1), we first fix a composition $\alpha = (\alpha_1, \ldots, \alpha_l)\in \Com(n)$. To construct a permutation $\pi = C_lC_{l-1}\ldots C_1\in K_\alpha$, we proceed as follows: we construct the cycle $C_1$ of length $\alpha_l$ by first defining $C_1(1)$ as the smallest unused entry which is $1$. We now have $\alpha_l - 1$ spots in $C_1$ to fill using $n-1$ entries in $\{2,3,\ldots, n\}$. We can fill the $\alpha_l - 1$ spots $C_1(2), \ldots ,C_1(\alpha_l)$ in $(n-1) \cdot (n-2) \cdot \ldots \cdot (n - \alpha_{l}+1)$ ways. We proceed to fill the cycle $C_{2}$ of length $\alpha_{l-1}$. We choose $C_{2}(1)$ to be the smallest entry not yet used to fill $C_1$. We already used $\alpha_l$ numbers to fill $C_1$ and another number to fill the first spot of $C_{2}$, which leaves us with filling the $\alpha_{l-1}-1$ spots in $C_2$ from $n-\alpha_l - 1$ entries chosen without repetition and in an order. So, the $\alpha_{l-1}-1$ spots of $C_{2}$ can be filled in $(n - \alpha_l-1)\cdot (n - \alpha_l-2) \cdot \ldots \cdot (n - \alpha_l - \alpha_{l-1}+1)$ ways. In general, we can fill the cycle $C_{l+1-i}$ in 
$
w_i := (n-\sum_{j=i}^{l-1} \alpha_i -1)(n-\sum_{j=i}^{l-1} \alpha_i -2) \ldots (n-\sum_{j=i}^{l} \alpha_i +1)
$
ways. In the product $\prod_{i=1}^l w_i$, we find the missing factors are of the form $(\alpha_1 + \alpha_2 + \ldots + \alpha_i)$. Thus we find
\[
\prod\limits_{i=1}^l  (\alpha_1  + \alpha_2 + \ldots + \alpha_{i}) w_i  = n!
\]
which gives us $|K_\alpha| = \prod_{i=1}^l w_i =  n!/Z_\alpha$.

To prove (2), we observe that any permutation $\pi$ with $\cycC(\pi) = \alpha$ has $\cycP(\pi) = \sort(\alpha)$ which can be seen by reordering the cycles of $\pi$ in decreasing order of length. Furthermore, any $\rho\in S_n$ with $\cycP(\rho) = \lambda$ can be expressed in canonical notation by reordering its cycles and it must be that $\cycC(\rho)$ is a rearrangement of $\lambda$. This shows that there is a bijection between $\{\pi \in S_n : \cycP(\pi) = \lambda\}$ and the disjoint union $\bigcup_{\alpha: \sort(\alpha) = \lambda} K_\alpha$. By using the result from (1), we have the statement (2).
\end{proof}
\begin{example}
We construct the permutation $\pi = C_3C_2C_1 = (4,7,5)(2,6)(1,8,3,9)\in \S_9$ with $\cycC(\pi) = (3,2,4)$ to illustrate the process described above. We start with $C_1 = (1,\blk,\blk,\blk)$. This leaves us with 8 possible options for the second spot, $C_1(2)$. We choose to fill it with $C_1(2)= 8$. Then, the third spot $C_1(3)$ has 7 options which we fill with 3, and finally we have 6 options to fill $C_1(4)$ , which we do so with a 9. This shows that there are $8\cdot 7\cdot 6$ ways to construct $C_1$, one of which is $(1,3,8,9)$. For $C_2$, we start with $(2,\blk)$ as 2 is the smallest unused entry. The remaining spot can be filled in 4 ways, which we do with 6 giving us $C_2 = (2,6)$. For $C_3$, we start with the smallest available value, 4 as 3 was used in $C_1$. To construct $C_3$ we begin with $(4,\blk,\blk)$ and the remaining spots can be filled in $2\cdot 1$ ways. Note that the number of choices are independent of the specific values we chose and so $|K_\alpha| = (8\cdot 7 \cdot 6) \cdot (4) \cdot (2\cdot 1) = 8!/(3\cdot 5 \cdot 9)$.
\end{example}
Define $\CS_n$, the set of \textit{choice sequences of length $n$}, to be tuples of positive integers of the form $\textbf{c}:=(c_n, c_{n-1}, \ldots, c_1)$ satisfying $1\leq c_i \leq i$ for all $i$. It is routine to verify that $|\CS_n| = n!$ and we will employ this equality in the proof of Propositions \ref{prop:H-P} and \ref{prop:E-P}. Furthermore, for any set $S$ and natural number $c$, define $c\acts S$ to be the $c$\,th smallest entry of $S$, and extend this definition to sub-sequences $(c_i, c_{i-1}, \ldots, c_j)$ of choice sequences as follows: define $a_i$ to be $c_i\acts S$, and let $S'$ be obtained from $S$ by removing $a_i$. Then, define $a_{i-1}=c_{i-1}\acts S'$ and let $S''$ be obtained from $S'$ by removing $a_{i-1}$. Continue performing this action and finally obtain $(a_i, a_{i-1}, \ldots a_j) := (c_i, c_{i-1}, \ldots, c_j) \acts S$. For instance, $ 3\acts \{1,2,4,7,9\} =4$ and $(3,1,2)\acts \{1,2,4,7,9\} = (4,1,7)$. The idea of using choice sequences in the next proof is inspired by \cite[Sec 7.2]{KLlocal}.
  
For a square-free polycomposition $\delta = (\alpha)^1$, define $Z_\delta = Z_\alpha$ and $K_\delta = K_\alpha$.

\begin{prop}\label{prop:H-P}
For $d\geq 0$, $\displaystyle{H_d = \sum\limits_{\delta \in \PCom_{\sqf}(d)} \dfrac{P_\delta}{Z_\delta}}$.
\end{prop}

\begin{proof}
Multiplying both sides by $d!$ and then using the monomial expansions in Lemmas \ref{lem:pbtinterps} and \ref{lem:markedinterps}(4), we must show
\[
\sum\limits_{T\in \wbt[d]} |\CS_d| \x_T = \sum\limits_{\delta \in \PCom_{\sqf}(d)} |K_\delta| \sum\limits_{\T\in \prbt\ast(\delta)} \x_{\T}.
\]
As there are no signs involved, we can prove the above statement by establishing a monomial weight-preserving bijection between $\CS_d\times\wbt[d]$ and $\bigcup_{\delta \in \PCom_{\sqf}(d)} K_\delta \times \prbt\ast(\delta)$. We first explicitly construct the bijection $\phi:\CS_d\times\wbt[d] \to \bigcup_{\delta \in \PCom_{\sqf}(d)} K_\delta \times \prbt\ast(\delta)$. Suppose the input is of the form $(\textbf{c} = (c_d, \ldots, c_1), T)$ where $\textbf{c}$ is a choice sequence and $T$ is a WBT.

We number the cells of the WBT $T$ in reading order, that is, we label the cells within a bar from left to right using consecutive natural numbers, and we label the bars in scanning order with the condition that the first element in each bar is the smallest unused element upto that point. The following is an example of cell numbering in reading order:
\[
\yt{12345,67,89}
\]
of the WBT $T = \yt{22222,11,33}$. The reader can concretely understand the description of $\phi$ below by following along with Example \ref{ex:H-P}.
\begin{itemize}
\item Let $T^{(1)} = T$ and $S^{(1)} = \{1, \ldots, d\}$. We mark the cell numbered $c_d$ in $T$ with $*$. Define $T_1$ to be the marked RBT formed by the bar $B$ containing the marked cell and bars below $B$ identical to $B$. Denote $\delta_1 = |T_1|$ and construct the cycle $C_1$ of length $\delta_1$ as $C_1 = (1,c_{d-1}, \ldots, c_{d - \delta_1 + 1})\acts S^{(1)}$. Note the $C_1$ starts with a 1.
\item We now remove $T_1$ from $T$ to obtain $T^{(2)}$ and obtain $S^{(2)}$ from $S^{(1)}$ by removing the entries in $C_1$. We number the cells of $\dg(T^{(2)})$ in reading order and mark the cell numbered $c_{d-\delta_1}$. Define $T_{2}$ to be the RBT formed by the bar containing the marked cell and all bars below it which are identical to the bar containing the marked cell. Let $\delta_2 = |T_2|$ and construct the cycle $C_2 = (1, c_{d-\delta_1  - 1}, c_{d - \delta_1 -2}, \ldots, c_{d - \delta_1 - \delta_2 + 1})\acts S^{(2)}$. Note that in this case, and for the rest of the cycles, the first value of the cycle is the smallest available value.
\item Continue this process to construct cycles $C_3, \ldots, C_k$ and PRBTs $T_3, \ldots, T_k$ for some $k >0$. Define $\pi =C_k\ldots C_1$ which is a permutation such that $\cycC(\pi) = (\delta_k, \ldots, \delta_1)$. Also, define $\T = (T_k, \ldots, T_1)$. Then, the output of $\phi$ for $(\textbf{c},T)$ is $\phi(\textbf{c}, T) = (\pi, \T)$.
\end{itemize}

To obtain $\phi^{-1}$, we start with $(\pi = C_k\ldots C_1,\T = (T_k, \ldots ,T_1))$. Suppose $\pi\in K_\alpha$ for some $\alpha = (\alpha_1, \ldots, \alpha_k)\in \Com(d)$.  Define
\begin{footnotesize} \[\textbf{a}(\pi) = \Bigg(C_k(\alpha_1), C_k(\alpha_1 - 1),\ldots, C_k(1), C_{k-1}(\alpha_{2}), C_{k-1}(\alpha_2 - 1), \ldots, C_{k-1}(1), \ldots, C_1(\alpha_k), C_1(\alpha_k-1),\ldots, C_1(1)\Bigg).\] 
\end{footnotesize}
So $\textbf{a}(\pi) =(a_1, \ldots, a_d)$ are the entries where we read the cycles of $\pi$ from left to right while reading right to left within a cycle. For $\pi = (56)(247)(138)$, we have $\textbf{a}(\pi) = (6,5,7,4,2,8,3,1)$.

We define $S^{(0)} = \varnothing$ and $S^{(j)} \subset \{1, \ldots, d\}$ to be the set obtained by inserting $a_j$ in $S^{(j-1)}$. Define $T^{(0)} = \varnothing$ be the empty WBT.
\begin{enumerate}
\item Suppose $a_j \neq C_k(1)$ for any $k$. If $a_j$ is the $s$th smallest element inserted in $S^{(j)}$, then let $c_j = s$ and $T^{(j)}= T^{(j-1)}$.
\item Suppose $a_j = C_r(1)$ for some $1\leq r \leq k$. Obtain $T^{(j)}_{*}$ by inserting $T_r$ in $T^{(j-1)}$ such that $T^{(j)}$ is a marked WBT. We number the cells of $T^{(j)}_{*}$ in reading order as in the description of $\phi$, and if we denote the number in the marked cell by $s$, then $c_j = s$. Define $T^{(j)}$ by removing the marking from $T^{(j)}_{*}$.
\end{enumerate}
We define the output of $(\textbf{c},T)= \phi^{-1}(\pi,\T)$ such that $\textbf{c} = (c_d, c_{d-1}, \ldots, c_1)$ and $T = T^{(d)}$.

\end{proof}
\begin{example}\label{ex:H-P}
Let $d = 14$. We start with the choice sequence $\textbf{c} = (10,3,9,3,8,1,2,7,5,3,3,1,1,1)$ and the WBT
\boks{0.35}
\begin{small}
\[
T = \yt{333,333,11,22,22,1,1}
\]
\end{small}
We construct the output $(\pi,\T)$ under the map $\phi$ as described above. We assign the following numbering to the cells of $T$ and mark the cell numbered $c_{14}=10$ as that is the first element of $\textbf{c}$.
\boks{0.37}
\begin{small}
\[
\yt{123,456,78,9{*(lightgray)10},{11}{12},{13},{14}}
\]
\end{small}
\boks{0.42}
There is one identical bar in $T$ below the marked bar and so $T_1 = \yt{2{2\ast},22}$. As the newly created RBT contains 4 cells, we create a cycle $C_1$ of length 4 starting at the smallest available value, 1. We fill the rest of the entries as follows:
\begin{align*}
(1,\blk,\blk,\blk) &\text{ via } 1\acts \{\underline{1},2,3,4,5,6,7,8,9,10,11,12,13,14\}\\
(1,4,\blk,\blk) &\text{ via } (c_{13}= 3) \acts \{2,3,\underline{4},5,6,7,8,9,10,11,12,13,14\}\\
(1,4,11,\blk) &\text{ via } (c_{12}=9)\acts \{2,3,5,6,7,8,9,10,\underline{11},12,13,14\}\\
(1,4,11,5) &\text{ via } (c_{11} = 3)\acts \{2,3,\underline{5},6,7,8,9,10,12,13,14\}
\end{align*}
We remove $T_1$ from $T$ to obtain 
\[
T^{(2)} = \yt{333,333,11,1,1}
\] 
whose cells we number as follows and mark the cell numbered $c_{10}=8$. 
\begin{small}
\boks{0.35}
\[\yt{123,456,7{*(lightgray)8},9,{10}}
\]
\end{small}
\boks{0.35}
As there are no bars identical to the row containing the marked cell, we get $T_2 = \yt{1{1\ast}}$. So, we construct cycle $C_2$ of length 2 as follows:
\begin{align*}
(2,\blk) &\text{ via } 1 \acts \{\underline{2},3,6,7,8,9,10,12,13,14\}\\
(2,3) &\text{ via } (c_9 = 1) \acts \{\underline{3},6,7,8,9,10,12,13,14\}
\end{align*}
Now, we remove $T_3$ from $T^{(2)}$ to find $$T^{(3)} = \yt{333,333,1,1} $$ which has the following numbering of cells wherein we mark the cell numbered $c_8 = 2$.
\[
\yt{1{*(lightgray)2}3,456,7,8}
\]
This gives us $T_3 = \yt{3{3\ast}3,333}$ and we construct a cycle $C_3$ of length 6 as follows:
\begin{align*}
(6,\blk,\blk,\blk,\blk,\blk) &\text{ via } 1 \acts \{\underline{6},7,8,9,10,12,13,14\}\\
(6,14,\blk,\blk,\blk,\blk) &\text{ via } (c_7 = 7) \acts \{7,8,9,10,12,13,\underline{14}\}\\
(6,14,12,\blk,\blk,\blk) &\text{ via } (c_6 = 5) \acts \{7,8,9,10,\underline{12},13\}\\
(6,14,12,9,\blk,\blk) &\text{ via } (c_5 = 3) \acts \{7,8,\underline{9},10,13\}\\
(6,14,12,9,10,\blk) &\text{ via } (c_4 = 3) \acts \{7,8,\underline{10},13\}\\
(6,14,12,9,10,7) &\text{ via } (c_3 = 1) \acts \{\underline{7},8,13\}\\
\end{align*}
Finally, we are left with $T^{(1)} = \yt{1,1}$ with the numbering $\yt{{*(lightgray)1},2}$. We mark the cell corresponding to $c_2 = 1$ to obtain $T_4 = \yt{{1\ast},1}$. We create the cycle $C_4$ of length 2 as follows:
\begin{align*}
(8,\blk) &\text{ via } 1 \acts \{\underline{8},13\}\\
(8,13) &\text{ via } (c_1 = 1) \acts \{\underline{13}\}
\end{align*}
Thus, we have $\phi(\textbf{c},T) = (\pi,\T)$ with $\pi = C_4C_3C_2C_1 = (8,13)(6,14,12,9,10,7)(2,3)(1,4,11,5)$ and 
\[
\T =(T_4,T_3,T_2,T_1)=\yt{{1\ast},1}\quad \yt{3{3\ast}3,333} \quad \yt{1{1\ast}}\quad \yt{2{2\ast},22}
\]

We perform $\phi^{-1}$ on the above output. We start with $S^{(0)} = \varnothing$ and $T = \varnothing$, and
\[
\begin{array}{cccccccccccccc}
a_1 & a_2 & a_3 & a_4 & a_5 & a_6 & a_7 & a_8 & a_9 & a_{10} & a_{11} & a_{12} & a_{13} &a_{14}\\
13 & 8 & 7 & 10 & 9 & 12 & 14 & 6 & 3 & 2 & 5 & 11 & 4 & 1.
\end{array}
\]
As $a_1 = 13$, we have $c_1 = 1$, $T^{(1)} = \varnothing$ and $S^{(1)} = \{13\}$. Then we insert $a_2 = 8 = C_4(1)$ which gives us $S^{(2)} = \{8,13\}$. We update our WBT to get $T^{(2)} = \yt{1,1}$ with the marked cell in position 1 giving us $c_2 = 1$. We find
\begin{align*}
c_3 = 1 &\text{ as } S^{(3)} = \{\underline{7},8,13\}\\
c_4 = 3 &\text{ as } S^{(4)} = \{7,8,\underline{10},13\}\\
c_5 = 3 &\text{ as } S^{(5)} = \{7,8,\underline{9},10,13\}\\
c_6 = 5 &\text{ as } S^{(6)} = \{7,8,9,10,\underline{12},13\}\\
c_7 = 7 &\text{ as } S^{(7)} = \{7,8,9,{10},12,13,\underline{14}\}
\end{align*}
and $T^{(i)} = T^{(2)}$ for $3\leq i \leq 7$.
We have $a_8 = 6 = C_2(1)$. So, we update our WBT to get $T^{(8)}= \yt{333,333,1,1}$ and as the marked cell is in position 2, we get $c_8 = 2$. Also, $S^{(8)} = \{6,7,8,9,10,12,13\}$. We insert $a_9 = 3$ to get $c_9 = 1$ as $S^{(9)} = \{\underline{3},6,7,8,9,{10},12,13,14\}$. We still have $T^{(9)} = T^{(8)}$. For $a_{10} = 2$, we update our WBT to obtain $$T^{(10)} =\yt{333,333,11,1,1},$$ and as the marked cell ends up in position 8, we get $c_{10} = 8$. We have $S^{(10)} = \{2,3,6,7,8,9,{10},12,13,14\}$. For the rest of the choice sequence, we find
\begin{align*}
c_{11} = 3 &\text{ as } S^{(11)} = \{2,3,\underline{5},6,7,8,9,{10},12,13,14\}\\
c_{12} = 9 &\text{ as } S^{(12)} = \{2,3,5,6,7,8,9,{10},\underline{11},13,14\}\\
c_{13} = 3 &\text{ as } S^{(13)} = \ \{2,3,\underline{4},5,6,7,8,9,{10},11,13,14\}
\end{align*}
with $T^{(13)} = T^{(12)} = T^{(11)} = T^{(10)}$. We get $T^{(14)} = T$ where the marked cell is in position 10, so $c_{14}= 10$. We see that this recovers our input $(\textbf{c},T)$.
\end{example}

\subsection{$E$ in $P$}\label{ssec:E-P}
Before we prove the $P$-expansion of $E_d$, we introduce some more notation. Let $\delta$ be a square-free polycomposition. Let $(\pi,\T)\in K_\delta \times \prbt\ast(\delta)$ be such that $\pi = C_1\ldots C_k$ is in canonical notation and $\T = (T_1, \ldots, T_k)$. If $T_i$ has $r$ rows and $d$ columns, then we write the cycle $C_i$ of length $|T_i| = rd$ with vertical lines delimiting $r$ sets of $d$ consecutive elements. For instance, if $\T=  \yt{2{2\ast},22,22} \:\: \yt{33{3\ast},333}$ then we can associate $C_1 = (3,4|6,10|7,11)$ with $T_1$ and $C_2 = (1,5,2|12,9,8)$ with $T_2$.
\begin{prop}\label{prop:E-P}
For $d\geq 0$,  $\displaystyle{E_d = \sum_{\delta \in \PCom_{\sqf}(d)} (-1)^{\ell(\delta)}\dfrac{P_\delta}{Z_\delta}}$.
\end{prop}
\begin{proof}
Multiplying the equality on both sides by $d!$ and expressing in terms of monomials using Lemmas \ref{lem:btinterps} and \ref{lem:markedinterps} gives:
\[
\sum\limits_{T\in \sbt[d]} |\CS_d|\sgn(T) \x_T = \sum\limits_{\delta \in \PCom_{\sqf}(d)} |K_\delta| \sum\limits_{\T\in \prbt\ast(\delta)} (-1)^{\ell(\T)} \x_{\T}.
\]
We define an involution $\psi$ on the set $\bigcup_{\delta \in \PCom_{\sqf}(d)} K_\delta \times \prbt\ast(\delta)$. For $\pi = C_1\ldots C_k \in K_\delta$ and $\T = (T_1, \ldots, T_k)\in \prbt\ast(\delta)$, if $\psi(\pi, \T) = (\pi',\T')\neq (\pi,\T)$, then $\T'$ has one diagram more or less than $\T$. Furthermore, the fixed points under $\psi$ should produce signed monomials that correspond to the the monomials on the left hand side and each monomial $\x_T$ for $T\in \sbt[d]$ should appear with a multiplicity $d!$. The fixed points under $\psi$ are $(\pi, \T)$ where $\T = (T_1,\ldots, T_k)$ and all $T_i$s are single, distinct bars. We define a bijection of such fixed points to $\CS_d\times \sbt[d]$ by restricting $\phi$, as defined the proof of Proposition \ref{prop:H-P}. We note that when $\T$ is a PRBT that contains distinct bars, then $T$ in $ (\textbf{c}, T) = \phi^{-1}(\pi, \T)$ contains distinct bars, which makes $T$ an SBT. The number of tableaux in $\T$ is equal to the number of bars in $T$, that is, $\ell(\T) = \ell(T)$. This means $\sgn(T) = (-1)^{\ell(T)} = (-1)^{\ell(\T)}$ and this gives us the correct sign for the monomials appearing on the left hand side.

We now define $\psi(\pi,\T)$ for non-fixed points $(\pi,\T)$. By our previous discussion, $\T$ must contain a $T$ which has more than one row, or if every $T$ in $\T$ is a single bar, then there must exist two bars in $\T$ which are identical. 
Let $B$ in $\T$ be the last bar in the scanning order for which an identical bar exists. Suppose $B$ is contained in $T_i$ and $\dg(T_i)$ contains $r$ rows and $d$ columns. If no other $T$ in $\T$ contains a bar identical to $B$, then set $m = \infty$. On the other hand, if $T_h$ is the rightmost (not including $T_i$) marked RBT containing bars identical to $B$, then set $m = C_h(1)$, the minimum element of the cycle $C_h$. 

\begin{itemize}
\item Let $m\leq \infty$ and $\ell(T_i) > 1$. Suppose there exists a $jd +c$ for $1\leq j \leq r-1$ and $1 \leq c \leq d$ such that $C_i(jd+c)$ is the smallest entry larger than $C_i(1)$ and less than $m$, occurring after the first vertical line in $C_i$. Such a value $jd+c$ always exists when $m = \infty$. Define $T$ to be the marked RBT obtained from $T_i$ by removing all rows {strictly} below the $j$th row and let $T'$ be the marked RBT formed by the removed rows with the cell in the $c$th column marked. Split the cycle $C_i$ as $D = (C_i(1), \ldots, C_i(jd))$ and $\hat{D} = (C_i(jd+1),\ldots, C_i(jd+c),\ldots, C_i(dr))$. Define $D'$ to be the cyclic shift of $\hat{D}$ such that $C_i(jd+c)$ is the first entry of $D'$.
Let \[\pi' = C_1 \ldots C_{p-1} \, D' \, C_p \ldots C_{i-1}\, D\, C_{i+1} \ldots C_{k}\] be expressed in canonical notation for some position $p$. We construct \[\T' = (T_1, \ldots, T_{p-1}, T', T_p, \ldots, T_{i-1}, T, T_{i+1}, \ldots, T_{k}).\] 
In the following example, we have $i = 3$, $h=1$ and $d = 3$. We have $m = C_1(1) = 8$. We find that 4 is the smallest value between $C_3(1)=1$ and $m = 8$ occurring after the first vertical line in $C_3$. We have $j = 1$ and $c = 2$, as 4 occurs after the first ($j = 1$) vertical bar in the second position ($c = 2$).
\begin{align*}\T &=  \yt{3{3^*}3} \quad \yt{{1\ast}1,11}\quad  \yt{33{3^*},333,333}\\
\pi &= (8,16,10)(2,6|5,12)(1,13,3|15,4,14|9,11,7)
\end{align*}
We remove the second and third row of $T_3$ and also remove the last six entries of $C_3$. We place the rows and entries as follows after cyclically shifting such that 4 is the minimum element in the cycle. As $c = 2$, we mark the second cell in the top row of the PRBT in the second position.
\begin{align*}\T' &=  \yt{3{3^*}3} \quad \yt{3{3\ast}3,333}\quad \yt{{1\ast}1,11}\quad  \yt{33{3^*}}\\
\pi' &= (8,16,10)(4,14,9|11,7,15)(2,6|5,12)(1,13,3)
\end{align*}
In the following example, $m = \infty$, $i = 3$ and $d = 3$. We have $j = 1$ and $c = 2$.

\begin{align*}\T &=   \yt{{1\ast}1,11}\quad  \yt{33{3^*},333,333}\\
\pi &= (2,6|5,12)(1,13,3|10,4,8|9,11,7)
\end{align*}
\begin{align*}\T' &=  \yt{3{3\ast}3,333}\quad \yt{{1\ast}1,11}\quad  \yt{33{3^*}}\\
\pi' &= (4,8,9|11,7,10)(2,6|5,12)(1,13,3)
\end{align*}

\item Let $m <\infty$. Suppose there is no value of $1\leq j \leq r-1$ and $1 \leq c \leq d$ such that $C_i(jd+c)$ is the smallest entry larger than $C_i(1)$ and less than $m$, occurring after the first vertical line in $C_i$. This includes the case where $T_i$ is a single bar, that is $\ell(T_i)=1$. Let $T$ be the marked RBT formed by appending the bars of $T_h$ below $T_i$ and removing the marking from $T_h$ while preserving the location of the marked cell in $T_i$. Furthermore, define $\hat{C}_h$ to be the cyclic shift of  $C_h$ such that $C_h(1)$ is in the position equal to the index of the column of the marked cell in $T_h$. Let $D = (C_i(1), \ldots, C_i(dr), \hat{C}_h(1), \ldots \hat{C}_h(|T_h|))$ which is the result of concatenating the cycles $C_i$ and $\hat{C}_h$. We then define $\pi' = C_1\ldots C_{h-1} C_{h+1} \ldots C_{i-1} \,D\, C_{i+1} \ldots C_k$ and $\T' = (T_1, \ldots, T_{h-1}, T_{h+1}, \ldots T_{i-1}, T, T_{i+1}, \ldots, T_k)$. 

In the example,
\begin{align*}
\T &= \yt{1{1^*},11} \quad \yt{3{3^*}3,333}\quad \yt{33{3^*}}  \quad \yt{{4\ast}44}\\
\pi &= (6,12|11,9)(5,10,14|16,8,13)(2,4,7)(1,15,3)
\end{align*} 
we have $i = 3$ as $T_3 = \yt{33{3^*}}$ contains the last bar $B =\yt{33{3\ast}}$ in the scanning order such that there exist bars identical to $B$ in $\T$. We have $h = 2$ as $T_2$ is the rightmost marked RBT (not including $T_3$) which contains the bar identical to $B$. This gives us $m = C_2(1) = 5$ and $C_3(1) = 2$. Note that $\ell(C_3) = 3$, so for $C_3(jd+c)$ to exist, $j$ should be zero, and so we have no value of $jd + c$ with $j\geq 1$, $d = 3$, and $1\leq c\leq 3$. We place the bars of $T_2$ below $T_3$ and we cyclically shift $C_2$ to make 5 the second entry as the marked cell is in the second column, and concatenate it to $C_3$.
The output $\psi(\pi, \T) = (\pi', \T')$ is
\begin{align*}\T' &= \yt{1{1^*},11} \quad \yt{33{3^*},3{3}3,333}  \quad \yt{{4\ast}44}\\
\pi' &= (6,12|11,9)(2,4,7|13,5,10|14,16,8)(1,15,3)
\end{align*}
In the example, 
\begin{align*}\T &=  \yt{3{3^*}3}\quad \yt{33{3\ast}} \quad \yt{1{1^*},11} \quad \yt{3{3\ast}3,333}\\
\pi &= (9,15,10)(4,14,8)(2,6|16,13)(1,12,3|11,7,5)
\end{align*}
$i = 4$, $h = 2$ and $d = 3$. We observe that $T_4$ is not a bar and there also does not exist $j\geq 1$ satisfying $1 =C_4(1) < jd + c < C_2(1) = 4$. Note that $C_4(3)= 3$ lies between $1$ and $4$ but does not lie past the first vertical line.

The output $\psi(\pi,\T) = (\pi',\T')$ is found by placing the bars of $T_2$ under $T_4$ with the markings of $T_4$ removed, and cyclically shifting $C_2$ until 4 is in the third position and appending this cyclic shift to $C_4$. Thus we have 
\begin{align*}\T' &=  \yt{3{3^*}3} \quad \yt{1{1^*},11} \quad \yt{3{3\ast}3,333,333}\\
\pi' &= (9,15,10)(2,6|16,13)(1,12,3|11,7,5|14,8,4).
\end{align*}
\end{itemize}
\end{proof}


\section{Expansions of $E^+$ in $E$, $H$, and $P$}\label{sec:E+in}
In Section \ref{ssec:HU-rec}, we combinatorially prove a recursion between $H$ and $E^+$ described in \cite{polysymm}. In Section \ref{ssec:HUE-form}, we present and prove a formula that expresses $E^+$ in terms of $H$ and $E$. We then use the results of the previous sections to find the $H$, $E$, and $P$ expansions of $E^+$. We also provide explicit bijective proofs for all the expansions in Sections \ref{ssec:U-E}, \ref{ssec:U-H}, and \ref{ssec:U-P}.
\subsection{Recursion between $H$ and $E^+$}\label{ssec:HU-rec}
In \cite{polysymm}, the authors present a recursion involving $H$ and $E^+$ without proof. We provide a combinatorial proof of the recursion as it serves as a warm-up for the techniques used in the next section. The notation $d^r$ denotes a block with degree $d$ and multiplicity $r$ and is not to be confused with exponentiation. Recall $H_{d^r} = H_d(\x^r_{**})$. 
\begin{prop}[\cite{polysymm}, Remark 10]\label{prop:HU-rec}
For $d\geq 0$, $H_d = \sum\limits_{k = 0}^d H_{k^2} E^+_{d-2k}$ where $H_i =  E_i = 0$ for $i<0$.
\end{prop}
\begin{proof}
By using the results of Lemma \ref{lem:btinterps}, we must prove the following statement for monomials:
\[
    \sum\limits_{T\in \wbt[d]} \x_T = 
    \sum\limits_{k=0}^d \sum\limits_{U\in \wbt[k]}\sum\limits_{V \in \sbt[d-2k]} \x_U^2 \x_{V}
\]
We do so by producing a bijection from $\wbt[d]$ to $\bigcup_{k\geq 0}\wbt[k]\times \sbt[d-2k]$ defined by $T\mapsto (F(T), G(T))$ such that $\x_{T} = \x_{F(T)}^2\x_{G(T)}$. Scan down the rows of $T$ starting at the top row and make disjoint pairs consisting of consecutive identical bars. Define the bars in $F(T)$ to be one copy from each pair, and let $G(T)$ be the set of unpaired bars (of which there can be at most 1 each). To define the inverse bijection, we start with $(U,V)\in \wbt[k]\times \sbt[d-2k]$ and create a WBT $U'$ containing two copies of each bar in $U$, and we define $(F^{-1}\times G^{-1}) (U,V)$ to be the WBT constructed by interleaving the parts of $U'$ and $V$ such that going from top to bottom, the bars weakly decrease in size and the labels increase weakly within bars of the same size.
\end{proof}
\begin{example}
In the example 
\boks{0.35}
    \[
    T = \yt{3333,3333,3333,222,11,11,11,11} \xmapsto{F\times G} \left( \yt{3333,11,11}, \yt{3333,222}\right),
    \]
we pair the first two bars of $T$ and obtain the first row of $F(T)$ by inserting one copy of $\yt{3333}$. The third row does not have a matching bar and thus becomes the top row of $G(T)$. The fourth bar does not have a matching bar either and becomes the second row of $G(T)$. The fifth and sixth bars form a pair, and so do the seventh and the eighth bars. One copy from each of the pairs form the second and third row of $F(T)$ respectively. The monomial associated with both $T$ and $(F(T),G(T))$ is $\x_T = x_{43}^3 x_{32}x_{21}^4 = (x_{43}x_{21}^2)^2 x_{43}x_{32} = \x_{F(T)}^2\x_{G(T)}$.
\end{example}
\subsection{Expansions in $E^+$ using a formula}\label{ssec:HUE-form}
In the spirit of the recursion just proved, we prove the following proposition bijectively.
\begin{prop}
For $d\geq 0$, $E_d^+ = \sum\limits_{k=0}^d H_{d-2k}E_{k^2}$ where $H_i, E_i = 0$ for $i<0$.
\end{prop}
\begin{proof}
By using the results of Lemma \ref{lem:btinterps} and recalling $E_{k^2} = E(\x_{**}^2)$, we reinterpret the formula as an identity on monomials:
\[
\sum\limits_{T\in \sbt[d]} \x_T = \sum\limits_{k=0}^d \sum\limits_{\substack{U\in \wbt[d-2k]\\ V\in \sbt[k]}} \sgn(V)\x_U \x_V^2.
\]
We define an involution $F'$ on $\bigcup_{k=0}^d \wbt[d-2k]\times \sbt[k]$ mapping $(U,V)$ to $(U',V')$ which preserves the monomial weight but negates the sign, giving us $\sgn(V')\x_{U'}\x_{V'}^2 =  -\sgn(V)\x_{U}\x_{V}^2$ for non-fixed points $(U,V)$. The only fixed points of $F'$ are $(U_0,V_0)$ where $U_0$ does not contain any repeated bars and $V_0 = \varnothing$. We can map these fixed points to the set of SBTs of size $d$ by $(U_0,V_0) \mapsto U_0$. As $\ell(V_0) = 0$, the monomials $\x_{U_0}\x_{V_0}^2$ corresponding to the fixed points appear with the sign $\sgn(\varnothing) = 1$. 
Now we define the action of $F'$ on non-fixed points $(U,V)$.

\begin{enumerate}
\item Suppose $V$ is non-empty and $U$ does not contain any repeated bars. Then define $U'$ by inserting two copies of the top row of $V$ in $U$ such that the insertion makes $U'$ a WBT. Let $V'$ be obtained from $V$ by removing the top row of $V$.

 \item Let $U$ contain repeated bars. Start by scanning down the rows of $U$, and let $B$ be the first bar such that the bar $C$ immediately below it is identical to it. Let $A$ be the top bar of $V$ if $V\neq \varnothing$. 
\begin{enumerate}
\item Suppose (i) $V= \varnothing$ or (ii) $V\neq \varnothing$, and $B$ has a larger length than $A$, or $B$ has the same length but a strictly smaller label than $A$. Obtain $U'$ by removing $B$ and $C$ from $U$, and obtain $V'$ by inserting $B$ above the top row of $V$. We see that $\ell(V') = \ell(V)+1$ and thus $\sgn(V') = -\sgn(V)$.
\item If $B$ and $A$ do not satisfy the conditions in (a), then obtain $V'$ by removing $A$ from $V$ and obtain $U'$ from $U$, by inserting two copies of $A$ in $U$ which must insert above $B$.
\end{enumerate}
\end{enumerate}

\end{proof}
\begin{example}
In the following example, the first bar in $U$ with an identical bar below it is $\yt{111}$ in the second row. We see that it has the same length as the top row of $V$ but a smaller label. Thus we remove these two bars from $U$ and insert one copy as the top row in $V'$:
    \[
    \left(U = \yt{4444,111,111,22,22}, V = \yt{222,333} \right) \xmapsto{F'}  \left(U'= \yt{4444,22,22}, V'=\yt{111,222,333} \right) 
    \]
Apply $F'$ to the just computed output $(U',V')$. In $U'$, the bar $\yt{22}$ in the second row is the first bar in an identical pair, but it has a smaller length than the top row of $V'$. So we remove the top row of $V'$ and insert two copies of it in $U'$, which gives us $(U,V)$.
\end{example}

To present the $H$, $E$, and $P$ expansions of $E^+$, we define some special subsets of polycompositions.
Let $\PCom_P(d)$ be the set of polycompositions of $d$ of the form $\alpha^1\beta^2$ where $\alpha$ or $\beta$ are (possibly empty) compositions. Define $\PCom_E(d)$ to be the set of polycompositions of $d$ of the form $\alpha^1(b)^2$ where $\alpha$ is a (possibly empty) composition and $b$ is a non-negative integer. Define $\PCom_H(d)$ to be the set of polycompositions of $d$ of the form $(a)^1\beta^2$ where $\beta$ is a (possibly empty) composition and $a$ is a non-negative integer.
\begin{prop}\label{prop:E+-HEP}
For $d\geq 0$,
    \begin{enumerate}
        \item  $E^+_d = \sum\limits_{\delta = \alpha^1(b)^2 \in \PCom_E(d)} (-1)^{\ell(\alpha)} E_{\delta},$
        \item $   E^+_d = \sum\limits_{\delta = (a)^1\beta^2 \in \PCom_H(d)} (-1)^{\ell(\beta)} H_{\delta},$ and 
           \item $E^+_d  = \sum\limits_{\delta = \alpha^1 \beta^2 \in \PCom_P(d)} (-1)^{\ell(\beta)} \dfrac{1}{Z_\alpha Z_\beta}P_{\delta}$
    \end{enumerate}
\end{prop}
\begin{proof}
Recall from Proposition \ref{prop:H-E,E-H} that $H_n = \sum_{\alpha^1\in \PCom_{\sqf}(n)} (-1)^{\ell(\alpha)}E_{\alpha^1}$. We substitute this into $E_d^+ = \sum\limits_{k=0}^d H_{d-2k}E_{k^2}$ to obtain
\[
E_d^+ = \sum\limits_{b=0}^d \sum\limits_{\alpha^1 \in \PCom_{\sqf}(d-2b)} (-1)^{\ell(\alpha)}E_{\alpha^1} E_{b^2}
\]
which proves (1). Similarly, using $E_n = \sum_{\beta^1\in \PCom_{\sqf}(n)} (-1)^{\ell(\beta)}H_{\beta^1}$ gives
\[
E_d^+ = \sum\limits_{k=0}^d \sum\limits_{\beta^1\in \PCom_{\sqf}(k)} (-1)^{\ell(\beta)}H_{d-2k} E_{\beta^2}
\]
which proves (2). The expansions $H_n = \sum\limits_{\delta \in \PCom_{\sqf}(n)} \dfrac{P_\delta}{Z_\delta}$ and
$E_n = \sum\limits_{\delta \in \PCom_{\sqf}(n)} (-1)^{\ell(\delta)}\dfrac{P_\delta}{Z_\delta}$ from Propositions \ref{prop:H-P} and \ref{prop:E-P} give us (3), noting that $E_{n^2} =  \sum\limits_{\delta \in \PCom_{\sqf}(n)} (-1)^{\ell(\delta)}\dfrac{P_{\delta^2}}{Z_\delta}$.
\end{proof}

\subsection{$E$-expansion of $E^+$}\label{ssec:U-E}
We present an involution proof of the $E$-expansion of $E^+$. Recall that bar tableaux appear in layers indexed by natural numbers and we omit empty layers. For $\T\in \wbt[d]$, if a WBT $T$ occurs in layer $r$ of $\T$, then $T$ contributes the factor $\x_T^r$ to $\x_{\T}$. Unlike the expansions in the previous section, we now have a new layer to work with and the main idea is that each bar in layer 2 counts twice. Our involutions involve transferring two copies of a bar from layer 2 to layer 1 and one copy of a bar in an identical pair of bars (when they exist) from layer 1 to layer 2. We show \[E^+_d = \sum_{\delta =\alpha^1 (b)^2\in \PCom_{E}(d)} (-1)^{\ell(\alpha)} E_{\delta}.\]

\begin{proof}[Involution proof of Proposition \ref{prop:E+-HEP} (1)]
Using the results from Lemmas \ref{lem:btinterps} and \ref{lem:pbtinterps}, we can rewrite the statement in terms of monomials as
\[
\sum\limits_{T\in \sbt[d]} \x_T = \sum\limits_{\delta\in \PCom_E(d)} \sum\limits_{\T\in \psbt(\delta)} (-1)^{k(\T)} \psgn(\T) \x_{\T}. 
\]
where $k(\T)$ is the number of non-empty SBTs in layer 1 of $\T$.  
We denote an element $\T\in \psbt(\delta)$ for $\delta\in \PCom_E(d)$ by $\T =(T_{11}, \ldots, T_{1k}, T_2)$ where (possibly empty) $T_{1i}$ appear in layer 1 and $T_2$ is the sole (possibly empty) SBT in layer 2.
We define an involution $\sigma:= \sigma^{E^+}_E$ on the set $ \bigcup_{\delta \in \PCom_E(d)} \psbt(\delta)$ as follows:
\begin{enumerate}
\item If $d=0$, then the right hand sum is over $\{\varnothing\}$. We have $k(\varnothing) = 0$, $\psgn(\varnothing) = 1$ and $\x_{\varnothing} = 1$, which matches the left hand side monomial $\x_{\varnothing} = 1$ where $\varnothing$ is the empty SBT.
\item Let $\psi(\T)$ be the output obtained by applying the {strict} stack-or-slash operation (cf. Section \ref{ssec:H-E}) on layer 1 which changes the number of diagrams by 1 but preserves the number of bars. For $\T$ not fixed under $\psi$, define $\sigma(\T) = \psi(\T)$.  We have the relation \[
(-1)^{k(\sigma(\T))} \psgn(\sigma(\T))\x_{\sigma(\T)} = - (-1)^{k(\T)} \psgn(\T)\x_{\T}.\]
\item If $\psi(\T)=\T$, then all $T_{1,i}$ in layer 1 must be bars with lengths weakly increasing with $i$ and weakly decreasing labels between bars of the same length.
We first consider the case where there exists a rightmost identical pair $(T_{1,i},T_{1,i+1})$ in layer 1. 
\begin{enumerate}
\item If $T_{1,i}$ has a larger length than the top row of $T_2$, or if it has the same length but a strictly smaller label, then obtain $\sigma(\T)$ by removing $T_{1,i}$ and $T_{1,i+1}$ from layer 1 and inserting a bar identical to $T_{1,i}$ above the top row of $T_2$. This condition also covers the case when layer 1 is non-empty but layer 2 is empty as the top row of $T_2$ has size zero.
\item If $T_{1,i}$ has a smaller length than the top row of $T_2$, or has the same length but a weakly larger label, then obtain $\sigma(\T)$ by removing the top row from $T_2$ and inserting two copies $(T,U)$ of this top row in layer 1 in a unique position such that it preserves the weakly increasing length and weakly decreasing labels between bars of the same length condition. If there exists a bar identical to $T$ in layer 1, then insert $T$ to the right of such bar. To see that $(T,U)$ is the rightmost identical pair in layer 1 of $\sigma(T)$, we compare $T$ with $T_{1,i}$. According to the assumption $T_{i,1}$ either has a smaller length than $T$, or has the same length but a larger label. In both these cases, $T_{1,i}$ lies to the left of $T$ and as $T_{i,1}$ was part of the rightmost identical pair in $\T$, $(T,U)$ is the rightmost identical pair in $\sigma(\T)$.
\end{enumerate}
In both the above cases, the number of diagrams in layer 1 changes by 2, which means $k(\sigma(\T)) = k(\T) \pm 2$ but the number of bars changes by 1, so $\psgn(\sigma(\T)) = - \psgn(\T)$. This shows that $(-1)^{k(\sigma(\T))} \psgn(\sigma(\T))\x_{\sigma(\T)} = - (-1)^{k(\T)} \psgn(\T)\x_{\T}$.
\item Now, suppose $\psi(\T) = \T$ and layer 1 does not contain an identical pair. In other words, all bars in layer 1 are distinct. 
\begin{enumerate}
\item If $T_2\neq \varnothing$, then obtain $\sigma(\T)$ by removing the top row of $T_2$ and inserting two copies of this top row in layer 1 while preserving the weakly increasing length and weakly decreasing labels between bars of the same length condition.
\item If $T_2=\varnothing$, then $\sigma(\T) = \T$. We can associate with such a fixed point the SBT $T$ where the $i$th row from the top of $T$ is the bar $T_{1,k+1-i}$. Notice that the number of diagrams in layer 1, namely $k(\T)$, is equal to the number of bars $\ell(\T)$ and so $(-1)^{k(\T)+\ell(\T)} = 1$ which shows that each monomial $\x_{\T} = \x_T$ for a fixed point $\T$ appears with coefficient 1, which agrees with the coefficient on the left hand side.
\end{enumerate}
\end{enumerate}
 
\end{proof}
\begin{example}
In the following example, we have a PSBT of shape $(1,1,1,2,2,2,2,3)^1(4)^2$ which contains $(\yt{11},\yt{11})$ as its rightmost identical pair in layer 1:
\begin{align*}
    1 \quad &\yt{3} \quad \yt{3} \quad \yt{1} \quad \yt{11}\quad \yt{11}\quad \yt{11} \quad \yt{11} \quad \yt{222} \\
    2\quad  & \yt{22,1,2}
\end{align*}
Under the action of $\sigma^{E^+}_E$, we remove the rightmost pair $(\yt{11},\yt{11})$ and insert $\yt{11}$ as the top row of the tableau in layer 2 as it has the same length as $\yt{22}$ but a smaller label. This gives us the output
\begin{align*}
1 \quad &\yt{3}\quad \yt{3} \quad \yt{1} \quad \yt{11}\quad \yt{11}\quad \yt{222} \\
2\quad & \yt{11,22,1,2}
 \end{align*}
\end{example}

\subsection{$H$-expansion of $E^+$}\label{ssec:U-H}

We prove by involution the result \[E^+_d = \sum_{\delta =(a)^1 \beta^2 \in \PCom_H(d)} (-1)^{\ell(\beta)} H_{\delta}.\]

\begin{proof}[Involution proof of Proposition \ref{prop:E+-HEP} (2)]
Using the results from Lemmas \ref{lem:btinterps} and \ref{lem:pbtinterps}, we can rewrite the statement in terms of monomials as
\[
\sum\limits_{T\in \sbt[d]} \x_T = \sum\limits_{\delta \in \PCom_H(d)} \sum\limits_{\T\in \pwbt(\delta)} (-1)^{k'(\T)} \x_{\T}. 
\] 
where $k'(\T)$ is the number of diagrams in layer 2 of $\T$.
We denote the PWBT $\T = (T_1, T_{2,1}, \ldots, T_{2,k})$, where $T_1$ appears in layer 1 and $T_{2,i}$ appears in layer 2 for $1\leq i\leq k$.
We recall the bijection $(F,G):\wbt[n]\to \bigcup_{k\geq 0}\wbt[k]\times \sbt[n-2k]$ from the proof of Proposition \ref{prop:HU-rec} which takes $T$ as an input and outputs $(F(T), G(T))$ briefly: find non-overlapping identical pairs $(B,B')$ in $T$ and for each such pair, create a bar identical to $B$ in $F(T)$. The bars in $T$ which are not a part of any identical pair form $G(T)$. We define an involution $\sigma = \sigma^{E^+}_E$ on $\bigcup_{\delta \in \PCom_H(d)}\pwbt(\delta)$ as follows:
\begin{itemize}
\item If $T_1$ contains an identical pair, then obtain $\sigma(\T)$ from $\T$ by removing $T_1$ from layer 1, replacing it with the SBT $G(T_1)$ and placing $F(T_1)$ as the new leftmost WBT in layer 2.
\item If $T_1$ does not contain an identical pair and layer 2 is non-empty, then define $T$ to be the unique WBT such that $G(T) = T_1$ and $F(T) = T_{2,1}$. Then obtain $\sigma(\T)$ from $\T$ by removing $T_1$ from layer 1, removing $T_{2,1}$ from layer 2, and placing $T$ in layer 1.
\end{itemize}
In both the above cases, the number of diagrams in layer 2 changes by 1, which means we have the relation $(-1)^{k'(\T)}\x_{\T} = - (-1)^{k'(\sigma(\T))}\x_{\sigma(\T)}$. The fixed points under $\sigma$ are those $\T$ for which $T_1$ does not contain an identical pair and layer 2 is empty. For such fixed $\T$, we have the unique map $\T \to T_1$ and $T_1$ is an SBT of size $d$. As $k'(\T) = 0$, we obtain monomials corresponding to the fixed points matching the monomials on the left.
\end{proof}
\boks{0.38}
\begin{example}
For $\T$ shown below, we have two identical pairs consisting of $\yt{22}$ and $\yt{2}$ in $T_1$. 
 \begin{align*}
    1\quad & T_1 = \yt{22,22,33,2,2}\\
    2 \quad& T_{2,1}=\yt{333} \quad T_{2,2}=\yt{111,77} \quad T_{2,3}=\yt{1,1,1}.
\end{align*}
So, we find $F(T_1) = \yt{22,2}$ and $G(T) = \yt{33}$, which gives us $\sigma(\T)$ shown here:
    \begin{align*}
    1\quad & \yt{33}\\
    2 \quad& \yt{22,2}\quad\yt{333} \quad \yt{111,77} \quad \yt{1,1,1}.
\end{align*}
\end{example}

\subsection{$P$-expansion of $E^+$}\label{ssec:U-P}

For a composition $\alpha$ and set $S$ with $|\alpha|$ elements, define $K_\alpha^{S}$ to be the set of permutations $\pi$ in $K_\alpha$ where the occurrence of $i$ in any cycle of $\pi$ is replaced with the $i$th smallest element of $S$. For instance, if $\alpha = (3,2,3)$, $\pi = (4,7,5)(2,6)(1,8,3)$ and $S = \{3, 4, 6, 8, 11, 12, 14, 50\}$, then the corresponding element in $K_\alpha^S$ is $(8,14,11)(4,12)(3,50,6)$. Similarly, for the symmetric group $\S_n$ and any $n$-element subset $T$, define $\S_n^T$ to be the set of permutations where we replace the occurrence of $i$ in each cycle with the $i$th smallest element of $T$. If $T$ is a marked RBT with $d$ columns and $r$ rows and $C$ is a cycle in $\S_{dr}$, then we associate $r$ sequences of length $d$ with bars in $T$ as follows: to the top row of $T$, associate $(C(1), \ldots C(d))$, to the second row $(C(d+1), \ldots, C(2d))$, and so on, associating to the $r$th row $(C((r-1)d + 1),\ldots C(rd))$. We denote the sequence associated in this manner with $B$ by $s_B = s_B(C)$. We omit the mention of the corresponding cycle in this notation as that would be clear from context. Call the first element of the sequence associated to $B$, the \textit{cyc-index of $B$}, and define the cyc-index of a marked RBT to be the cyc-index of its topmost bar. 

We present a proof of the expansion \[E^+_d  = \sum_{\delta = \alpha^1 (\beta)^2 \in \PCom_P(d)} (-1)^{\ell(\beta)} \frac{1}{Z_\alpha Z_\beta}P_{\delta}.\]
\begin{proof}[Involution proof of Proposition \ref{prop:E+-HEP} (3)]
By using the results from lemmas \ref{lem:btinterps} and \ref{lem:pbtinterps}, we have to prove the following equality involving monomials:
\[
\sum\limits_{T\in \sbt[d]} \x_T = \sum\limits_{\substack{\alpha,\beta\in \Com\\ |\alpha|+2|\beta| = d}} (-1)^{\ell(\beta)}\dfrac{1}{Z_\alpha Z_\beta} \sum\limits_{\T\in \prbt\ast(\alpha^1\beta^2)} \x_T.
\]
By performing some multiplications and divisions, we rewrite the equation above as
\[
\sum\limits_{T\in \sbt[d]} d!\cdot \x_T = \sum\limits_{\substack{\alpha,\beta\in \Com\\ |\alpha|+2|\beta| = d}} (-1)^{\ell(\beta)}\dfrac{d!}{|\alpha|!\cdot  |\beta|! \cdot |\beta|!} \cdot  \dfrac{|\alpha|!}{Z_\alpha} \cdot \dfrac{|\beta|!}{ Z_\beta} \cdot |\beta|! \sum\limits_{\T\in \prbt\ast(\alpha^1\beta^2)} \x_T \: (*).
\]
Let $(S_\alpha$, $S_\beta$, $S_\gamma)$ be an ordered set partition of $\{1, \ldots, d\}$ which means $S_\alpha\cup S_\beta \cup S_\gamma = \{1, \ldots, d\}$ and $S_\alpha, S_\beta, S_\gamma$ are pairwise disjoint. For $a, b, d\geq 0$ with $a + 2b = d$, denote by $\mathcal{S}(a,b)$ the set of ordered set partitions $(S_\alpha$, $S_\beta$, $S_\gamma)$ of $\{1, \ldots, d\}$ satisfying $|S_\alpha| = a$, $|S_\beta| = |S_\gamma| = b$. The choice of such sets is given by the multinomial coefficient $\frac{d!}{|\alpha|!\cdot  |\beta|! \cdot |\beta|!}$ appearing in (*). The monomials on the right hand side of (*) arise from the set,
\[
\] 
For a fixed $\alpha,\beta$, let $a = |\alpha|$ and $b = |\beta|$, then write any $(\pi,\rho,\sigma, (\T,\U)) \in K_\alpha^{S_\alpha} \times K_\beta^{S_\beta} \times \S_{|\beta|}^{S_\gamma} \times \prbt\ast(\alpha^1\beta^2)$ explicitly as:
\[
\Bigg( C_1\ldots C_a,D_1\ldots D_b,\sigma(1)\ldots \sigma(b), (T_1, \ldots, T_a, U_1, \ldots, U_b)\Bigg) \qquad (**)
\]
where $T_i$ for $1\leq i\leq k$ are marked RBTs in layer 1 while $U_i$ for $1\leq i\leq l$ are marked RBTs in layer 2. Here $\pi = C_1\ldots C_a$ and $\rho = D_1\ldots D_b$ are in canonical notation. We write the permutation $\sigma$ in one-line notation where $\sigma(i)$ is the denotes the letter $i$ maps to under $\sigma$. In the following example, we choose $a = 14$ and $b = 7$ with $\alpha = (6,2,6)$ and $\beta = (4,3)$. We choose $S_\alpha = \{4,5,6,8,11,15,17,18,19,21,22,24,25,28\}$, $S_\beta = \{1,7,10,14,16,20,27\}$ and $S_\gamma = \{2,3,9,12,13,23,26\}$.
 We visualize the object (**) as follows, keeping in mind that $\T$ is in layer 1 and $\U$ is in layer 2:
\begin{footnotesize}
\[\left.
\begin{array}{cccc}
\let\quad\thinspace
\T= \yt{3{3^*}3,333} & \yt{{4\ast},4} &\yt{{2\ast}2,22,22}\\[0.5cm]
\pi = (8,18,28|25,22,11) &(6,21) &(4,15|17,24|5,19)
\end{array}\right| 
\left.\begin{array}{cc}
\let\quad\thinspace
\U= \yt{{2\ast},2,2} & \yt{3{3\ast},33} \\[0.5cm]
\rho = (7|16|10) &(1,27|20,14)
\end{array}\right|
\sigma = 9,2,13,23,26,12,3
\]
\end{footnotesize}

We define a sign-reversing involution $\Psi$ on the set $\mathcal{S}[d]$. Denote the output $\Psi(\pi,\rho,\sigma, (\T,\U))$ by $ (\pi',\rho',\sigma',(\T',\U'))$.
\begin{enumerate}
\item If layer 2 contains at least two identical bars, then define $\pi' = \pi$, $\sigma' = \sigma$, and $\T' = \T$. On the other hand, let $(\rho',\U') = \psi(\rho, \U)$ where $\psi$ is the involution defined in the proof of Proposition \ref{prop:E-P}. The output of the above example is:
\begin{footnotesize}
\[\left.
\begin{array}{cccc}
\let\quad\thinspace
\T'= \yt{3{3^*}3,333} & \yt{{4\ast},4} &\yt{{2\ast}2,22,22}\\[0.5cm]
\pi' = (8,18,28|25,22,11) &(6,21) &(4,15|17,24|5,19)
\end{array}\right| 
\left.\begin{array}{ccc}
\let\quad\thinspace
\U'= \yt{3{3\ast}} & \yt{{2\ast},2,2} & \yt{3{3\ast}} \\[0.5cm]
\rho' =(14,20) & (7|16|10) &(1,27)
\end{array}\right|
\sigma' = 9,2,13,23,26,12,3
\]
\end{footnotesize}
\item Now, suppose that $\U$ does not contain any identical bars. These are the fixed points under $\psi$ and consist of $(\pi,\rho, \sigma,(\T,\U))$ where each $U_i\in \U$ is a bar and all $U_i$ are distinct. 
\begin{enumerate}
\item Suppose there does not exist a bar in layer 1 with a matching bar. This means $\T$ consists of marked RBTs which are all single bars and are distinct.
\begin{enumerate}
\item Suppose $\U = \varnothing$. The restriction of the bijection $\phi^{-1}$ (as defined in the proof of (1) in Theorem \ref{thm:H-P,E-P}) in the case with unique bars gives us exactly the set $\CS_d\times \sbt[d]$. This gives us the desired monomials on the left hand side of (*).
\item If $\U$ is not empty, then let $A:= U_1$ be the leftmost bar in $\U$.  Define $\U'$ by removing $A$ from $\U$ and let $\rho'$ be obtained from $\rho$ by removing $D_1$. Define $s = (\sigma(1), \ldots, \sigma(|A|))$ and let $\sigma'$ be obtained by removing the entries of $s$ from $\sigma$. 
\begin{enumerate}
\item Suppose the cyc-index of $A$ is smaller than all entries in $s$. Define $A'$ to be the RBT consisting of two copies of $A$ with the marked cell in the same column as $A$. Let $D'$ be obtained by appending $s$ to $D_1$ on the right. Let $\pi'$ be the permutation obtained by rearranging $C_1\ldots C_a D'$ in canonical notation. Let $\T'$ be obtained from $\T$ after rearranging the RBTs of $\T$ according to $\pi'$ and inserting $A'$ in the position corresponding to the cycle $D'$. In the following example, $s = (11,9)$ all of whose entries are larger than 6, the cyc-index of $A = \yt{{2\ast}2}$. We obtain $D' = (6,12|11,9)$.

\begin{minipage}{0.8\textwidth}
\begin{footnotesize}
\[\left.
\begin{array}{cccc}
\let\quad\thinspace
\T= \yt{{1\ast}1} & \yt{3{3\ast}33} &\yt{{3\ast}}\\[0.5cm]
\pi = (7,13)&(4,10,5,8)&(1)
\end{array}\right| 
\left.\begin{array}{cc}
\let\quad\thinspace
\U= \yt{{2\ast}2} & \yt{{1\ast}}\\[0.5cm]
\rho = (6,12)&(2)
\end{array}\right|
\sigma = (11,9,3)
\]
\end{footnotesize}
\end{minipage}

The output obtained is

\begin{minipage}{0.8\textwidth}
\begin{footnotesize}
\[\left.
\begin{array}{cccc}
\let\quad\thinspace
\T'= \yt{{2\ast}2} & \yt{{2\ast}2,22} & \yt{3{3\ast}33} &\yt{{3\ast}}\\[0.5cm]
\pi' = (7,13)& (6,12|11,9)&(4,10,5,8)&(1)
\end{array}\right| 
\left.\begin{array}{cc}
\let\quad\thinspace
\U'= \ \yt{{1\ast}}\\[0.5cm]
\rho' = (2)
\end{array}\right|
\sigma' = (3)
\]
\end{footnotesize}
\end{minipage}

\item Suppose the cyc-index of $A$ is larger some entry in $s$. Let $A'$ be a bar identical to $A$ with marking in the same column and let $A''$ be a bar identical to $A$ with a marked cell in the same position as the minimum element of $s$. Let $\hat{s}$ be the cyclic shift of $s$ such that the minimum element is the first element. Let $\pi'$ be obtained by expressing $C_1\ldots C_a D_1  \hat{s}$ in canonical notation. Obtain $\T'$ by rearranging RBTs of $\T$ according to $\pi'$ and inserting $A'$ in the position corresponding to $D_1$ and $A''$ in the position corresponding to $\hat{s}$. In the following example $s = (6,3,11)$ and the cyc-index of $A = \yt{{2\ast}22}$ is 4. 

\begin{minipage}{0.8\textwidth}
\begin{footnotesize}
\[\left.
\begin{array}{cccc}
\let\quad\thinspace
\T= \yt{{1\ast}1} & \yt{3{3\ast}} \\[0.5cm]
\pi = (5,8)&(1,10)
\end{array}\right| 
\left.\begin{array}{cc}
\let\quad\thinspace
\U= \yt{{2\ast}22} & \yt{{1\ast}}\\[0.5cm]
\rho = (4,12,9)&(2)
\end{array}\right|
\sigma = (6,11,3,7)
\]
\end{footnotesize}
\end{minipage}

We obtain the output:

\begin{minipage}{0.8\textwidth}
\begin{footnotesize}
\[\left.
\begin{array}{cccc}
\let\quad\thinspace
\T'= \yt{{1\ast}1} & \yt{{2\ast}22} & \yt{22{2\ast}} & \yt{3{3\ast}} \\[0.5cm]
\pi' = (5,8)&  (4,12,9)& (3,6,11) &(1,10)
\end{array}\right| 
\left.\begin{array}{cc}
\let\quad\thinspace
\U'=  \yt{{1\ast}}\\[0.5cm]
\rho' =(2)
\end{array}\right|
\sigma' = (7)
\]
\end{footnotesize}
\end{minipage}

\end{enumerate}

\end{enumerate}
\item Suppose there exists a bar in layer 1 of $\T$ with a bar identical to it in layer 1 but not in layer 2. Let $B$ the first such bar in scanning order. Suppose $B$ is contained in $T_i$ and the first identical bar $B'$ after $B$ in the scanning order is contained in $T_j$ which may or may not be equal to $T_i$. If $T_i\neq T_j$, then it must be that $T_i = B$. In the following example, $a = 11$ and $b = 3$ with $\alpha = (2,2,3,4)$ and $\beta = (1,2)$. We find that $T_2 = B= \yt{3{3\ast}}$ is the first bar in the scanning order containing an identical bar in layer 1 but not layer 2. This makes $B'$ the top bar in $T_4$. Notice that $\yt{{2\ast}}$ contains an identical bar in layer 1 {and} in layer 2, and is thus not chosen as $B$.

\begin{minipage}{0.8\textwidth}
\begin{footnotesize}
\[\left.
\begin{array}{cccc}
\let\quad\thinspace
\T= \yt{{2\ast},2} & \yt{3{3\ast}} &\yt{4{4\ast}4} &\yt{{3\ast}3,33}\\[0.5cm]
\pi = (9|14)&(8,10)&(5,15,6)&(4,17|16,11)
\end{array}\right| 
\left.\begin{array}{cc}
\let\quad\thinspace
\U= \yt{{2\ast}} & \yt{{1\ast}1}\\[0.5cm]
\rho = (7)&(2,12)
\end{array}\right|
\sigma = (13,1,3)
\]
\end{footnotesize}
\end{minipage}
\begin{enumerate}
\item Suppose $U=\varnothing$ or if $U\neq \varnothing$ then the cyc-index of $B$ is larger than the cyc-index of $U_1$ which is $D_1(1)$. Obtain $\rho'$ by adding $s_B$ as the leftmost cycle to $\rho$ and $\U'$ by inserting a bar identical to $B$ with the marked cell in the position as the leftmost bar in layer 2. 
\begin{enumerate}
\item Suppose $T_i = T_j$. Obtain the permutation $C_1'\ldots C_k'$ from $C_i$ as follows: remove the entries of $s_B$ and $s_{B'}$ from $C_i$ to obtain $\hat{C}_1$ and divide $\hat{C}_1$ into $|T_i|/|B| - 2$ equal groups of size $|B|$ from left to right. If the minimum entry in $\hat{C}_1$ is in position greater than  $m|B|$ and smaller than $(m+1)|B|$ for $m\geq 0$, then let $\hat{C}_2$ be the first $m|B|$ entries of $\hat{C}_1$ and let $C_1'$ be the rest of the entries. Repeat this process on $\hat{C}_2$ obtain $C_2'$ and so on. 
Let $T_l'$ for $1\leq l\leq k$ be an RBT with $\ell(C_l')/|B|$ bars identical to $B$ where the top bar has a marked cell in the column corresponding to the minimum element of $C_l'$. Obtain $\pi'$ from $\pi$ by expressing $C_1\ldots C_{i-1} C_1'\ldots C_k'\ldots C_{i+1} \ldots C_a$ in canonical notation. Obtain $\T'$ from $\T$ by removing $T_i$, and reordering the RBTs $T_1, \ldots, T_{i-1}, T_1',\ldots, T_k', T_{i+1}, \ldots, T_a$ in accordance with $\pi'$.
Define $\hat{s}$ as the cyclic shift of $s_{B'}$ which results in the minimum entry of $\hat{s}$ being in the same position as the marked cell of $B'$. Define $\sigma' =  \hat{s} \sigma$ in one-line notation. In the following example, $B = \yt{3{3\ast}3}$ is the top row of $T_2$ and has a cyc-index 3 which is larger than the cyc-index, 2, of $U_1$. Here $s_B = (3,14,7)$ and $s_{B'} = (18,24,9)$. We find $C_1' = (23,4,16)$, $C_2' = (6,10,5,19,11,22)$ and $C_3' = (15,12,20)$.

\begin{minipage}{0.8\textwidth}
\begin{footnotesize}
\[\left.
\begin{array}{cccc}
\let\quad\thinspace
\T= \yt{{2\ast},2} & \yt{3{3\ast}3,333,333,333,333,333}\\[0.7cm]
\pi = (8|17)&(3,14,7|18,24,9|15,12,20|6,10,5|19,11,22|23,4,16)
\end{array}\right| 
\left.\begin{array}{cc}
\let\quad\thinspace
\U= \yt{{2\ast},2}\\[0.5cm]
\rho = (2|21)
\end{array}\right|
\sigma = (13,1)
\]
\end{footnotesize}
\end{minipage}

The output obtained is

\begin{minipage}{0.8\textwidth}
\begin{footnotesize}
\[\left.
\begin{array}{cccc}
\let\quad\thinspace
\T'=  \yt{3{3\ast}3} &\yt{{2\ast},2}  & \yt{33{3\ast},333} & \yt{3{3\ast}3}\\[0.7cm]
\pi' = (12,20,15)& (8|17)& (5,19,11|22,6,10) &(4,16,23)
\end{array}\right| \]\[
\left.\begin{array}{cc}
\let\quad\thinspace
\U'= \yt{3{3\ast}3}&\yt{{2\ast},2}\\[0.5cm]
\rho' = (3,14,7)&(2|21)
\end{array}\right|
\sigma' = (18,24,9,13,1)
\]
\end{footnotesize}
\end{minipage}

\item Suppose $T_i\neq T_j$. Obtain $C_1', \ldots, C_k'$ from $C_j$ by removing the entries of $s_{B'}$ and following the procedure in (2)(b)(i)(A). Similarly, obtain $T_1', \ldots, T_k'$ with bars identical to $B$ as stated above. Obtain $\pi'$ from $\pi$ by removing $C_i$ and expressing  $C_1\ldots C_{i-1} C_{i+1}\ldots C_{j-1}\ldots$ $C_1'\ldots C_k'\ldots C_{j+1} \ldots C_a$ in canonical notation. Obtain $\T'$ from $\T$ by removing $T_i$ and $T_j$, and reordering the RBTs $T_1, \ldots, T_{i-1}, T_{i+1}, \ldots$ $T_{j-1}, T_1',\ldots, T_k', T_{j+1}, \ldots, T_a$ in accordance with $\pi'$.
Define $\sigma' = s_{B'}\sigma$.
For the example

\begin{minipage}{0.8\textwidth}
\begin{footnotesize}
\[\left.
\begin{array}{cccc}
\let\quad\thinspace
\T= \yt{{2\ast},2} & \yt{3{3\ast}} &\yt{4{4\ast}4} &\yt{{3\ast}3,33}\\[0.5cm]
\pi = (9,14)&(8,10)&(5,15,6)&(4,17|16,11)
\end{array}\right| 
\left.\begin{array}{cc}
\let\quad\thinspace
\U= \yt{{2\ast}} & \yt{{1\ast}1}\\[0.5cm]
\rho = (7)&(2,12)
\end{array}\right|
\sigma = (13,1,3)
\]
\end{footnotesize}
\end{minipage}

the cyc-index of $T_2 = B$ is 8 and is larger than the cyc-index of $U_1 = \yt{{2\ast}}$ which is 7. We have $s_B = (8,10)$ and $s_{B'} = (4,17)$. The action of $\Psi$ gives

\begin{minipage}{0.8\textwidth}
\begin{footnotesize}
\[\left.
\begin{array}{cccc}
\let\quad\thinspace
\T'= \yt{3{3\ast}} & \yt{{2\ast},2} &\yt{4{4\ast}4}\\[0.5cm]
\pi' = (11,16)&(9,14)&(5,15,6)
\end{array}\right| 
\left.\begin{array}{ccc}
\let\quad\thinspace
\U'= \yt{3{3\ast}} & \yt{{2\ast}} & \yt{{1\ast}1}\\[0.5cm]
\rho' = (8,10) &(7)&(2,12)
\end{array}\right|
\sigma' = (4,17,13,1,3)
\]
\end{footnotesize}
\end{minipage}
 \end{enumerate}
 \item Suppose the cyc-index of $B$ is smaller than the cyc-index of $A:= U_1$ which is $D_1(1)$. We apply the same procedure as in (2)(a)(ii). Note that in this case, $A'$ in layer 1 is the first bar in scanning order of $(\T', \U')$ with a matching bar as it has the largest cyc-index in layer 1.
 
\end{enumerate}
\item Suppose that each bar in layer 1 has a matching bar in layer 2, and possibly some matching bars in layer 1. Obtain $\U'$ by removing $A:= U_1$ from $\U$, and $\rho'$ be obtained by removing $D_1$ from $\rho$. Suppose, the cyc-index of $A$ is such that $T_{r_1}$, $T_{r_2}$, $\ldots$, $T_{r_k}$ (read in scanning order from left to right) are marked RBTs containing bars identical to $A$ with cyc-indices larger than the cyc-index of $A$ for $k\geq 0$. If $k =0$, then no such marked RBT exists.

\begin{enumerate}
\item Suppose for all $1\leq i \leq |A|$, $\sigma(i)$ is greater than the cyc-index of $A$. Define $\hat{T}$ to be the marked RBT with $\ell(T_{r_1}) + \ell(T_{r_2}) + \ell(T_{r_k}) +2$ bars identical to $A$ and the marked cell in the same column as $A$. The term $+2$ arises from two copies of bars identical to $A$ in the first and second row. For $1\leq i\leq k$, let $\hat{C}_i$ be the cyclic shift of $C_{r_i}$ where the minimal element is in the position equal to the index of the column of $T_{r_i}$ containing the marked cell.
Define $\hat{C} = (s_{A}, \sigma(1),\ldots \sigma(|A|), \hat{C}_1, \ldots, \hat{C}_k)$. Obtain $\hat{\pi}$ from $\pi$ by removing $C_{r_1}, \ldots, C_{r_k}$ from $\pi$. Let $\pi'$ be obtained by expressing $\hat{\pi} \cdot \hat{C}$ in canonical notation. Obtain $\T'$ from $\T$ by removing $T_{r_1}, \ldots, T_{r_k}$ and inserting $\hat{T}$ in the position corresponding to $\hat{C}$ in $\pi'$. Obtain $\sigma'$ by removing the entries corresponding to $s_A$. 

\begin{minipage}{0.8\textwidth}
\begin{footnotesize}
\[\left.
\begin{array}{cccc}
\let\quad\thinspace
\T= \yt{2{2\ast},22,22} & \yt{{2\ast}2} &\yt{{4\ast},4} &\yt{2{2\ast}}\\[0.5cm]
\pi = (9,18|10,12|16,13) & (8,14) & (5|17) & (1,7) 
\end{array}\right| 
\left.\begin{array}{cc}
\let\quad\thinspace
\U=  \yt{{2\ast}2} & \yt{{4\ast}}\\[0.5cm]
\rho = (4,11)&(2)
\end{array}\right|
\sigma = (6,15,3)
\]
\end{footnotesize}
\end{minipage}

We observe that $A = \yt{{2\ast}2}$. Both $\sigma(1) = 6$ and $\sigma(2) = 15$ are greater than 4. The RBTs $T_1$, $T_2$ and $T_4$ have bars identical to $A$ but only $T_1$ and $T_2$ have a cyc-index larger than $A$ while $T_4$ does not. We construct $\hat{T}$ with $\ell(T_1) + \ell(T_2)  + 2 = 6$ bars identical to $A$ with the marked cell in the first column. We cyclically shift $C_1$ once to obtain $\hat{C}_1$ so that 9 is in the second position to match with the marked cell. We do not shift $(8,14)$ as the marked cell is in the first column. 

\begin{minipage}{0.8\textwidth}
\begin{footnotesize}
\[\left.
\begin{array}{ccc}
\let\quad\thinspace
\T' =  \yt{{4\ast},4} & \yt{{2\ast}2,22,22,22,22,22}&\yt{2{2\ast}}\\[0.5cm]
\pi' = (5|17) & (4,11|\, 6,15|\, 13,9|\, 18,10|\, 12,16|\, 8,14) &(1,7)
\end{array}\middle| 
\begin{array}{cc}
\let\quad\thinspace
\U'=  \yt{{4\ast}}\\[0.5cm]
\rho' = (2)
\end{array}\right|
\sigma' = (3)
\]
\end{footnotesize}
\end{minipage}

Note that applying $\Psi$ to the above output results in case (2)(b), and we split the cycle $\hat{C}$ as $(4,11), (6,15), (9,18|10,12|16,13), (8,14)$ where the first split corresponds to $A$, the second corresponds to the entries we append to $\sigma'$, and the third and fourth correspond to the RBTs $T_1$ and $T_2$.

\item Suppose $\sigma(\hat{a})$ for some $1\leq \hat a \leq |A|$ is the least value such that $\sigma(\hat{a})$ is smaller than the cyc-index of $A$. Then define 
$\hat{s}$ to be the cyclic shift of $(\sigma(1),\ldots \sigma(|A|))$ such that the first element is $\sigma(\hat{a})$. For $1\leq i\leq k$, let $\hat{C}_i$ be the cyclic shift of $C_{r_i}$ where the minimal element is in the position equal to the index of the column of $T_{r_i}$ containing the marked cell. Define $\hat{C} = (\hat{s},\hat{C}_1, \ldots, \hat{C}_k)$ and $\hat{\pi}$ by removing $C_{r_1}\ldots C_{r_k}$ from $\pi$. Define $\pi'$ to be the permutation $\hat{\pi}\cdot s_A\cdot \hat{C}$ in canonical notation. Define $\hat{T}$ to be the marked RBT with $\ell(T_{r_1}) + \ell(T_{r_2}) +\ldots+ \ell(T_{r_k}) +1$ bars identical to $A$ with the cell in column $\hat{a}$ marked. Obtain $\T'$ from $\T$ by removing $T_{r_1},\ldots T_{r_k}$, inserting $A$ in the position corresponding to $(s_A)$ and $\hat{T}$ in the position corresponding to $\hat{C}$ with the other RBTs of $\T$ inserted according to $\hat{\pi}$. Obtain $\sigma'$ by removing the first $|A|$ entries from $\sigma$. In the following example, we choose $a = 13$ and $b = 5$ with $\alpha = (9,4)$ and $\beta = (3,2)$.

\begin{minipage}{0.8\textwidth}
\begin{footnotesize}
\[\left.
\begin{array}{cccc}
\let\quad\thinspace
\T= \yt{33{3\ast},333,333} & \yt{2{2\ast},22} \\[0.5cm]
\pi =(8,18,9|22,20,10|17,12,15) & (3,13|4,16) 
\end{array}\right| 
\left.\begin{array}{cc}
\let\quad\thinspace
\U=  \yt{{3\ast}33} & \yt{1{1\ast}}\\[0.5cm]
\rho = (5,19,23) & (2,7)
\end{array}\right|
\sigma = (14,21,1,11,6)
\]
\end{footnotesize}
\end{minipage}

Applying $\Psi$ on the above object yields the output:

\begin{minipage}{0.8\textwidth}
\begin{footnotesize}
\[\left.
\begin{array}{ccc}
\let\quad\thinspace
\T'= \yt{2{2\ast},22} & \yt{{3\ast}33} &\yt{33{3\ast},333,333,333}  \\[0.5cm]
\pi' = (3,13|4,16)  & (5,19,23) & (1,14,21|12,15,8|18,9,22|20,10,17)
\end{array}\right| 
\left.\begin{array}{c}
\let\quad\thinspace
\U'=  \yt{1{1\ast}}\\[0.5cm]
\rho' = (2,7)
\end{array}\right|
\sigma' = (11,6)
\]
\end{footnotesize}
\end{minipage}

\end{enumerate}
\end{enumerate}


\end{enumerate}
We remark that case (2)(b)(i) leads to (2)(c)(i), and vice versa. Similarly, (2)(b)(ii) leads to (2)(c)(ii).\qedhere
\end{proof}

\section{Expansions of $E$, $H$ and $P$ in $E^+$}\label{sec:inE+}
The combinatorial proofs of the $E^+$-expansions of $E_d$, $H_d$, and $P_d$ require us to deal with multiple layers indexed by powers of 2.
We call a polycomposition $\delta$ of $n$ \textit{dyadic} if all multiplicities are powers of 2, that is, $\delta|^{(i)} = \varnothing$ for all $u\neq 2^k$. For instance, $\delta = (2,1)^1(3,2)^2(1)^8(1,2,1)^{16}$ is a dyadic polycomposition of $85$. We denote the set of dyadic polycompositions of $n$ by $\PCom_\dya(n)$. The indexing polycompositions that appear in the previous sections are all dyadic polycompositions. In particular, square-free polycompositions that appear in Section \ref{sec:HEP} are dyadic with the only multiplicity being $2^0 = 1$. The polycompositions in Section \ref{sec:E+in} are dyadic with multiplicities restricted to 1 and 2. 
\subsection{$E^+$-Expansion of $H$}
The $E^+$-expansion of $H_d$ is indexed by a restricted set of dyadic polycompositions of $d$ where each $\delta|^i$ for $i\geq 1$ has at most one part. We denote this set by $\PCom_\dya'(d)$. For example, $(7)^2(3)^4(1)^{16}\in \PCom_{\dya}'(42)$. Note that each polycomposition in this set is also a type as a composition consisting of a single part is also a partition. In this case, owing to the appearance of powers of 2 as multiplicities, we keep track of them using layers indexed by powers of 2. We omit empty layers in our visualizations. Recall that a WBT $T$ in $\T$ appearing in layer $r$ contributes the factor $\x_T^r$ to $\x_{\T}$.
\begin{prop}\label{prop:HinE+}
For $d\geq 0$, ${H_d = \sum_{\delta\in \PCom_\dya'(d)} E^+_\delta}$.
\end{prop}
\begin{proof}
By the results of Lemma \ref{lem:btinterps} and \ref{lem:pbtinterps}, it is sufficient to prove the following identity on monomials:
\[
\sum\limits_{T\in \wbt[d]}\x_T = \sum\limits_{\delta\in \PCom_\dya'(d)}\sum\limits_{\T\in \psbt(\delta)} \x_{\T}.
\]
Define $\sigma:=\sigma^H_{E^+}: \wbt[d]\to \bigcup_{\delta\in \PCom_\dya'(d)} \psbt(\delta)$ as follows: write the output $\sigma(T)$ as $(T_0, T_1, \ldots)$ where $T_i$ is an SBT appearing in layer $2^i$. Let $n_{i,j}(T)$ be the number of copies of bars with length $i$ and label $j$ in $T$, and place a bar with length $i$ and label $j$ in $T_k$ if and only if $2^k$ appears in the binary expansion of $n_{i,j}(T)$. This map is a bijection as binary expansions are unique. To define the inverse bijection $\sigma^{-1}$, we include $2^k$ copies of each bar appearing in layer $2^k$ of $\T$ to form the WBT $\sigma^{-1}(\T)$.
\end{proof}
\begin{example}
The following example illustrates the action of $\sigma$ for $d = 25$.
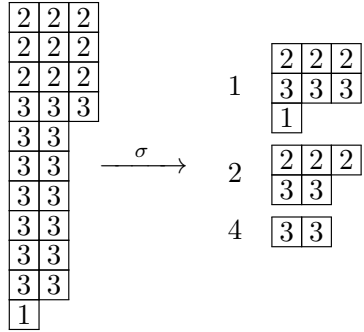
\begin{center}
\begin{minipage}{0.3\textwidth}
\begin{multicols}{2}
\yt{222,222,222,333,33,33,33,33,33,33,1} $\xrightarrow{\quad\sigma\quad}$
\columnbreak

\begin{align*}
    1\quad & \yt{222,333,1}\\
    2 \quad& \yt{222,33}\\
    4 \quad& \yt{33}
\end{align*}
\end{multicols}
\end{minipage}
\end{center}
\end{example}

\subsection{$E^+$-Expansion of $E$}
We prove the following expansion.
\begin{prop}\label{prop:EinE+}
For $d\geq 0$, $\displaystyle{E_d = \sum_{\delta\in \PCom_\dya(d)} (-1)^{\ell(\delta)}E_\delta^+}$.
\end{prop}
\begin{proof}
From the results of Lemma \ref{lem:btinterps} and \ref{lem:pbtinterps}, we interpret the above expansion in terms of monomials as
\[
\sum\limits_{T\in \sbt[d]} \sgn(T)\x_T = \sum\limits_{\delta\in\PCom_{\dya}(d)}\sum\limits_{\T\in \psbt(\delta)} (-1)^{\ell(\T)} \x_{\T}.
\]
We describe an involution $\sigma: = \sigma^{E}_{E^+}$ on the set $\bigcup\limits_{\delta\in\PCom_{\dya}(d)}\psbt(\delta)$ as follows. We consider a PSBT $\T$.
\begin{enumerate}
\item Suppose there exists a layer containing either (i) an SBT with more than one row, or (ii) two consecutive bars $(B,B')$ such that $B'$ has a strictly smaller length than $B$, or if they have the same length $B$ has a strictly smaller label than $B'$. Let $r$ be the smallest index of such a layer. Let $\sigma(\T)$ be the output of the {strict} stack-or-slash operation (cf. Section \ref{ssec:H-E}) acting on layer $r$. This operation changes the number of diagrams by one, that is, $\ell(\sigma(\T)) = \ell(\T)\pm 1$. We get $(-1)^{\ell(\sigma(\T))} \x_{\sigma(\T)} =  - (-1)^{\ell(\T)} \x_{\T}$ which allows us to pair the monomials corresponding to $\T$ and $\sigma(\T)$ for cancellation.
\item Suppose each SBT in $\T$ is a bar and within each layer, the bars weakly increase in size and the labels weakly decrease within bars of the same size. Note that such a $\T$ is fixed under the strict stack-or-slash operation. Let $2^k$ be the highest index of the non-empty layer in $\T$.
\begin{enumerate}
\item Suppose there exists a rightmost identical pair $(B,B')$ in layer $2^k$. Obtain $\sigma(\T)$ from $\T$ by removing $(B,B')$ from layer $2^k$ and inserting $T$ in layer $2^{k+1}$. In the following example, the rightmost pair $(B,B')=(\yt{22},\yt{22})$ in layer 4 is removed and $\yt{22}$ is inserted in layer 8.
\begin{multicols}{2}
\noindent
\begin{align*}
1\quad & \yt{2} \quad \yt{22} \quad \yt{11} \\
\T = 4\quad & \yt{1}\quad \yt{22}\quad \yt{22}\quad \yt{11}\\
8\quad &
\end{align*}
\columnbreak
\noindent
\begin{align*}
1\quad & \yt{2} \quad \yt{22} \quad \yt{11} \\
\sigma(\T) = 4\quad & \yt{1}\quad  \yt{11}\\
8\quad &\yt{22}
\end{align*}
\end{multicols}
\item Suppose layer $2^k$ does not contain an identical pair.
\begin{enumerate}
\item If $2^k > 1$, let $B$ be the rightmost bar in layer $2^k$. 
\begin{enumerate}
\item If layer $2^{k-1}$ contains no identical pair, then remove $B$ from layer 2 and insert the identical pair $(B,B')$ in layer $2^{k-1}$ while preserving the weakly increasing length and weakly decreasing label between bars of the same size conditions. Note that this covers the case when layer $2^{k-1}$ is empty. 
\begin{multicols}{2}
\noindent
\begin{align*}
1\quad & \yt{2} \quad \yt{22} \quad \yt{11} \\
\T = 4\quad & \yt{1}\quad  \yt{11}\\
8\quad &\yt{3}\quad \yt{11}
\end{align*}
\columnbreak
\noindent
\begin{align*}
1\quad & \yt{2} \quad \yt{22} \quad \yt{11} \\
\sigma(\T) = 4\quad & \yt{1}\quad \yt{11}\quad \yt{11}\quad \yt{11}\\
8\quad &\yt{3}
\end{align*}
\end{multicols}
\item If layer $2^{k-1}$ contains a rightmost identical pair $(U,U')$ such that $B$ has a larger length than $U$, or the same length but a weakly smaller label, then obtain $\sigma(\T)$ from $\T$ by removing $B$ from layer $2^k$ and inserting the identical pair $(B,B')$ in layer $2^{k-1}$ while preserving the weakly increasing length among bars and weakly decreasing label between bars of the same size conditions. Note that this results in $(B,B')$ being the rightmost bar in layer $2^{k-1}$ of $\sigma(\T)$. In the following example $B$ is $\yt{22}$ in layer 8 and $(U,U') = (\yt{22}, \yt{22})$ in layer 4. In this case, $B$ has the same length and a weakly smaller label than $U$.
\begin{multicols}{2}
\noindent
\begin{align*}
1\quad & \yt{2} \quad \yt{33} \quad \yt{11} \\
\T = 4\quad & \yt{22}\quad  \yt{22}\\
8\quad &\yt{3}\quad \yt{22}
\end{align*}
\columnbreak
\noindent
\begin{align*}
1\quad & \yt{2} \quad \yt{33} \quad \yt{11} \\
\sigma(\T) = 4\quad & \yt{22}\quad  \yt{22}\quad \yt{22}\quad \yt{22}\\
8\quad &\yt{3}
\end{align*}
\end{multicols}
\item If layer $2^{k-1}$ contains a rightmost identical pair $(U,U')$ such that $B$ has a smaller length than $U$, or the same length but a strictly larger label, then obtain $\sigma(\T)$ from $\T$ by removing $(U,U')$ from layer $2^{k-1}$ and inserting $U$ in layer $2^k$ while preserving the weakly increasing length and weakly decreasing label between bars of the same size conditions. Note that this results in $U$ being the last bar in the scanning order for $\sigma(\T)$. The illustration of this action can be seen by acting $\sigma$ on the output $\sigma(\T)$ in the previous example.
\end{enumerate}
\item If $2^k = 1$, then $\sigma(\T) = \T$. We can map such fixed points $\T$ consisting of distinct bars to an SBT $T$. The $i$th row from the top of $T$ is the $i$th bar from the right in layer 1 of $\T$. As the number of bars in $\T$ is the number of rows in $T$, we get $\ell(T) = \ell(\T)$ which matches the sign of the monomial on each side. The following example shows this correspondence
\[
\T = 1 \quad \yt{2} \quad \yt{1}\quad \yt{33}\quad \yt{11}\quad \yt{444} \mapsto \yt{444,11,33,1,2}.
\]
\end{enumerate}
\end{enumerate}
\end{enumerate}

Notice that in both in parts 1, 2(a), and 2(b)(i) the number of diagrams changes by one and we have $(-1)^{\ell(\sigma(\T))} \x_{\sigma(\T)} =  - (-1)^{\ell(\T)} \x_{\T}$  which pair up and cancel out. To see that $\sigma$ is an involution, we note that (1) is already an involution. Now, consider the action of $\sigma$ on the output $\sigma(\T)$ of 2(a). If $(B,B')$ was the only identical pair in layer $2^{k-1}$, then we find ourselves in case 2(b)(i)(A), which gives us back $\T$. If there was another identical pair $(C,C')$ in layer $2^{k-1}$ of $\T$, then $C$ either has a smaller length than $B$, or the same length but a weakly larger label. This case is handled by 2(b)(i)(B) which gives us back $\T$ under $\sigma$.

Now, consider the action of $\sigma$ on the output $\sigma(\T)$ of 2(b)(i)(A). The rightmost bar $C$ in layer $2^k$ must have a smaller length, or the same length but a strictly larger label than $B$. This is handled by 2(b)(i)(C) and gives us back $\T$. A similar argument works for the output of 2(b)(i)(B). The output $\sigma(\T)$ of 2(b)(i)(C) results in either case (A) or (B), and gives us back $\T$.

\end{proof}

\subsection{$E^+$-expansion of $P$}
The dyadic polycompositions that appear in this expansion have exactly one multiplicity, and we call them \textit{singular dyadic polycompositions}. The set of singular dyadic polycompositions of $n$ is denoted by $\PCom_\dya\ast(n)$. Recall that for $\delta = (d_1\ldots d_k)^{r} \in \PCom_\dya\ast(n)$, $L(\delta) = rd_k$. For $\T$ in $\psbt\ast(\delta)$, which is the set of SBTs such that one cell of the last SBT (in scanning order) is marked, recall that $\wt\ast(\T)$ is the layer index occupied by the tableau containing the marked bar.
\begin{prop}\label{prop:PinE+}
For $d\geq 0$, $\displaystyle{P_d=\sum_{\delta\in \PCom_\dya\ast(d)} (-1)^{\ell(\delta)-1} L(\delta) E_\delta^+}$.
\end{prop}
\begin{proof}
Using Lemmas \ref{lem:btinterps} and \ref{lem:markedinterps}, we must prove
\[
  \sum\limits_{T\in \rbt\ast[d]} \x_T = \sum\limits_{\delta\in \PCom_\dya\ast(d)}\sum\limits_{\T \in \psbt\ast(\delta)} (-1)^{\ell(\T)-1}\wt\ast(\T) \x_{\T}.
\]
Let the involution $\sigma$ on $\bigcup_{\delta\in \PCom_\dya\ast(d)}\psbt\ast(\delta)$ be $\sigma^P_E$ as defined in the proof of Proposition \ref{prop:P-H,P-E} part (2), but this time acting on layer $r$. We showed that $\sigma(\T)$ has one diagram more or less than $\T$ and preserves the layer in which the marked bar is located, thus $(-1)^{\ell(\sigma(\T))-1}\wt\ast(\sigma(\T))\x_{\sigma(\T)} = -(-1)^{\ell(\T)-1}\wt\ast(\T) \x_{\T}$. The fixed points under $\sigma$ are marked PSBTs $\T$ such that all SBTs in $\T$ are bars that are identical to each other. As the bars are identical, it follows that their lengths must divide $d$. Choose a divisor $k$ of $d$ such that $k = 2^xy$ for some non-negative integer $x$ and an odd positive integer $y$. Choose $j\geq 1$ and $1\leq c \leq d/k$. The marked RBT $T$ corresponding to a monomial term $x_{d/k,j}^{k}$ consists of $k$ identical bars of length $d/k$ with label $j$ and a marked cell in column $c$. Consider the set of marked PSBTs $\{\T_z\}_{0\leq z \leq x}$ where $\T_z$ consists of $2^{x-z}y$ bars of length $d/k$ with label $j$ in layer $2^z$ such that the marked cell is in column $c$. So, $\T_0$ consists of $2^xy$ bars in layer 1, $\T_1$ has $2^{x-1}y$ bars in layer 2, and so on, until $\T_x$ has $y$ bars in layer $2^x$. By factoring in all the contributions from the bars, we compute $\x_{\T_z} = ((x_{d/k,j})^{2^{x-z}y})^{2^z} = x_{d/k,j}^{k}$ for all $0\leq z\leq x$. We see that only layer $2^x$ has an odd number of bars, thus $(-1)^{\ell(\T_z)-1} = 1$ when $z = x$, and $(-1)^{\ell(\T_z)} = -1$ for all other values of $z$. This gives us 
\[
\sum\limits_{z=0}^x (-1)^{\ell(\T_z)-1}\wt\ast(\T_z) = 2^x - 2^{x-1} - \ldots - 2 -1 = 1
\]
which shows that the monomial $x_{d/k,j}^{k}$ arising from PSBTs with the marked cell in a certain column appears with coefficient 1 on both sides of the identity. Accounting for all possible columns that can be marked, we get that  $x_{d/k,j}^{k}$ appears with the coefficient $d/k$. As the above proof holds for all arbitrary divisors $k$ of $d$ and all labels $j$, we obtain the statement of the proposition. 
\end{proof}
\begin{example}
In the following example, we choose $d = 24$, $k = 12$, $j = 5$ and $c = 2$. We find $k = 2^2\cdot 3$, so $x = 2$ and $y = 3$. So, $\T_0$, $\T_1$ and $\T_2$ consist of identical bars $\yt{55}$ where the rightmost bar has a marked cell in the second column corresponding to $c = 2$. The monomials arising from $\T_0$, $\T_1$ and $\T_2$ respectively are $(-1)^{\ell(\T_0)-1}\x_{\T_0} = -x_{2,5}^{12}$, $(-1)^{\ell(\T_1)-1}\x_{\T_1} = -x_{2,5}^{12}$ and $(-1)^{\ell(\T_2)-1}\x_{\T_2} = x_{2,5}^{12}$. We have $\T_0$, $\T_1$ and $\T_2$ as follows
\[
\T_0: 1 \quad \yt{55}\; \yt{55}\; \yt{55}\;\yt{55}\; \yt{55}\;\yt{55}\;\yt{55}\;\yt{55}\;\yt{55}\;\yt{55}\;\yt{55}\;\yt{5{5\ast}}
\]
\[
\T_1: 2 \quad \yt{55}\;\yt{55}\;\yt{55}\;\yt{55}\;\yt{55}\;\yt{5{5\ast}}
\]
\[
\T_2: 4 \quad \yt{55}\;\yt{55}\;\yt{5{5\ast}}
\]
\end{example}
\section{Polybrick tabloids and expansions among plethystic bases}\label{sec:polybrick}
The expansions from the previous sections expressing a plethystic basis element $F_d$ in terms of another bass $G$ can be extended to obtain $G$-expansions of $F_\tau$ for an arbitrary type $\tau$. The objects that appear in these general expansions are polysymmetric analogs of {brick tabloids} introduced by Remmel and E\u{g}ecio\u{g}lu \cite{eg-rem}. We first recall brick tabloids, which appear when studying the transition matrices between pairs of bases of $\sym$ involving $h$, $e$, and $p$. 

\subsection{Brick tabloids}\label{ssec:bricktabs}
Define a \textit{brick} to be a collection of consecutive cells in a row of a partition diagram. We visualize a brick by placing a rectangle over the collection of cells, and we say that the brick \textit{spans} those cells. The \textit{length} of a brick is the number of cells it spans. A \textit{tiling} of $\lambda\in \Par$ is a decomposition of $\dg(\lambda)$ as a disjoint union of bricks. For partitions $\lambda$ and $\mu$ of $n$, define a \textit{brick tabloid of shape $\lambda$ and content $\mu$} to be a tiling of $\lambda$ such that the partition formed by the multiset of lengths of the bricks is $\mu$. A \textit{labeled brick with label $l$} is a brick with an associated natural number $l$ which we visualize by placing the label next to the bottom-right corner of the brick. An \textit{ordered brick tabloid of shape $\lambda$ and content $\mu$} is a decomposition of $\dg(\lambda)$ as disjoint union of $\ell(\mu)$ labeled bricks with unique labels in $\{1,\ldots, \ell(\mu)\}$ such that the labels increase left to right in each row.
\begin{example}
For $\lambda = (8,4,2)$ and $\mu = (3,3,2,2,2,1,1)$, shown below are a brick tabloid $T$ and an ordered brick tabloid $T'$ of shape $\lambda$ and content $\mu$.
\[
T =
\raisebox{16.0pt}{\begin{tikzpicture}[scale =0.400,baseline=(current bounding box.north)]
\draw (0,0) rectangle (1,-1);
\draw (1,0) rectangle (2,-1);
\draw (2,0) rectangle (3,-1);
\draw (3,0) rectangle (4,-1);
\draw (4,0) rectangle (5,-1);
\draw (5,0) rectangle (6,-1);
\draw (6,0) rectangle (7,-1);
\draw (7,0) rectangle (8,-1);
\draw (0,-1) rectangle (1,-2);
\draw (1,-1) rectangle (2,-2);
\draw (2,-1) rectangle (3,-2);
\draw (3,-1) rectangle (4,-2);
\draw (0,-2) rectangle (1,-3);
\draw (1,-2) rectangle (2,-3);
\draw (0.250,-0.250) rectangle ( 2.75 ,-0.750);
\draw (3.25,-0.250) rectangle ( 4.75 ,-0.750);
\draw (5.25,-0.250) rectangle ( 7.75 ,-0.750);
\draw (0.250,-1.25) rectangle ( 0.750 ,-1.75);
\draw (1.25,-1.25) rectangle ( 2.75 ,-1.75);
\draw (3.25,-1.25) rectangle ( 3.75 ,-1.75);
\draw (0.250,-2.25) rectangle ( 1.75 ,-2.75);
\end{tikzpicture}}\qquad T' = \raisebox{16.0pt}{\begin{tikzpicture}[scale =0.400,baseline=(current bounding box.north)]
\draw (0,0) rectangle (1,-1);
\draw (1,0) rectangle (2,-1);
\draw (2,0) rectangle (3,-1);
\draw (3,0) rectangle (4,-1);
\draw (4,0) rectangle (5,-1);
\draw (5,0) rectangle (6,-1);
\draw (6,0) rectangle (7,-1);
\draw (7,0) rectangle (8,-1);
\draw (0,-1) rectangle (1,-2);
\draw (1,-1) rectangle (2,-2);
\draw (2,-1) rectangle (3,-2);
\draw (3,-1) rectangle (4,-2);
\draw (0,-2) rectangle (1,-3);
\draw (1,-2) rectangle (2,-3);
\draw (0.350,-0.250) rectangle ( 2.55 ,-0.700);
\draw ( 2.80 ,-0.720) node[scale =0.600] {2};
\draw (3.35,-0.250) rectangle ( 4.55 ,-0.700);
\draw ( 4.80 ,-0.720) node[scale =0.600] {4};
\draw (5.35,-0.250) rectangle ( 6.55 ,-0.700);
\draw ( 6.80 ,-0.720) node[scale =0.600] {5};
\draw (7.25,-0.250) rectangle ( 7.65 ,-0.700);
\draw ( 7.84 ,-0.720) node[scale =0.600] {7};
\draw (0.350,-1.25) rectangle ( 2.55 ,-1.70);
\draw ( 2.80 ,-1.72) node[scale =0.600] {1};
\draw (3.25,-1.25) rectangle ( 3.65 ,-1.70);
\draw ( 3.84 ,-1.72) node[scale =0.600] {6};
\draw (0.350,-2.25) rectangle ( 1.55 ,-2.70);
\draw ( 1.80 ,-2.72) node[scale =0.600] {3};
\end{tikzpicture}}
\]
\end{example}
For a brick tabloid $T$, define $L(T)$ to be the product of the lengths of all bricks that appear at the end of each row. For $T$ in the above example, the first row ends with a brick of length 3, the second row ends with a brick of length 1, and the third row ends with a brick of length 2, giving us $L(T)= 3\cdot 1 \cdot 2 = 6$. Similarly, $L(T') = 1\cdot 1\cdot 2 = 2$.

\subsection{Simple Polybrick tabloids}
This section covers the analogs of brick tabloids that arise when computing the expansions between $F$ and $G$ for $F,G\in \{H,E,P\}$.

Let $\sigma$ and $\tau = d_1^{r_1}d_2^{r_2}\ldots d_s^{r_s}$ with $r_1 \leq r_2 \leq \ldots \leq r_s$ be types of $n$. Define a \textit{simple polybrick tabloid, $T$, of shape $\sigma$ and content $\tau$} to be the tensor product of brick tabloids $T_i$ of shape $\sigma|^i$ and content $\tau|^i$ for $i\geq 1$. In other words, the partition formed by the lengths of bricks in $\sigma|^i$ must be $\tau|^i$ for each $i\geq 1$. Denote the set of simple polybrick tabloids of shape $\sigma$ and content $\tau$ by $\pt^{\mathrm{simp}}(\sigma, \tau)$. Also, define an \textit{ordered simple polybrick tabloid of shape $\sigma$ and content $\tau$} to be a tensor product of ordered brick tabloids of shape $\sigma|^i$ and content $\tau|^i$ for $i\geq 1$ such that for each $1\leq l \leq s$, we have a labeled brick of length $d_l$ in the $r_l$th tensor factor with label $l$, and the labels increase from left to right in each row of each tensor factor. Denote the set of ordered simple polybrick tabloids of shape $\sigma$ and content $\tau$ by $\pt^{\mathrm{osimp}}(\sigma, \tau)$. Denote by $\L(T)$ for $T\in \pt^{\mathrm{simp}}(\sigma, \tau)$ the product of the sizes of all bricks that appear at the end of each row, i.e., $\L(T) = \prod_{i\geq 1} L(T_i)$.
\begin{example}
\boks{0.3}
Here are some examples of simple polybrick tabloids of shape $\sigma = (5,3)^1(2,1)^2(3,2)^3$ and content $\tau = (3,2,2,1)^1(2,1)^2(2,2,1)^3$:
\begin{footnotesize}
\[
T_1 = 
\raisebox{16.0pt}{\begin{tikzpicture}[scale =0.380,baseline=(current bounding box.north)]
\draw (0,0) rectangle (1,-1);
\draw (1,0) rectangle (2,-1);
\draw (2,0) rectangle (3,-1);
\draw (3,0) rectangle (4,-1);
\draw (4,0) rectangle (5,-1);
\draw (0,-1) rectangle (1,-2);
\draw (1,-1) rectangle (2,-2);
\draw (2,-1) rectangle (3,-2);
\draw (0.250,-0.250) rectangle ( 2.75 ,-0.750);
\draw (3.25,-0.250) rectangle ( 4.75 ,-0.750);
\draw (0.250,-1.25) rectangle ( 0.750 ,-1.75);
\draw (1.25,-1.25) rectangle ( 2.75 ,-1.75);
\end{tikzpicture}}
\otimes \raisebox{16.0pt}{\begin{tikzpicture}[scale =0.380,baseline=(current bounding box.north)]
\draw (0,0) rectangle (1,-1);
\draw (1,0) rectangle (2,-1);
\draw (0,-1) rectangle (1,-2);
\draw (0.250,-0.250) rectangle ( 1.75 ,-0.750);
\draw (0.250,-1.25) rectangle ( 0.750 ,-1.75);
\end{tikzpicture}}
\otimes \raisebox{16.0pt}{\begin{tikzpicture}[scale =0.380,baseline=(current bounding box.north)]
\draw (0,0) rectangle (1,-1);
\draw (1,0) rectangle (2,-1);
\draw (2,0) rectangle (3,-1);
\draw (0,-1) rectangle (1,-2);
\draw (1,-1) rectangle (2,-2);
\draw (0.250,-0.250) rectangle ( 0.750 ,-0.750);
\draw (1.25,-0.250) rectangle ( 2.75 ,-0.750);
\draw (0.250,-1.25) rectangle ( 1.75 ,-1.75);
\end{tikzpicture}}\quad \Bigg\vert\,
T_2 = 
\raisebox{16.0pt}{\begin{tikzpicture}[scale =0.380,baseline=(current bounding box.north)]
\draw (0,0) rectangle (1,-1);
\draw (1,0) rectangle (2,-1);
\draw (2,0) rectangle (3,-1);
\draw (3,0) rectangle (4,-1);
\draw (4,0) rectangle (5,-1);
\draw (0,-1) rectangle (1,-2);
\draw (1,-1) rectangle (2,-2);
\draw (2,-1) rectangle (3,-2);
\draw (0.250,-0.250) rectangle ( 2.75 ,-0.750);
\draw (3.25,-0.250) rectangle ( 4.75 ,-0.750);
\draw (0.250,-1.25) rectangle ( 1.75 ,-1.75);
\draw (2.25,-1.25) rectangle ( 2.75 ,-1.75);
\end{tikzpicture}}
\otimes \raisebox{16.0pt}{\begin{tikzpicture}[scale =0.380,baseline=(current bounding box.north)]
\draw (0,0) rectangle (1,-1);
\draw (1,0) rectangle (2,-1);
\draw (0,-1) rectangle (1,-2);
\draw (0.250,-0.250) rectangle ( 1.75 ,-0.750);
\draw (0.250,-1.25) rectangle ( 0.750 ,-1.75);
\end{tikzpicture}}
\otimes \raisebox{16.0pt}{\begin{tikzpicture}[scale =0.380,baseline=(current bounding box.north)]
\draw (0,0) rectangle (1,-1);
\draw (1,0) rectangle (2,-1);
\draw (2,0) rectangle (3,-1);
\draw (0,-1) rectangle (1,-2);
\draw (1,-1) rectangle (2,-2);
\draw (0.250,-0.250) rectangle ( 1.75 ,-0.750);
\draw (2.25,-0.250) rectangle ( 2.75 ,-0.750);
\draw (0.250,-1.25) rectangle ( 1.75 ,-1.75);
\end{tikzpicture}}
\quad \Bigg\vert\,
T_3 = 
\raisebox{16.0pt}{\begin{tikzpicture}[scale =0.380,baseline=(current bounding box.north)]
\draw (0,0) rectangle (1,-1);
\draw (1,0) rectangle (2,-1);
\draw (2,0) rectangle (3,-1);
\draw (3,0) rectangle (4,-1);
\draw (4,0) rectangle (5,-1);
\draw (0,-1) rectangle (1,-2);
\draw (1,-1) rectangle (2,-2);
\draw (2,-1) rectangle (3,-2);
\draw (0.250,-0.250) rectangle ( 1.75 ,-0.750);
\draw (2.25,-0.250) rectangle ( 4.75 ,-0.750);
\draw (0.250,-1.25) rectangle ( 1.75 ,-1.75);
\draw (2.25,-1.25) rectangle ( 2.75 ,-1.75);
\end{tikzpicture}}
\otimes \raisebox{16.0pt}{\begin{tikzpicture}[scale =0.380,baseline=(current bounding box.north)]
\draw (0,0) rectangle (1,-1);
\draw (1,0) rectangle (2,-1);
\draw (0,-1) rectangle (1,-2);
\draw (0.250,-0.250) rectangle ( 1.75 ,-0.750);
\draw (0.250,-1.25) rectangle ( 0.750 ,-1.75);
\end{tikzpicture}}
\otimes \raisebox{16.0pt}{\begin{tikzpicture}[scale =0.380,baseline=(current bounding box.north)]
\draw (0,0) rectangle (1,-1);
\draw (1,0) rectangle (2,-1);
\draw (2,0) rectangle (3,-1);
\draw (0,-1) rectangle (1,-2);
\draw (1,-1) rectangle (2,-2);
\draw (0.250,-0.250) rectangle ( 0.750 ,-0.750);
\draw (1.25,-0.250) rectangle ( 2.75 ,-0.750);
\draw (0.250,-1.25) rectangle ( 1.75 ,-1.75);
\end{tikzpicture}}
\]
\end{footnotesize}
We compute $\L(T_1) = 32$, $\L(T_2) = 8$ and $\L(T_3) =24$. The following are {all} six ordered simple polybrick tabloids of shape $\sigma$ and content $\tau$:
\begin{footnotesize}
\[
T_1' = 
\raisebox{16.0pt}{\begin{tikzpicture}[scale =0.450,baseline=(current bounding box.north)]
\draw (0,0) rectangle (1,-1);
\draw (1,0) rectangle (2,-1);
\draw (2,0) rectangle (3,-1);
\draw (3,0) rectangle (4,-1);
\draw (4,0) rectangle (5,-1);
\draw (0,-1) rectangle (1,-2);
\draw (1,-1) rectangle (2,-2);
\draw (2,-1) rectangle (3,-2);
\draw (0.350,-0.250) rectangle ( 2.55 ,-0.700);
\draw ( 2.80 ,-0.720) node[scale =0.600] {1};
\draw (3.35,-0.250) rectangle ( 4.55 ,-0.700);
\draw ( 4.80 ,-0.720) node[scale =0.600] {2};
\draw (0.350,-1.25) rectangle ( 1.55 ,-1.70);
\draw ( 1.80 ,-1.72) node[scale =0.600] {3};
\draw (2.25,-1.25) rectangle ( 2.65 ,-1.70);
\draw ( 2.84 ,-1.72) node[scale =0.600] {4};
\end{tikzpicture}}
\otimes \raisebox{16.0pt}{\begin{tikzpicture}[scale =0.45,baseline=(current bounding box.north)]
\draw (0,0) rectangle (1,-1);
\draw (1,0) rectangle (2,-1);
\draw (0,-1) rectangle (1,-2);
\draw (0.350,-0.250) rectangle ( 1.55 ,-0.700);
\draw ( 1.80 ,-0.720) node[scale =0.600] {5};
\draw (0.250,-1.25) rectangle ( 0.650 ,-1.70);
\draw ( 0.840 ,-1.72) node[scale =0.600] {6};
\end{tikzpicture}}
\otimes \raisebox{16.0pt}{\begin{tikzpicture}[scale =0.45,baseline=(current bounding box.north)]
\draw (0,0) rectangle (1,-1);
\draw (1,0) rectangle (2,-1);
\draw (2,0) rectangle (3,-1);
\draw (0,-1) rectangle (1,-2);
\draw (1,-1) rectangle (2,-2);
\draw (0.350,-0.250) rectangle ( 1.55 ,-0.700);
\draw ( 1.80 ,-0.720) node[scale =0.600] {7};
\draw (2.25,-0.250) rectangle ( 2.65 ,-0.700);
\draw ( 2.84 ,-0.720) node[scale =0.600] {9};
\draw (0.350,-1.25) rectangle ( 1.55 ,-1.70);
\draw ( 1.80 ,-1.72) node[scale =0.600] {8};
\end{tikzpicture}}
\quad \Bigg\vert\,
T_2' = \raisebox{16.0pt}{\begin{tikzpicture}[scale =0.45,baseline=(current bounding box.north)]
\draw (0,0) rectangle (1,-1);
\draw (1,0) rectangle (2,-1);
\draw (2,0) rectangle (3,-1);
\draw (3,0) rectangle (4,-1);
\draw (4,0) rectangle (5,-1);
\draw (0,-1) rectangle (1,-2);
\draw (1,-1) rectangle (2,-2);
\draw (2,-1) rectangle (3,-2);
\draw (0.350,-0.250) rectangle ( 2.55 ,-0.700);
\draw ( 2.80 ,-0.720) node[scale =0.600] {1};
\draw (3.35,-0.250) rectangle ( 4.55 ,-0.700);
\draw ( 4.80 ,-0.720) node[scale =0.600] {2};
\draw (0.350,-1.25) rectangle ( 1.55 ,-1.70);
\draw ( 1.80 ,-1.72) node[scale =0.600] {3};
\draw (2.25,-1.25) rectangle ( 2.65 ,-1.70);
\draw ( 2.84 ,-1.72) node[scale =0.600] {4};
\end{tikzpicture}}
\otimes \raisebox{16.0pt}{\begin{tikzpicture}[scale =0.45,baseline=(current bounding box.north)]
\draw (0,0) rectangle (1,-1);
\draw (1,0) rectangle (2,-1);
\draw (0,-1) rectangle (1,-2);
\draw (0.350,-0.250) rectangle ( 1.55 ,-0.700);
\draw ( 1.80 ,-0.720) node[scale =0.600] {5};
\draw (0.250,-1.25) rectangle ( 0.650 ,-1.70);
\draw ( 0.840 ,-1.72) node[scale =0.600] {6};
\end{tikzpicture}}
\otimes \raisebox{16.0pt}{\begin{tikzpicture}[scale =0.45,baseline=(current bounding box.north)]
\draw (0,0) rectangle (1,-1);
\draw (1,0) rectangle (2,-1);
\draw (2,0) rectangle (3,-1);
\draw (0,-1) rectangle (1,-2);
\draw (1,-1) rectangle (2,-2);
\draw (0.350,-0.250) rectangle ( 1.55 ,-0.700);
\draw ( 1.80 ,-0.720) node[scale =0.600] {8};
\draw (2.25,-0.250) rectangle ( 2.65 ,-0.700);
\draw ( 2.84 ,-0.720) node[scale =0.600] {9};
\draw (0.350,-1.25) rectangle ( 1.55 ,-1.70);
\draw ( 1.80 ,-1.72) node[scale =0.600] {7};
\end{tikzpicture}}
\]\[
T_3' = \raisebox{16.0pt}{\begin{tikzpicture}[scale =0.45,baseline=(current bounding box.north)]
\draw (0,0) rectangle (1,-1);
\draw (1,0) rectangle (2,-1);
\draw (2,0) rectangle (3,-1);
\draw (3,0) rectangle (4,-1);
\draw (4,0) rectangle (5,-1);
\draw (0,-1) rectangle (1,-2);
\draw (1,-1) rectangle (2,-2);
\draw (2,-1) rectangle (3,-2);
\draw (0.350,-0.250) rectangle ( 2.55 ,-0.700);
\draw ( 2.80 ,-0.720) node[scale =0.600] {1};
\draw (3.35,-0.250) rectangle ( 4.55 ,-0.700);
\draw ( 4.80 ,-0.720) node[scale =0.600] {3};
\draw (0.350,-1.25) rectangle ( 1.55 ,-1.70);
\draw ( 1.80 ,-1.72) node[scale =0.600] {2};
\draw (2.25,-1.25) rectangle ( 2.65 ,-1.70);
\draw ( 2.84 ,-1.72) node[scale =0.600] {4};
\end{tikzpicture}}
\otimes \raisebox{16.0pt}{\begin{tikzpicture}[scale =0.45,baseline=(current bounding box.north)]
\draw (0,0) rectangle (1,-1);
\draw (1,0) rectangle (2,-1);
\draw (0,-1) rectangle (1,-2);
\draw (0.350,-0.250) rectangle ( 1.55 ,-0.700);
\draw ( 1.80 ,-0.720) node[scale =0.600] {5};
\draw (0.250,-1.25) rectangle ( 0.650 ,-1.70);
\draw ( 0.840 ,-1.72) node[scale =0.600] {6};
\end{tikzpicture}}
\otimes \raisebox{16.0pt}{\begin{tikzpicture}[scale =0.45,baseline=(current bounding box.north)]
\draw (0,0) rectangle (1,-1);
\draw (1,0) rectangle (2,-1);
\draw (2,0) rectangle (3,-1);
\draw (0,-1) rectangle (1,-2);
\draw (1,-1) rectangle (2,-2);
\draw (0.350,-0.250) rectangle ( 1.55 ,-0.700);
\draw ( 1.80 ,-0.720) node[scale =0.600] {8};
\draw (2.25,-0.250) rectangle ( 2.65 ,-0.700);
\draw ( 2.84 ,-0.720) node[scale =0.600] {9};
\draw (0.350,-1.25) rectangle ( 1.55 ,-1.70);
\draw ( 1.80 ,-1.72) node[scale =0.600] {7};
\end{tikzpicture}}
\quad \Bigg\vert\,
T_4' = 
\raisebox{16.0pt}{\begin{tikzpicture}[scale =0.45,baseline=(current bounding box.north)]
\draw (0,0) rectangle (1,-1);
\draw (1,0) rectangle (2,-1);
\draw (2,0) rectangle (3,-1);
\draw (3,0) rectangle (4,-1);
\draw (4,0) rectangle (5,-1);
\draw (0,-1) rectangle (1,-2);
\draw (1,-1) rectangle (2,-2);
\draw (2,-1) rectangle (3,-2);
\draw (0.350,-0.250) rectangle ( 2.55 ,-0.700);
\draw ( 2.80 ,-0.720) node[scale =0.600] {1};
\draw (3.35,-0.250) rectangle ( 4.55 ,-0.700);
\draw ( 4.80 ,-0.720) node[scale =0.600] {3};
\draw (0.350,-1.25) rectangle ( 1.55 ,-1.70);
\draw ( 1.80 ,-1.72) node[scale =0.600] {2};
\draw (2.25,-1.25) rectangle ( 2.65 ,-1.70);
\draw ( 2.84 ,-1.72) node[scale =0.600] {4};
\end{tikzpicture}}
\otimes \raisebox{16.0pt}{\begin{tikzpicture}[scale =0.45,baseline=(current bounding box.north)]
\draw (0,0) rectangle (1,-1);
\draw (1,0) rectangle (2,-1);
\draw (0,-1) rectangle (1,-2);
\draw (0.350,-0.250) rectangle ( 1.55 ,-0.700);
\draw ( 1.80 ,-0.720) node[scale =0.600] {5};
\draw (0.250,-1.25) rectangle ( 0.650 ,-1.70);
\draw ( 0.840 ,-1.72) node[scale =0.600] {6};
\end{tikzpicture}}
\otimes \raisebox{16.0pt}{\begin{tikzpicture}[scale =0.45,baseline=(current bounding box.north)]
\draw (0,0) rectangle (1,-1);
\draw (1,0) rectangle (2,-1);
\draw (2,0) rectangle (3,-1);
\draw (0,-1) rectangle (1,-2);
\draw (1,-1) rectangle (2,-2);
\draw (0.350,-0.250) rectangle ( 1.55 ,-0.700);
\draw ( 1.80 ,-0.720) node[scale =0.600] {7};
\draw (2.25,-0.250) rectangle ( 2.65 ,-0.700);
\draw ( 2.84 ,-0.720) node[scale =0.600] {9};
\draw (0.350,-1.25) rectangle ( 1.55 ,-1.70);
\draw ( 1.80 ,-1.72) node[scale =0.600] {8};
\end{tikzpicture}}
\]
\[
T_5' = \raisebox{16.0pt}{\begin{tikzpicture}[scale =0.45,baseline=(current bounding box.north)]
\draw (0,0) rectangle (1,-1);
\draw (1,0) rectangle (2,-1);
\draw (2,0) rectangle (3,-1);
\draw (3,0) rectangle (4,-1);
\draw (4,0) rectangle (5,-1);
\draw (0,-1) rectangle (1,-2);
\draw (1,-1) rectangle (2,-2);
\draw (2,-1) rectangle (3,-2);
\draw (0.350,-0.250) rectangle ( 1.55 ,-0.700);
\draw ( 1.80 ,-0.720) node[scale =0.600] {2};
\draw (2.35,-0.250) rectangle ( 3.55 ,-0.700);
\draw ( 3.80 ,-0.720) node[scale =0.600] {3};
\draw (4.25,-0.250) rectangle ( 4.65 ,-0.700);
\draw ( 4.84 ,-0.720) node[scale =0.600] {4};
\draw (0.350,-1.25) rectangle ( 2.55 ,-1.70);
\draw ( 2.80 ,-1.72) node[scale =0.600] {1};
\end{tikzpicture}}
\otimes \raisebox{16.0pt}{\begin{tikzpicture}[scale =0.45,baseline=(current bounding box.north)]
\draw (0,0) rectangle (1,-1);
\draw (1,0) rectangle (2,-1);
\draw (0,-1) rectangle (1,-2);
\draw (0.350,-0.250) rectangle ( 1.55 ,-0.700);
\draw ( 1.80 ,-0.720) node[scale =0.600] {5};
\draw (0.250,-1.25) rectangle ( 0.650 ,-1.70);
\draw ( 0.840 ,-1.72) node[scale =0.600] {6};
\end{tikzpicture}}
\otimes \raisebox{16.0pt}{\begin{tikzpicture}[scale =0.45,baseline=(current bounding box.north)]
\draw (0,0) rectangle (1,-1);
\draw (1,0) rectangle (2,-1);
\draw (2,0) rectangle (3,-1);
\draw (0,-1) rectangle (1,-2);
\draw (1,-1) rectangle (2,-2);
\draw (0.350,-0.250) rectangle ( 1.55 ,-0.700);
\draw ( 1.80 ,-0.720) node[scale =0.600] {7};
\draw (2.25,-0.250) rectangle ( 2.65 ,-0.700);
\draw ( 2.84 ,-0.720) node[scale =0.600] {9};
\draw (0.350,-1.25) rectangle ( 1.55 ,-1.70);
\draw ( 1.80 ,-1.72) node[scale =0.600] {8};
\end{tikzpicture}}
\quad \Bigg\vert\,
T_6' = 
\raisebox{16.0pt}{\begin{tikzpicture}[scale =0.45,baseline=(current bounding box.north)]
\draw (0,0) rectangle (1,-1);
\draw (1,0) rectangle (2,-1);
\draw (2,0) rectangle (3,-1);
\draw (3,0) rectangle (4,-1);
\draw (4,0) rectangle (5,-1);
\draw (0,-1) rectangle (1,-2);
\draw (1,-1) rectangle (2,-2);
\draw (2,-1) rectangle (3,-2);
\draw (0.350,-0.250) rectangle ( 1.55 ,-0.700);
\draw ( 1.80 ,-0.720) node[scale =0.600] {2};
\draw (2.35,-0.250) rectangle ( 3.55 ,-0.700);
\draw ( 3.80 ,-0.720) node[scale =0.600] {3};
\draw (4.25,-0.250) rectangle ( 4.65 ,-0.700);
\draw ( 4.84 ,-0.720) node[scale =0.600] {4};
\draw (0.350,-1.25) rectangle ( 2.55 ,-1.70);
\draw ( 2.80 ,-1.72) node[scale =0.600] {1};
\end{tikzpicture}}
\otimes \raisebox{16.0pt}{\begin{tikzpicture}[scale =0.45,baseline=(current bounding box.north)]
\draw (0,0) rectangle (1,-1);
\draw (1,0) rectangle (2,-1);
\draw (0,-1) rectangle (1,-2);
\draw (0.350,-0.250) rectangle ( 1.55 ,-0.700);
\draw ( 1.80 ,-0.720) node[scale =0.600] {5};
\draw (0.250,-1.25) rectangle ( 0.650 ,-1.70);
\draw ( 0.840 ,-1.72) node[scale =0.600] {6};
\end{tikzpicture}}
\otimes \raisebox{16.0pt}{\begin{tikzpicture}[scale =0.45,baseline=(current bounding box.north)]
\draw (0,0) rectangle (1,-1);
\draw (1,0) rectangle (2,-1);
\draw (2,0) rectangle (3,-1);
\draw (0,-1) rectangle (1,-2);
\draw (1,-1) rectangle (2,-2);
\draw (0.350,-0.250) rectangle ( 1.55 ,-0.700);
\draw ( 1.80 ,-0.720) node[scale =0.600] {8};
\draw (2.25,-0.250) rectangle ( 2.65 ,-0.700);
\draw ( 2.84 ,-0.720) node[scale =0.600] {9};
\draw (0.350,-1.25) rectangle ( 1.55 ,-1.70);
\draw ( 1.80 ,-1.72) node[scale =0.600] {7};
\end{tikzpicture}}
\]
\[T_4' = \yt{11122,334}\otimes\yt{55,6}\otimes\yt{889,77}\quad \Bigg\vert\quad T_5' = \yt{11133,224}\otimes\yt{55,6}\otimes\yt{889,77} \quad \Bigg\vert\quad T_6' = \yt{22334,111}\otimes\yt{55,6}\otimes\yt{889,77}\]
\end{footnotesize}
\end{example} 
\begin{theorem}\label{thm:HE}
For $\sigma, \tau \in \Typ(n)$,
\begin{enumerate}
\item The coefficient of $E_{\tau}$ in $H_{\sigma}$ is $(-1)^{\ell(\tau)} |\pt^{\mathrm{simp}}(\sigma, \tau)|$.
\item The coefficient of $H_{\tau}$ in $E_{\sigma}$ is $(-1)^{\ell(\tau)} |\pt^{\mathrm{simp}}(\sigma, \tau)|$.
\end{enumerate}
\end{theorem}
\begin{proof}
We prove (1), and the proof of (2) is the same. Each $d^{r}\in \sigma$ (with repetition) corresponds to a row of length $d$ in the $r$th tensor factor $r$ of $\dg^\otimes(\sigma)$. We pick a square-free polycomposition $\delta = (\alpha)^1$ of $d$ and place bricks of lengths $\alpha_1$, $\alpha_2,\ldots$ from left to right in the row. This construction corresponds to the term $(-1)^{\ell(\delta)}E_{\delta^r}$ in the $E$-expansion of $H_{d^r}$, and the exponent of $-1$ is the number of bricks in the row. We perform the same procedure for all $d^r\in \sigma$. We construct the type $\tau$ defined by partitions $\tau|^r$ formed by the lengths of bricks in $\sigma|^r$ for $r\geq 1$. This gives us a simple polybrick tabloid $T$ of shape $\sigma$ and content $\tau$. With this construction we associate the term $(-1)^{\ell(\tau)}E_\tau$ where $\ell(\tau)$ is the total number of bricks in $T$. Constructing all simple polybrick tabloids of shape $\sigma$ and content $\tau$ gives us the intended coefficient.
\end{proof}
\begin{example}
The coefficient of $H_{(2,1,1,1)^1(2,1,1)^2}$ in $E_{(3,2)^1(2,2)^2}$ is $-6$ which can be computed using the following six simple polybrick tabloids and by observing that $\ell((2,1,1,1)^1(2,1,1)^2) = 7$:
\boks{0.32}
\begin{footnotesize}
\[T_1 = \raisebox{16.0pt}{\begin{tikzpicture}[scale =0.400,baseline=(current bounding box.north)]
\draw (0,0) rectangle (1,-1);
\draw (1,0) rectangle (2,-1);
\draw (2,0) rectangle (3,-1);
\draw (0,-1) rectangle (1,-2);
\draw (1,-1) rectangle (2,-2);
\draw (0.250,-0.250) rectangle ( 1.75 ,-0.750);
\draw (2.25,-0.250) rectangle ( 2.75 ,-0.750);
\draw (0.250,-1.25) rectangle ( 0.750 ,-1.75);
\draw (1.25,-1.25) rectangle ( 1.75 ,-1.75);
\end{tikzpicture}}
\otimes \raisebox{16.0pt}{\begin{tikzpicture}[scale =0.400,baseline=(current bounding box.north)]
\draw (0,0) rectangle (1,-1);
\draw (1,0) rectangle (2,-1);
\draw (0,-1) rectangle (1,-2);
\draw (1,-1) rectangle (2,-2);
\draw (0.250,-0.250) rectangle ( 1.75 ,-0.750);
\draw (0.250,-1.25) rectangle ( 0.750 ,-1.75);
\draw (1.25,-1.25) rectangle ( 1.75 ,-1.75);
\end{tikzpicture}}
\quad\Bigg\vert\,
T_2=\raisebox{16.0pt}{\begin{tikzpicture}[scale =0.400,baseline=(current bounding box.north)]
\draw (0,0) rectangle (1,-1);
\draw (1,0) rectangle (2,-1);
\draw (2,0) rectangle (3,-1);
\draw (0,-1) rectangle (1,-2);
\draw (1,-1) rectangle (2,-2);
\draw (0.250,-0.250) rectangle ( 0.750 ,-0.750);
\draw (1.25,-0.250) rectangle ( 2.75 ,-0.750);
\draw (0.250,-1.25) rectangle ( 0.750 ,-1.75);
\draw (1.25,-1.25) rectangle ( 1.75 ,-1.75);
\end{tikzpicture}}
\otimes \raisebox{16.0pt}{\begin{tikzpicture}[scale =0.400,baseline=(current bounding box.north)]
\draw (0,0) rectangle (1,-1);
\draw (1,0) rectangle (2,-1);
\draw (0,-1) rectangle (1,-2);
\draw (1,-1) rectangle (2,-2);
\draw (0.250,-0.250) rectangle ( 1.75 ,-0.750);
\draw (0.250,-1.25) rectangle ( 0.750 ,-1.75);
\draw (1.25,-1.25) rectangle ( 1.75 ,-1.75);
\end{tikzpicture}}
\quad\Bigg\vert\,
T_3 = \raisebox{16.0pt}{\begin{tikzpicture}[scale =0.400,baseline=(current bounding box.north)]
\draw (0,0) rectangle (1,-1);
\draw (1,0) rectangle (2,-1);
\draw (2,0) rectangle (3,-1);
\draw (0,-1) rectangle (1,-2);
\draw (1,-1) rectangle (2,-2);
\draw (0.250,-0.250) rectangle ( 1.75 ,-0.750);
\draw (2.25,-0.250) rectangle ( 2.75 ,-0.750);
\draw (0.250,-1.25) rectangle ( 0.750 ,-1.75);
\draw (1.25,-1.25) rectangle ( 1.75 ,-1.75);
\end{tikzpicture}}
\otimes \raisebox{16.0pt}{\begin{tikzpicture}[scale =0.400,baseline=(current bounding box.north)]
\draw (0,0) rectangle (1,-1);
\draw (1,0) rectangle (2,-1);
\draw (0,-1) rectangle (1,-2);
\draw (1,-1) rectangle (2,-2);
\draw (0.250,-0.250) rectangle ( 0.750 ,-0.750);
\draw (1.25,-0.250) rectangle ( 1.75 ,-0.750);
\draw (0.250,-1.25) rectangle ( 1.75 ,-1.75);
\end{tikzpicture}}
\]\[
T_4 =
\raisebox{16.0pt}{\begin{tikzpicture}[scale =0.400,baseline=(current bounding box.north)]
\draw (0,0) rectangle (1,-1);
\draw (1,0) rectangle (2,-1);
\draw (2,0) rectangle (3,-1);
\draw (0,-1) rectangle (1,-2);
\draw (1,-1) rectangle (2,-2);
\draw (0.250,-0.250) rectangle ( 0.750 ,-0.750);
\draw (1.25,-0.250) rectangle ( 2.75 ,-0.750);
\draw (0.250,-1.25) rectangle ( 0.750 ,-1.75);
\draw (1.25,-1.25) rectangle ( 1.75 ,-1.75);
\end{tikzpicture}}
\otimes \raisebox{16.0pt}{\begin{tikzpicture}[scale =0.400,baseline=(current bounding box.north)]
\draw (0,0) rectangle (1,-1);
\draw (1,0) rectangle (2,-1);
\draw (0,-1) rectangle (1,-2);
\draw (1,-1) rectangle (2,-2);
\draw (0.250,-0.250) rectangle ( 0.750 ,-0.750);
\draw (1.25,-0.250) rectangle ( 1.75 ,-0.750);
\draw (0.250,-1.25) rectangle ( 1.75 ,-1.75);
\end{tikzpicture}}
\quad\Bigg\vert\,
T_5=
\raisebox{16.0pt}{\begin{tikzpicture}[scale =0.400,baseline=(current bounding box.north)]
\draw (0,0) rectangle (1,-1);
\draw (1,0) rectangle (2,-1);
\draw (2,0) rectangle (3,-1);
\draw (0,-1) rectangle (1,-2);
\draw (1,-1) rectangle (2,-2);
\draw (0.250,-0.250) rectangle ( 0.750 ,-0.750);
\draw (1.25,-0.250) rectangle ( 1.75 ,-0.750);
\draw (2.25,-0.250) rectangle ( 2.75 ,-0.750);
\draw (0.250,-1.25) rectangle ( 1.75 ,-1.75);
\end{tikzpicture}}
\otimes \raisebox{16.0pt}{\begin{tikzpicture}[scale =0.400,baseline=(current bounding box.north)]
\draw (0,0) rectangle (1,-1);
\draw (1,0) rectangle (2,-1);
\draw (0,-1) rectangle (1,-2);
\draw (1,-1) rectangle (2,-2);
\draw (0.250,-0.250) rectangle ( 1.75 ,-0.750);
\draw (0.250,-1.25) rectangle ( 0.750 ,-1.75);
\draw (1.25,-1.25) rectangle ( 1.75 ,-1.75);
\end{tikzpicture}}
\quad \Bigg\vert\,
T_6 = 
\raisebox{16.0pt}{\begin{tikzpicture}[scale =0.400,baseline=(current bounding box.north)]
\draw (0,0) rectangle (1,-1);
\draw (1,0) rectangle (2,-1);
\draw (2,0) rectangle (3,-1);
\draw (0,-1) rectangle (1,-2);
\draw (1,-1) rectangle (2,-2);
\draw (0.250,-0.250) rectangle ( 0.750 ,-0.750);
\draw (1.25,-0.250) rectangle ( 1.75 ,-0.750);
\draw (2.25,-0.250) rectangle ( 2.75 ,-0.750);
\draw (0.250,-1.25) rectangle ( 1.75 ,-1.75);
\end{tikzpicture}}
\otimes \raisebox{16.0pt}{\begin{tikzpicture}[scale =0.400,baseline=(current bounding box.north)]
\draw (0,0) rectangle (1,-1);
\draw (1,0) rectangle (2,-1);
\draw (0,-1) rectangle (1,-2);
\draw (1,-1) rectangle (2,-2);
\draw (0.250,-0.250) rectangle ( 0.750 ,-0.750);
\draw (1.25,-0.250) rectangle ( 1.75 ,-0.750);
\draw (0.250,-1.25) rectangle ( 1.75 ,-1.75);
\end{tikzpicture}}
\]
\end{footnotesize}
\end{example}

\begin{theorem}\label{thm:P-HE}
For $\sigma, \tau \in \Typ(n)$,
\begin{enumerate}
\item The coefficient of $H_{\tau}$ in $P_{\sigma}$ is $\sum_{T\in \pt^{\mathrm{simp}}(\sigma, \tau)}(-1)^{\ell(\tau) - \ell(\sigma)} \L(T)$.
\item The coefficient of $E_{\tau}$ in $P_{\sigma}$ is $\sum_{T\in \pt^{\mathrm{simp}}(\sigma, \tau)}(-1)^{\ell(\tau)} \L(T)$.
\end{enumerate}
\end{theorem}
\begin{proof}
Each $d^r\in \sigma$ corresponds to a row of length $d$ in the $r$th tensor factor of $\dg^\otimes(\sigma)$. A choice of a square-free polycomposition $\delta = (\alpha)^1$ corresponds to placing bricks of lengths $\alpha_1$, $\alpha_2,\ldots$ from left to right. This corresponds to the term $(-1)^{\ell(\delta)-1}L(\delta) H_{\delta^r}$ arising from the $H$-expansion of $P_{d^r}$. Continue this process for all $d^r\in \sigma$ to obtain a simple polybrick tabloid $T$ where the shape is $\sigma$ and the content is the type $\tau$ such that $\tau|^i$ is the partition formed by the bricks placed in $\dg(\sigma|^i)$. Taking the product of the terms above gives us $(-1)^{\ell(\tau)-\ell(\sigma)} \L(T) H_\tau$. The factor $\L(T)$ arises by taking a product over all weights $L(\delta)$. Summing the coefficients over all simple polybrick tabloids of shape $\sigma$ and content $\tau$ proves (1). The proof for (2) follows similarly.
\end{proof}
\begin{example}
The coefficient of $H_{(3,2,1)^1(2,1)^2}$ in $P_{(3,3)^1(2,1)^2}$ is $-36$ which can be computed using the following simple polybrick tabloids with $\L(T_1) = \L(T_2) = 6$ and $\L(T_3) = \L(T_4) = 12$, and by observing $\ell(\tau) - \ell(\sigma) = 1$:
\begin{footnotesize}
\[
T_1 =
\raisebox{16.0pt}{\begin{tikzpicture}[scale =0.400,baseline=(current bounding box.north)]
\draw (0,0) rectangle (1,-1);
\draw (1,0) rectangle (2,-1);
\draw (2,0) rectangle (3,-1);
\draw (0,-1) rectangle (1,-2);
\draw (1,-1) rectangle (2,-2);
\draw (2,-1) rectangle (3,-2);
\draw (0.250,-0.250) rectangle ( 2.75 ,-0.750);
\draw (0.250,-1.25) rectangle ( 1.75 ,-1.75);
\draw (2.25,-1.25) rectangle ( 2.75 ,-1.75);
\end{tikzpicture}}
\otimes \raisebox{16.0pt}{\begin{tikzpicture}[scale =0.400,baseline=(current bounding box.north)]
\draw (0,0) rectangle (1,-1);
\draw (1,0) rectangle (2,-1);
\draw (0,-1) rectangle (1,-2);
\draw (0.250,-0.250) rectangle ( 1.75 ,-0.750);
\draw (0.250,-1.25) rectangle ( 0.750 ,-1.75);
\end{tikzpicture}}
\quad\Bigg\vert\,
T_2 =
\raisebox{16.0pt}{\begin{tikzpicture}[scale =0.400,baseline=(current bounding box.north)]
\draw (0,0) rectangle (1,-1);
\draw (1,0) rectangle (2,-1);
\draw (2,0) rectangle (3,-1);
\draw (0,-1) rectangle (1,-2);
\draw (1,-1) rectangle (2,-2);
\draw (2,-1) rectangle (3,-2);
\draw (0.250,-0.250) rectangle ( 1.75 ,-0.750);
\draw (2.25,-0.250) rectangle ( 2.75 ,-0.750);
\draw (0.250,-1.25) rectangle ( 2.75 ,-1.75);
\end{tikzpicture}}
\otimes \raisebox{16.0pt}{\begin{tikzpicture}[scale =0.400,baseline=(current bounding box.north)]
\draw (0,0) rectangle (1,-1);
\draw (1,0) rectangle (2,-1);
\draw (0,-1) rectangle (1,-2);
\draw (0.250,-0.250) rectangle ( 1.75 ,-0.750);
\draw (0.250,-1.25) rectangle ( 0.750 ,-1.75);
\end{tikzpicture}}
\quad\Bigg\vert\,
T_3 =
\raisebox{16.0pt}{\begin{tikzpicture}[scale =0.400,baseline=(current bounding box.north)]
\draw (0,0) rectangle (1,-1);
\draw (1,0) rectangle (2,-1);
\draw (2,0) rectangle (3,-1);
\draw (0,-1) rectangle (1,-2);
\draw (1,-1) rectangle (2,-2);
\draw (2,-1) rectangle (3,-2);
\draw (0.250,-0.250) rectangle ( 2.75 ,-0.750);
\draw (0.250,-1.25) rectangle ( 0.750 ,-1.75);
\draw (1.25,-1.25) rectangle ( 2.75 ,-1.75);
\end{tikzpicture}}
\otimes \raisebox{16.0pt}{\begin{tikzpicture}[scale =0.400,baseline=(current bounding box.north)]
\draw (0,0) rectangle (1,-1);
\draw (1,0) rectangle (2,-1);
\draw (0,-1) rectangle (1,-2);
\draw (0.250,-0.250) rectangle ( 1.75 ,-0.750);
\draw (0.250,-1.25) rectangle ( 0.750 ,-1.75);
\end{tikzpicture}}
\quad\Bigg\vert\,
T_4 =
\raisebox{16.0pt}{\begin{tikzpicture}[scale =0.400,baseline=(current bounding box.north)]
\draw (0,0) rectangle (1,-1);
\draw (1,0) rectangle (2,-1);
\draw (2,0) rectangle (3,-1);
\draw (0,-1) rectangle (1,-2);
\draw (1,-1) rectangle (2,-2);
\draw (2,-1) rectangle (3,-2);
\draw (0.250,-0.250) rectangle ( 0.750 ,-0.750);
\draw (1.25,-0.250) rectangle ( 2.75 ,-0.750);
\draw (0.250,-1.25) rectangle ( 2.75 ,-1.75);
\end{tikzpicture}}
\otimes \raisebox{16.0pt}{\begin{tikzpicture}[scale =0.400,baseline=(current bounding box.north)]
\draw (0,0) rectangle (1,-1);
\draw (1,0) rectangle (2,-1);
\draw (0,-1) rectangle (1,-2);
\draw (0.250,-0.250) rectangle ( 1.75 ,-0.750);
\draw (0.250,-1.25) rectangle ( 0.750 ,-1.75);
\end{tikzpicture}}
\]
\end{footnotesize}
\end{example}

Recall the definition $z_\lambda = \prod_{i\geq 1} i^{m_i(\lambda)}m_i(\lambda)!$ where $m_i(\lambda)$ is the number of times a part equal to $i$ appears in $\lambda$. We define its polysymmetric analog $z_\tau^\otimes = \prod_{i\geq 1} z_{\tau|^i}$ for any type $\tau$. Also recall for $\alpha = (\alpha_1, \alpha_2, \ldots, \alpha_k)$, we have 
\[
Z_\alpha = (\alpha_1)(\alpha_1+\alpha_2)\ldots (\alpha_1+\alpha_2+\ldots+\alpha_k).
\]
For a square-free polycompositions $\delta$, $Z_\delta = Z_\alpha$.
\begin{theorem}\label{thm:H-P,E-P}
For $\sigma, \tau\in \Typ(n)$,
\begin{enumerate}
\item The coefficient of $P_{\tau}$ in $H_{\sigma}$ is $|\pt^{\mathrm{osimp}}(\sigma,\tau)|/z^\otimes_\tau$.
\item The coefficient of $P_{\tau}$ in $E_{\sigma}$ is $(-1)^{\ell(\tau)}|\pt^{\mathrm{osimp}}(\sigma,\tau)|/z^\otimes_\tau$.
\end{enumerate}
\end{theorem}
\begin{proof}
We prove (1) for square-free types $\sigma = (d_1,\ldots, d_s)^1$. For each $d$ in $\sigma$, $H_{d} = \sum\limits_{\delta \in \PCom_{\sqf}(d)} \dfrac{P_{\delta}}{Z_\delta}$ from which it follows that $H_d = \sum\limits_{\mu\in \Par(d)} \sum\limits_{\substack{\alpha\in \Com(d)\\ \sort(\alpha) = \mu}} \dfrac{P_{(\alpha)^1}}{Z_{\alpha}}$. From Lemma \ref{lem:z-harmonic}, we obtain $H_{d} = \sum\limits_{\mu\in \Par(d)} \dfrac{P_{(\mu)^1}}{z_\mu}$. For each $d_i$ in $\sigma$, choose partitions $\mu^{(i)}\in \Par(d_i)$ and construct an ordered brick tabloid of shape $\lambda$ and content $\mu$ is the multiset union $\bigcup_{i=1}^{s}\mu^{(i)}$ such that for two bricks of the same length with labels $l>l'$, the brick with label $l$ appears weakly below the brick with label $l'$. Call this ordered brick tabloid $T_{\lambda,\mu}$. It is routine to check that it is unique. From $T_{\lambda, \mu}$, we can generate other ordered brick tabloids of shape $\lambda$ and content $\mu$. For each $i$, we permute the bricks of length $i$ under the constraint that the labels strictly increase in a given row. The bricks of length $i$ can be permuted in $m_i(\mu)!$ ways. We divide this by $m_i(\mu^{(1)})! m_i(\mu^{(2)})! \ldots$  to ensure that within each row the bricks are in strictly increasing order of their labels. We have 
\[
\prod\limits_{i\geq 1} \dfrac{m_i(\mu)!}{ m_i(\mu^{(1)})! m_i(\mu^{(2)})!\ldots} = \text{number of ordered brick tabloids of shape } \lambda \text{ and content } \mu.
\]
Multiplying and dividing by a factor of $\prod\limits_{i\geq 1} i^{m_i(\mu^{(1)})} i^{m_i(\mu^{(2)})} \ldots i^{m_i(\mu^{(s)})}$ gives
\[
\dfrac{z_\mu}{z_{\mu^{(1)}}z_{\mu^{(2)}}\ldots z_{\mu^{(s)}}} = \text{number of ordered brick tabloids of shape } \lambda \text{ and content } \mu.
\]
From the above expansion of $H_d$, we have 
\[
H_{(d_1, \ldots,d_s)^1} = \sum\limits_{\substack{\mu = \bigcup_{i=1}^{s}\mu^{(i)}\\ \mu^{(j)}\in \Par(d_j)}} \dfrac{P_{(\mu)^1}}{z_{\mu^{(1)}}z_{\mu^{(2)}}\ldots z_{\mu^{(s)}}}
\] 
From our above discussion, we deduce that for square-free types $\sigma$,  \[H_\sigma = \sum_{\tau \in \PCom_{\sqf}(|\sigma|)} |\pt^{\mathrm{osimp}}(\sigma,\tau)|/z^\otimes_\tau.\] 
Here, we also use that for a square-free permutation $\sigma = (\lambda)^1$, we have $z_\sigma^\otimes = z_{\lambda}$.
For a general type $\sigma$, we can construct an ordered brick tabloid of shape $\sigma|^r$ and content $\tau|^r$ for $r\geq 1$ by following the above discussion. This gives us (1). The quantity $\ell(\tau)$ is the number of bricks and the sign in (2) is obtained by assigning $-1$ to each brick used and multiplying the sign for each brick.
\end{proof}
\begin{example}
Let $\tau = (2,2,1)^1(1,1)^2$ and $\sigma = (3,2)^1(1,1)^2$. We first compute $\ell(\tau) = 5$ and $z_{\tau}^\otimes = z_{2,2,1}z_{1,1} = (2^2\cdot 2! \cdot 1\cdot 1!) \cdot (1^2\cdot 2!) = 16$. We find $|\pt^{\mathrm{osimp}}(\sigma,\tau)| = 4$ by listing the following four polybrick tabloids.
\[
T_1 =
\raisebox{16.0pt}{\begin{tikzpicture}[scale =0.400,baseline=(current bounding box.north)]
\draw (0,0) rectangle (1,-1);
\draw (1,0) rectangle (2,-1);
\draw (2,0) rectangle (3,-1);
\draw (0,-1) rectangle (1,-2);
\draw (1,-1) rectangle (2,-2);
\draw (0.350,-0.250) rectangle ( 1.55 ,-0.700);
\draw ( 1.80 ,-0.720) node[scale =0.600] {1};
\draw (2.25,-0.250) rectangle ( 2.65 ,-0.700);
\draw ( 2.84 ,-0.720) node[scale =0.600] {3};
\draw (0.350,-1.25) rectangle ( 1.55 ,-1.70);
\draw ( 1.80 ,-1.72) node[scale =0.600] {2};
\end{tikzpicture}}
\otimes \raisebox{16.0pt}{\begin{tikzpicture}[scale =0.400,baseline=(current bounding box.north)]
\draw (0,0) rectangle (1,-1);
\draw (0,-1) rectangle (1,-2);
\draw (0.250,-0.250) rectangle ( 0.650 ,-0.700);
\draw ( 0.840 ,-0.720) node[scale =0.600] {4};
\draw (0.250,-1.25) rectangle ( 0.650 ,-1.70);
\draw ( 0.840 ,-1.72) node[scale =0.600] {5};
\end{tikzpicture}}
\quad\Bigg\vert\,
T_2 =
\raisebox{16.0pt}{\begin{tikzpicture}[scale =0.400,baseline=(current bounding box.north)]
\draw (0,0) rectangle (1,-1);
\draw (1,0) rectangle (2,-1);
\draw (2,0) rectangle (3,-1);
\draw (0,-1) rectangle (1,-2);
\draw (1,-1) rectangle (2,-2);
\draw (0.350,-0.250) rectangle ( 1.55 ,-0.700);
\draw ( 1.80 ,-0.720) node[scale =0.600] {1};
\draw (2.25,-0.250) rectangle ( 2.65 ,-0.700);
\draw ( 2.84 ,-0.720) node[scale =0.600] {3};
\draw (0.350,-1.25) rectangle ( 1.55 ,-1.70);
\draw ( 1.80 ,-1.72) node[scale =0.600] {2};
\end{tikzpicture}}
\otimes \raisebox{16.0pt}{\begin{tikzpicture}[scale =0.400,baseline=(current bounding box.north)]
\draw (0,0) rectangle (1,-1);
\draw (0,-1) rectangle (1,-2);
\draw (0.250,-0.250) rectangle ( 0.650 ,-0.700);
\draw ( 0.840 ,-0.720) node[scale =0.600] {5};
\draw (0.250,-1.25) rectangle ( 0.650 ,-1.70);
\draw ( 0.840 ,-1.72) node[scale =0.600] {4};
\end{tikzpicture}}
\quad\Bigg\vert\,
T_3 =
\raisebox{16.0pt}{\begin{tikzpicture}[scale =0.400,baseline=(current bounding box.north)]
\draw (0,0) rectangle (1,-1);
\draw (1,0) rectangle (2,-1);
\draw (2,0) rectangle (3,-1);
\draw (0,-1) rectangle (1,-2);
\draw (1,-1) rectangle (2,-2);
\draw (0.350,-0.250) rectangle ( 1.55 ,-0.700);
\draw ( 1.80 ,-0.720) node[scale =0.600] {2};
\draw (2.25,-0.250) rectangle ( 2.65 ,-0.700);
\draw ( 2.84 ,-0.720) node[scale =0.600] {3};
\draw (0.350,-1.25) rectangle ( 1.55 ,-1.70);
\draw ( 1.80 ,-1.72) node[scale =0.600] {1};
\end{tikzpicture}}
\otimes \raisebox{16.0pt}{\begin{tikzpicture}[scale =0.400,baseline=(current bounding box.north)]
\draw (0,0) rectangle (1,-1);
\draw (0,-1) rectangle (1,-2);
\draw (0.250,-0.250) rectangle ( 0.650 ,-0.700);
\draw ( 0.840 ,-0.720) node[scale =0.600] {4};
\draw (0.250,-1.25) rectangle ( 0.650 ,-1.70);
\draw ( 0.840 ,-1.72) node[scale =0.600] {5};
\end{tikzpicture}}
\quad\Bigg\vert\,
T_2 =
\raisebox{16.0pt}{\begin{tikzpicture}[scale =0.400,baseline=(current bounding box.north)]
\draw (0,0) rectangle (1,-1);
\draw (1,0) rectangle (2,-1);
\draw (2,0) rectangle (3,-1);
\draw (0,-1) rectangle (1,-2);
\draw (1,-1) rectangle (2,-2);
\draw (0.350,-0.250) rectangle ( 1.55 ,-0.700);
\draw ( 1.80 ,-0.720) node[scale =0.600] {2};
\draw (2.25,-0.250) rectangle ( 2.65 ,-0.700);
\draw ( 2.84 ,-0.720) node[scale =0.600] {3};
\draw (0.350,-1.25) rectangle ( 1.55 ,-1.70);
\draw ( 1.80 ,-1.72) node[scale =0.600] {1};
\end{tikzpicture}}
\otimes \raisebox{16.0pt}{\begin{tikzpicture}[scale =0.400,baseline=(current bounding box.north)]
\draw (0,0) rectangle (1,-1);
\draw (0,-1) rectangle (1,-2);
\draw (0.250,-0.250) rectangle ( 0.650 ,-0.700);
\draw ( 0.840 ,-0.720) node[scale =0.600] {5};
\draw (0.250,-1.25) rectangle ( 0.650 ,-1.70);
\draw ( 0.840 ,-1.72) node[scale =0.600] {4};
\end{tikzpicture}}
\]

So, the coefficient of $P_\tau$ in $E_\sigma$ is $(-1)^5 4/16 = -1/4$.
\end{example}


\subsection{Double polybrick tabloids}
Now we deal with the $H$, $E$, and $P$ expansions of $E^+_\tau$ for $\tau\in \Typ$. The combinatorial objects that appear in these expansions are tiled by two kinds of bricks.
Define a \textit{doublebrick} to be a brick of even length that is marked with a $+$ sign in the superscript. A \textit{double polybrick tabloid of shape $\sigma$ and content $\tau = d_1^{r_1}\ldots d_s^{r_s}$} with $r_1\leq r_2\leq \ldots \leq r_s$ is a tiling of $\dg^\otimes(\sigma)$ using bricks and doublebricks such that
\begin{itemize}
\item  the doublebricks appear to the right of all bricks in a given row
\item  for each $1\leq i \leq s$, we either place a brick of length $d_i$ in the $r_i$th tensor factor, or a doublebrick of length $2d_i$ in the tensor factor $r_i/2$. The latter is only possible when $r$ is even.
\end{itemize}
We denote the set of double polybrick tabloids of shape $\sigma$ and content $\tau$ by $\pt^{\mathrm{doub}}(\sigma, \tau)$. For $T\in \pt^{\mathrm{doub}}(\sigma, \tau)$, define $\ell_1(T)$ to be the number of bricks in $T$ and $\ell_2(T)$ to be the number of doublebricks in $T$. 
\begin{example}
The following is an example of a double polybrick tabloid of shape $\sigma = (5,2)^1(5,3)^2$ and content $(2,2,1)^1(3,1,1)^2(2)^4$ with $\ell_1(T) = 5$ and $\ell_2(T) = 2$:
\[
T=\raisebox{16.0pt}{\begin{tikzpicture}[scale =0.400,baseline=(current bounding box.north)]
\draw (0,0) rectangle (1,-1);
\draw (1,0) rectangle (2,-1);
\draw (2,0) rectangle (3,-1);
\draw (3,0) rectangle (4,-1);
\draw (4,0) rectangle (5,-1);
\draw (0,-1) rectangle (1,-2);
\draw (1,-1) rectangle (2,-2);
\draw (0.250,-0.350) rectangle ( 0.650 ,-0.750);
\draw (1.35,-0.350) rectangle ( 2.55 ,-0.750);
\draw (3.35,-0.350) rectangle ( 4.55 ,-0.750);
\draw ( 4.75 ,-0.300) node {\tiny +};
\draw (0.350,-1.35) rectangle ( 1.55 ,-1.75);
\end{tikzpicture}}
\otimes \raisebox{16.0pt}{\begin{tikzpicture}[scale =0.400,baseline=(current bounding box.north)]
\draw (0,0) rectangle (1,-1);
\draw (1,0) rectangle (2,-1);
\draw (2,0) rectangle (3,-1);
\draw (3,0) rectangle (4,-1);
\draw (4,0) rectangle (5,-1);
\draw (0,-1) rectangle (1,-2);
\draw (1,-1) rectangle (2,-2);
\draw (2,-1) rectangle (3,-2);
\draw (0.250,-0.350) rectangle ( 0.650 ,-0.750);
\draw (1.35,-0.350) rectangle ( 4.55 ,-0.750);
\draw ( 4.75 ,-0.300) node {\tiny +};
\draw (0.350,-1.35) rectangle ( 2.55 ,-1.75);
\end{tikzpicture}}
\]
Note that the doublebrick in the first factor contributes the block $1^2$ to the content and the doublebrick in the second factor contributes the block $2^4$ to the content.
\end{example}
We can think of a simple polybrick tabloid as an element of $\pt^{\mathrm{doub}}(\sigma, \tau)$ with $\ell_2(T) = 0$. Define an \textit{$E$-double polybrick tabloid of shape $\sigma$ and content $\tau$} to be a $T\in\pt^{\mathrm{doub}}(\sigma, \tau)$ such that at most one doublebrick appears in each row. Similarly, an \textit{$H$-double polybrick tabloid of shape $\sigma$ and content $\tau$} is a $T\in\pt^{\mathrm{doub}}(\sigma, \tau)$ such that at most one brick appears in each row. Denote these sets by $\pt_E^{\mathrm{doub}}(\sigma, \tau)$ and $\pt_H^{\mathrm{doub}}(\sigma, \tau)$ respectively.
A \textit{labeled doublebrick with label $l$} is a doublebrick with an associated label $l$. An \textit{ordered double polybrick tabloid of shape $\sigma$ and content $\tau = d_1^{r_1}\ldots d_s^{r_s}$} is defined to be a tiling of $\dg^\otimes(\sigma)$ such that we either have a brick of label $i$ and length $d_i$ in the $r_i$th tensor factor, or a doublebrick with label $i$ of length $2d_i$ in the $r_i/2$th tensor factor, and the labels increase from left to right within a row. Denote the set of these by $\pt^{\mathrm{odoub}}(\sigma, \tau)$.
\begin{example}
The following is an example of an ordered double polybrick tabloid of shape $\sigma = (5,2)^1(5,3)^2$ and content $\tau = (2,1)^1(3,1,1,1)^2(2)^4$:
\[
\raisebox{20.0pt}{\begin{tikzpicture}[scale =0.500,baseline=(current bounding box.north)]
\draw (0,0) rectangle (1,-1);
\draw (1,0) rectangle (2,-1);
\draw (2,0) rectangle (3,-1);
\draw (3,0) rectangle (4,-1);
\draw (4,0) rectangle (5,-1);
\draw (0,-1) rectangle (1,-2);
\draw (1,-1) rectangle (2,-2);
\draw (0.250,-0.350) rectangle ( 0.650 ,-0.750);
\draw ( 0.850 ,-0.700) node[scale =0.600] {2};
\draw (1.35,-0.350) rectangle ( 2.55 ,-0.750);
\draw ( 2.75 ,-0.700) node[scale =0.600] {4};
\draw ( 2.75 ,-0.300) node[scale =0.600] {+};
\draw (3.35,-0.350) rectangle ( 4.55 ,-0.750);
\draw ( 4.75 ,-0.700) node[scale =0.600] {6};
\draw ( 4.75 ,-0.300) node[scale =0.600] {+};
\draw (0.350,-1.35) rectangle ( 1.55 ,-1.75);
\draw ( 1.75 ,-1.70) node[scale =0.600] {1};
\end{tikzpicture}}
\otimes \raisebox{20.0pt}{\begin{tikzpicture}[scale =0.500,baseline=(current bounding box.north)]
\draw (0,0) rectangle (1,-1);
\draw (1,0) rectangle (2,-1);
\draw (2,0) rectangle (3,-1);
\draw (3,0) rectangle (4,-1);
\draw (4,0) rectangle (5,-1);
\draw (0,-1) rectangle (1,-2);
\draw (1,-1) rectangle (2,-2);
\draw (2,-1) rectangle (3,-2);
\draw (0.250,-0.350) rectangle ( 0.650 ,-0.750);
\draw ( 0.850 ,-0.700) node[scale =0.600] {5};
\draw (1.35,-0.350) rectangle ( 4.55 ,-0.750);
\draw ( 4.75 ,-0.700) node[scale =0.600] {7};
\draw ( 4.75 ,-0.300) node[scale =0.600] {+};
\draw (0.350,-1.35) rectangle ( 2.55 ,-1.75);
\draw ( 2.75 ,-1.70) node[scale =0.600] {3};
\end{tikzpicture}}
\]
\end{example}
\begin{theorem}\label{thm:E+inE}
For $\sigma, \tau\in \Typ(n)$, the coefficient of $E_\tau$ in $E^+_\sigma$ is $\sum_{T\in \pt_E^{\mathrm{doub}}(\sigma, \tau)} (-1)^{\ell_1(T)}$.
\end{theorem}
\begin{proof}
Each $d^r\in \sigma$ corresponds to a row of length $d$ in the $r$th tensor factor of $\dg^\otimes(\sigma)$. Choosing a polycomposition $\delta = \alpha^1(b)^2$ of $d$ corresponds to tiling the row with bricks of length $\alpha_1$, $\alpha_2,\ldots$, along with one doublebrick of length $2b$. This tiling corresponds to the choice of the term $(-1)^{\ell(\alpha)} E_{\alpha (b)^2}$ appearing in the expansion of $E^+_{d^r}$. The exponent in the sign $(-1)^{\ell(\alpha)}$ is the number of bricks used. Doing this process for all blocks in $\sigma$, we construct a double polybrick tabloid $T$ on the shape $\sigma$ with an associated sign $(-1)^{\ell_1(T)}$. To find its content, we construct the blocks in $\tau$ as follows: for each brick of length $a$ in the $m$th tensor factor, we create a block $a^m$ in $\tau$ and for each doublebrick of length $2b$ in the tensor factor $m$, we create the block $(b)^{2m}$. Using the expansion proved in Proposition \ref{prop:E+-HEP}, the statement in the theorem follows.
\end{proof}
\begin{example}
The coefficient of $E_{(2,1)^1(2,1,1)^2}$ in $E^+_{(5,2)^1(2)^2}$ is $-1$ and can be computed using the following tabloids with $\ell_1(T_1) = \ell_1(T_2) = 3$ and $\ell_1(T_3) = 4$:
\[T_1 = 
\raisebox{16.0pt}{\begin{tikzpicture}[scale =0.400,baseline=(current bounding box.north)]
\draw (0,0) rectangle (1,-1);
\draw (1,0) rectangle (2,-1);
\draw (2,0) rectangle (3,-1);
\draw (3,0) rectangle (4,-1);
\draw (4,0) rectangle (5,-1);
\draw (0,-1) rectangle (1,-2);
\draw (1,-1) rectangle (2,-2);
\draw (0.350,-0.350) rectangle ( 1.55 ,-0.750);
\draw (2.25,-0.350) rectangle ( 2.65 ,-0.750);
\draw (3.35,-0.350) rectangle ( 4.55 ,-0.750);
\draw ( 4.75 ,-0.300) node[scale =0.520] {+};
\draw (0.350,-1.35) rectangle ( 1.55 ,-1.75);
\draw ( 1.75 ,-1.30) node[scale =0.520] {+};
\end{tikzpicture}}
\otimes \raisebox{16.0pt}{\begin{tikzpicture}[scale =0.400,baseline=(current bounding box.north)]
\draw (0,0) rectangle (1,-1);
\draw (1,0) rectangle (2,-1);
\draw (0.350,-0.350) rectangle ( 1.55 ,-0.750);
\end{tikzpicture}}
\quad\Bigg\vert\,
T_2 = 
\raisebox{16.0pt}{\begin{tikzpicture}[scale =0.400,baseline=(current bounding box.north)]
\draw (0,0) rectangle (1,-1);
\draw (1,0) rectangle (2,-1);
\draw (2,0) rectangle (3,-1);
\draw (3,0) rectangle (4,-1);
\draw (4,0) rectangle (5,-1);
\draw (0,-1) rectangle (1,-2);
\draw (1,-1) rectangle (2,-2);
\draw (0.250,-0.350) rectangle ( 0.650 ,-0.750);
\draw (1.35,-0.350) rectangle ( 2.55 ,-0.750);
\draw (3.35,-0.350) rectangle ( 4.55 ,-0.750);
\draw ( 4.75 ,-0.300) node[scale =0.520] {+};
\draw (0.350,-1.35) rectangle ( 1.55 ,-1.75);
\draw ( 1.75 ,-1.30) node[scale =0.520] {+};
\end{tikzpicture}}
\otimes \raisebox{16.0pt}{\begin{tikzpicture}[scale =0.400,baseline=(current bounding box.north)]
\draw (0,0) rectangle (1,-1);
\draw (1,0) rectangle (2,-1);
\draw (0.350,-0.350) rectangle ( 1.55 ,-0.750);
\end{tikzpicture}}
\quad\Bigg\vert\,
T_3 = 
\raisebox{16.0pt}{\begin{tikzpicture}[scale =0.400,baseline=(current bounding box.north)]
\draw (0,0) rectangle (1,-1);
\draw (1,0) rectangle (2,-1);
\draw (2,0) rectangle (3,-1);
\draw (3,0) rectangle (4,-1);
\draw (4,0) rectangle (5,-1);
\draw (0,-1) rectangle (1,-2);
\draw (1,-1) rectangle (2,-2);
\draw (0.250,-0.350) rectangle ( 0.650 ,-0.750);
\draw (1.35,-0.350) rectangle ( 4.55 ,-0.750);
\draw ( 4.75 ,-0.300) node[scale =0.520] {+};
\draw (0.350,-1.35) rectangle ( 1.55 ,-1.75);
\end{tikzpicture}}
\otimes \raisebox{16.0pt}{\begin{tikzpicture}[scale =0.400,baseline=(current bounding box.north)]
\draw (0,0) rectangle (1,-1);
\draw (1,0) rectangle (2,-1);
\draw (0.250,-0.350) rectangle ( 0.650 ,-0.750);
\draw (1.25,-0.350) rectangle ( 1.65 ,-0.750);
\end{tikzpicture}}
\]
\end{example}
\begin{theorem}
For $\sigma, \tau\in \Typ(n)$, the coefficient of $H_\tau$ in $E^+_\sigma$ is $\sum_{T\in \pt_H^{\mathrm{doub}}(\sigma, \tau)} (-1)^{\ell_2(T)}$.
\end{theorem}
\begin{proof}
The proof is similar to the proof of Theorem \ref{thm:E+inE} but with the sign now accounting for the number of doublebricks instead.
\end{proof}
\begin{example}
The coefficient of $H_{(2)^1(1,1,1)^2(1)^4}$ in $E^+_{(4,2)^1(2,1)^2}$ is $-2$ as $\ell_2(T_1) = \ell(T_2) = 3$:
\[T_1 = 
\raisebox{20.0pt}{\begin{tikzpicture}[scale =0.500,baseline=(current bounding box.north)]
\draw (0,0) rectangle (1,-1);
\draw (1,0) rectangle (2,-1);
\draw (2,0) rectangle (3,-1);
\draw (3,0) rectangle (4,-1);
\draw (0,-1) rectangle (1,-2);
\draw (1,-1) rectangle (2,-2);
\draw (0.350,-0.350) rectangle ( 1.55 ,-0.750);
\draw ( 1.75 ,-0.300) node[scale =0.650] {+};
\draw (2.35,-0.350) rectangle ( 3.55 ,-0.750);
\draw ( 3.75 ,-0.300) node[scale =0.650] {+};
\draw (0.350,-1.35) rectangle ( 1.55 ,-1.75);
\end{tikzpicture}}
\otimes \raisebox{20.0pt}{\begin{tikzpicture}[scale =0.500,baseline=(current bounding box.north)]
\draw (0,0) rectangle (1,-1);
\draw (1,0) rectangle (2,-1);
\draw (0,-1) rectangle (1,-2);
\draw (0.350,-0.350) rectangle ( 1.55 ,-0.750);
\draw ( 1.75 ,-0.300) node[scale =0.650] {+};
\draw (0.250,-1.35) rectangle ( 0.650 ,-1.75);
\end{tikzpicture}}
\quad\Bigg\vert\,
T_2 = 
\raisebox{20.0pt}{\begin{tikzpicture}[scale =0.500,baseline=(current bounding box.north)]
\draw (0,0) rectangle (1,-1);
\draw (1,0) rectangle (2,-1);
\draw (2,0) rectangle (3,-1);
\draw (3,0) rectangle (4,-1);
\draw (0,-1) rectangle (1,-2);
\draw (1,-1) rectangle (2,-2);
\draw (0.350,-0.350) rectangle ( 1.55 ,-0.750);
\draw (2.35,-0.350) rectangle ( 3.55 ,-0.750);
\draw ( 3.75 ,-0.300) node[scale =0.650] {+};
\draw (0.350,-1.35) rectangle ( 1.55 ,-1.75);
\draw ( 1.75 ,-1.30) node[scale =0.650] {+};
\end{tikzpicture}}
\otimes \raisebox{20.0pt}{\begin{tikzpicture}[scale =0.500,baseline=(current bounding box.north)]
\draw (0,0) rectangle (1,-1);
\draw (1,0) rectangle (2,-1);
\draw (0,-1) rectangle (1,-2);
\draw (0.350,-0.350) rectangle ( 1.55 ,-0.750);
\draw ( 1.75 ,-0.300) node[scale =0.650] {+};
\draw (0.250,-1.35) rectangle ( 0.650 ,-1.75);
\end{tikzpicture}}
\]
\end{example}
Recall that $z_\tau^\otimes$ is the product $\prod_{i\geq 1}z_{\tau|^i}$.
\begin{theorem}\label{thm:E+inP}
For $\sigma, \tau\in \Typ(n)$, the coefficient of $P_\tau$ in $E^+_\sigma$ is given by $\sum_T (-1)^{\ell_2(T)}/z_{\tau}^\otimes$ where the sum is over all $T\in \pt^{\mathrm{odoub}}(\sigma, \tau)$.
\end{theorem}
\begin{proof}
The proof is similar to the proof of Theorem \ref{thm:H-P,E-P} with the sign accounting for doublebricks similar to the proof of Theorem \ref{thm:E+inE}.
\end{proof}
\begin{example}
The coefficient of $P_{(2,1)^1(2,1,1)^2}$ in $E^+_{(5,2)^1(2)^2}$ is computed using the following four ordered double polybrick tabloids:
\[T_1 = 
\raisebox{20.0pt}{\begin{tikzpicture}[scale =0.500,baseline=(current bounding box.north)]
\draw (0,0) rectangle (1,-1);
\draw (1,0) rectangle (2,-1);
\draw (2,0) rectangle (3,-1);
\draw (3,0) rectangle (4,-1);
\draw (4,0) rectangle (5,-1);
\draw (0,-1) rectangle (1,-2);
\draw (1,-1) rectangle (2,-2);
\draw (0.350,-0.350) rectangle ( 1.55 ,-0.750);
\draw ( 1.75 ,-0.700) node[scale =0.650] {1};
\draw (2.25,-0.350) rectangle ( 2.65 ,-0.750);
\draw ( 2.85 ,-0.700) node[scale =0.650] {2};
\draw (3.35,-0.350) rectangle ( 4.55 ,-0.750);
\draw ( 4.75 ,-0.700) node[scale =0.650] {4};
\draw ( 4.75 ,-0.300) node[scale =0.650] {+};
\draw (0.350,-1.35) rectangle ( 1.55 ,-1.75);
\draw ( 1.75 ,-1.70) node[scale =0.650] {5};
\draw ( 1.75 ,-1.30) node[scale =0.650] {+};
\end{tikzpicture}}
\otimes \raisebox{20.0pt}{\begin{tikzpicture}[scale =0.500,baseline=(current bounding box.north)]
\draw (0,0) rectangle (1,-1);
\draw (1,0) rectangle (2,-1);
\draw (0.350,-0.350) rectangle ( 1.55 ,-0.750);
\draw ( 1.75 ,-0.700) node[scale =0.650] {3};
\end{tikzpicture}}
\quad\Bigg\vert\,
T_2 = 
\raisebox{20.0pt}{\begin{tikzpicture}[scale =0.500,baseline=(current bounding box.north)]
\draw (0,0) rectangle (1,-1);
\draw (1,0) rectangle (2,-1);
\draw (2,0) rectangle (3,-1);
\draw (3,0) rectangle (4,-1);
\draw (4,0) rectangle (5,-1);
\draw (0,-1) rectangle (1,-2);
\draw (1,-1) rectangle (2,-2);
\draw (0.350,-0.350) rectangle ( 1.55 ,-0.750);
\draw ( 1.75 ,-0.700) node[scale =0.650] {1};
\draw (2.25,-0.350) rectangle ( 2.65 ,-0.750);
\draw ( 2.85 ,-0.700) node[scale =0.650] {2};
\draw (3.35,-0.350) rectangle ( 4.55 ,-0.750);
\draw ( 4.75 ,-0.700) node[scale =0.650] {5};
\draw ( 4.75 ,-0.300) node[scale =0.650] {+};
\draw (0.350,-1.35) rectangle ( 1.55 ,-1.75);
\draw ( 1.75 ,-1.70) node[scale =0.650] {4};
\draw ( 1.75 ,-1.30) node[scale =0.650] {+};
\end{tikzpicture}}
\otimes \raisebox{20.0pt}{\begin{tikzpicture}[scale =0.500,baseline=(current bounding box.north)]
\draw (0,0) rectangle (1,-1);
\draw (1,0) rectangle (2,-1);
\draw (0.350,-0.350) rectangle ( 1.55 ,-0.750);
\draw ( 1.75 ,-0.700) node[scale =0.650] {3};
\end{tikzpicture}}
\]
\[T_3 = 
\raisebox{20.0pt}{\begin{tikzpicture}[scale =0.500,baseline=(current bounding box.north)]
\draw (0,0) rectangle (1,-1);
\draw (1,0) rectangle (2,-1);
\draw (2,0) rectangle (3,-1);
\draw (3,0) rectangle (4,-1);
\draw (4,0) rectangle (5,-1);
\draw (0,-1) rectangle (1,-2);
\draw (1,-1) rectangle (2,-2);
\draw (0.250,-0.350) rectangle ( 0.650 ,-0.750);
\draw ( 0.850 ,-0.700) node[scale =0.650] {2};
\draw (1.35,-0.350) rectangle ( 2.55 ,-0.750);
\draw ( 2.75 ,-0.700) node[scale =0.650] {4};
\draw ( 2.75 ,-0.300) node[scale =0.650] {+};
\draw (3.35,-0.350) rectangle ( 4.55 ,-0.750);
\draw ( 4.75 ,-0.700) node[scale =0.650] {5};
\draw ( 4.75 ,-0.300) node[scale =0.650] {+};
\draw (0.350,-1.35) rectangle ( 1.55 ,-1.75);
\draw ( 1.75 ,-1.70) node[scale =0.650] {1};
\end{tikzpicture}}
\otimes \raisebox{20.0pt}{\begin{tikzpicture}[scale =0.500,baseline=(current bounding box.north)]
\draw (0,0) rectangle (1,-1);
\draw (1,0) rectangle (2,-1);
\draw (0.350,-0.350) rectangle ( 1.55 ,-0.750);
\draw ( 1.75 ,-0.700) node[scale =0.650] {3};
\end{tikzpicture}}
\quad\Bigg\vert\,
T_4 = 
\raisebox{20.0pt}{\begin{tikzpicture}[scale =0.500,baseline=(current bounding box.north)]
\draw (0,0) rectangle (1,-1);
\draw (1,0) rectangle (2,-1);
\draw (2,0) rectangle (3,-1);
\draw (3,0) rectangle (4,-1);
\draw (4,0) rectangle (5,-1);
\draw (0,-1) rectangle (1,-2);
\draw (1,-1) rectangle (2,-2);
\draw (0.250,-0.350) rectangle ( 0.650 ,-0.750);
\draw ( 0.850 ,-0.700) node[scale =0.650] {2};
\draw (1.35,-0.350) rectangle ( 4.55 ,-0.750);
\draw ( 4.75 ,-0.700) node[scale =0.650] {3};
\draw ( 4.75 ,-0.300) node[scale =0.650] {+};
\draw (0.350,-1.35) rectangle ( 1.55 ,-1.75);
\draw ( 1.75 ,-1.70) node[scale =0.650] {1};
\end{tikzpicture}}
\otimes \raisebox{20.0pt}{\begin{tikzpicture}[scale =0.500,baseline=(current bounding box.north)]
\draw (0,0) rectangle (1,-1);
\draw (1,0) rectangle (2,-1);
\draw (0.250,-0.350) rectangle ( 0.650 ,-0.750);
\draw ( 0.850 ,-0.700) node[scale =0.650] {4};
\draw (1.25,-0.350) rectangle ( 1.65 ,-0.750);
\draw ( 1.85 ,-0.700) node[scale =0.650] {5};
\end{tikzpicture}}
\]
We have $\ell_2(T_1) = \ell_2(T_2) = \ell_2(T_3) = 2$ and $\ell_2(T_4) = 1$. This gives us $\sum_{T}(-1)^{\ell_2(T)} = 2$ and we compute $z_{(2,1)(2,1,1)^2}^\otimes = 8$ which results in the coefficient $1/4$.
\end{example}
\subsection{Dyadic polybrick tabloids}
We now describe objects that appear in the $E^+$-expansions of $H_\sigma$, $E_\sigma$ and $P_\sigma$. 
Define a \textit{$k$-brick} to be a brick of length a multiple of $2^k$ which we draw with a $k$ on the top-right corner\footnote{We do so to avoid confusion with labels in the previous section which go in the subscript.}. 
For types $\sigma$ and $\tau$ of $n$ with $\tau = d_1^{r_1} \ldots d_s^{r_s}$ and $r_1 \leq r_2\leq \ldots \leq r_s$, define a \textit{dyadic polybrick tabloid of shape $\sigma$ and content $\tau$} to be a tiling of $\dg^\otimes(\sigma)$ with $k$-bricks constructed as follows: for each $1\leq i\leq s$, we choose $k_i$ such that $2^{k_i}$ divides $r_i$ and place a $k_i$-brick of length $2^{k_i}d_i$ in the $r_i/2^{k_i}$th tensor factor, with the added condition that for $k'>k$, in any given row a $k'$-brick appears to the right of a $k$-brick if both are present in that row. We add a marking $k_i$ to denote that it is a $k_i$-brick.  Denote the set of dyadic polybrick tabloids of shape $\sigma$ and content $\tau$ by $\pt^{\dya}(\sigma, \tau)$. Define a \textit{dyadic polybrick} to mean a $k$-brick for some $k\geq 0$.
\begin{example}
The following is an example of a dyadic polybrick tabloid of shape $\sigma = (7,4)^1(4,4)^2(3)^3(2)^4$ and content $\tau = (2,1)(4,1,1)^2(1)^3(2,2,1)^4(1)^6$.
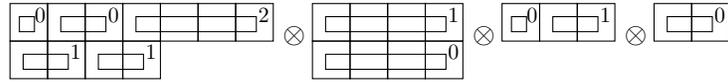
\[
\raisebox{16.0pt}{\begin{tikzpicture}[scale =0.500,baseline=(current bounding box.north)]
\draw (0,0) rectangle (1,-1);
\draw (1,0) rectangle (2,-1);
\draw (2,0) rectangle (3,-1);
\draw (3,0) rectangle (4,-1);
\draw (4,0) rectangle (5,-1);
\draw (5,0) rectangle (6,-1);
\draw (6,0) rectangle (7,-1);
\draw (0,-1) rectangle (1,-2);
\draw (1,-1) rectangle (2,-2);
\draw (2,-1) rectangle (3,-2);
\draw (3,-1) rectangle (4,-2);
\draw (0.250,-0.350) rectangle ( 0.650 ,-0.750);
\draw ( 0.800 ,-0.300) node[scale =0.750] {0};
\draw (1.35,-0.350) rectangle ( 2.55 ,-0.750);
\draw ( 2.75 ,-0.300) node[scale =0.750] {0};
\draw (3.35,-0.350) rectangle ( 6.55 ,-0.750);
\draw ( 6.75 ,-0.300) node[scale =0.750] {2};
\draw (0.350,-1.35) rectangle ( 1.55 ,-1.75);
\draw ( 1.75 ,-1.30) node[scale =0.750] {1};
\draw (2.35,-1.35) rectangle ( 3.55 ,-1.75);
\draw ( 3.75 ,-1.30) node[scale =0.750] {1};
\end{tikzpicture}}
\otimes \raisebox{16.0pt}{\begin{tikzpicture}[scale =0.500,baseline=(current bounding box.north)]
\draw (0,0) rectangle (1,-1);
\draw (1,0) rectangle (2,-1);
\draw (2,0) rectangle (3,-1);
\draw (3,0) rectangle (4,-1);
\draw (0,-1) rectangle (1,-2);
\draw (1,-1) rectangle (2,-2);
\draw (2,-1) rectangle (3,-2);
\draw (3,-1) rectangle (4,-2);
\draw (0.350,-0.350) rectangle ( 3.55 ,-0.750);
\draw ( 3.75 ,-0.300) node[scale =0.750] {1};
\draw (0.350,-1.35) rectangle ( 3.55 ,-1.75);
\draw ( 3.75 ,-1.30) node[scale =0.750] {0};
\end{tikzpicture}}
\otimes \raisebox{16.0pt}{\begin{tikzpicture}[scale =0.500,baseline=(current bounding box.north)]
\draw (0,0) rectangle (1,-1);
\draw (1,0) rectangle (2,-1);
\draw (2,0) rectangle (3,-1);
\draw (0.250,-0.350) rectangle ( 0.650 ,-0.750);
\draw ( 0.800 ,-0.300) node[scale =0.750] {0};
\draw (1.35,-0.350) rectangle ( 2.55 ,-0.750);
\draw ( 2.75 ,-0.300) node[scale =0.750] {1};
\end{tikzpicture}}
\otimes \raisebox{16.0pt}{\begin{tikzpicture}[scale =0.500,baseline=(current bounding box.north)]
\draw (0,0) rectangle (1,-1);
\draw (1,0) rectangle (2,-1);
\draw (0.350,-0.350) rectangle ( 1.55 ,-0.750);
\draw ( 1.75 ,-0.300) node[scale =0.750] {0};
\end{tikzpicture}}
\]
In the first tensor factor, the $2$-brick corresponds to the block $1^4$  in $\tau$ and it has length $2^2\cdot 1 = 4$. In the second factor, the $1$-brick corresponds to the block $2^4$ and it has length $2^1\cdot 2 =4$.
\end{example}

\begin{theorem}\label{thm:EinE+}
For $\sigma, \tau\in \Typ(n)$, the coefficient of $E^+_\tau$ in $E_\sigma$ is $(-1)^{\ell(\tau)}|\pt^\dya(\sigma, \tau)|$.
\end{theorem}
\begin{proof}
Each $d^r\in \sigma$ corresponds to a row with $d$ cells in the tensor factor $r$. Choose a dyadic polycomposition $\delta = \alpha^1\beta^2 \gamma^4\ldots$ of $d$ and place bricks $\alpha_1$, $\alpha_2,\ldots$ marked with 0, followed by bricks of length $2\beta_1$, $2\beta_2$ and so on marked with a 1, followed by bricks of length $4\gamma_1$, $4\gamma_2,\ldots$, and tile all the cells of the row in this manner. This tiling corresponds to the term $(-1)^{\ell(\delta)}E^+_{\alpha^1 \beta^2\gamma^4\ldots}$ in the expansion of $E_{d^r}$. The exponent of the sign is the number of dyadic polybricks used. Performing this tiling for all $d^r\in \sigma$, we obtain a dyadic polybrick tabloid $T$ on the shape $\sigma$. We compute the content $\tau$ as follows: for each $k$-brick in the factor $r$ of length $2^k b$, we create the block $b^{2^k r}$ in $\tau$. The term associated to $T$ is $(-1)^{\ell(\tau)}E^+_\tau$ and constructing all possible dyadic polybrick tabloids of shape $\sigma$ and content $\tau$ gives us the needed coefficient.
\end{proof}
\boks{0.42} 
\begin{example}
Let $\tau = (2,2)^1(1,1)^2(1)^4$ and $\sigma = (6,2)^1(2)^2$. The coefficient of $E^+_{\tau}$ in $E_{\sigma}$ is $-3$ as $\ell(\tau) = 5$ and we have the following three dyadic polybrick tabloids of shape $\sigma$ and content $\tau$:
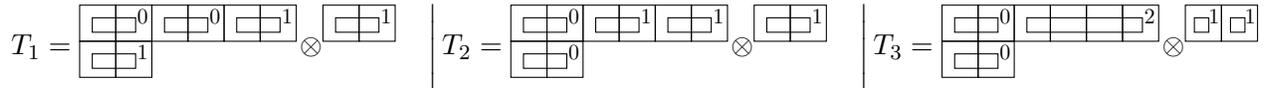
\[
T_1 = 
\raisebox{20.0pt}{\begin{tikzpicture}[scale =0.480,baseline=(current bounding box.north)]
\draw (0,0) rectangle (1,-1);
\draw (1,0) rectangle (2,-1);
\draw (2,0) rectangle (3,-1);
\draw (3,0) rectangle (4,-1);
\draw (4,0) rectangle (5,-1);
\draw (5,0) rectangle (6,-1);
\draw (0,-1) rectangle (1,-2);
\draw (1,-1) rectangle (2,-2);
\draw (0.350,-0.350) rectangle ( 1.55 ,-0.750);
\draw ( 1.75 ,-0.300) node[scale =0.750] {0};
\draw (2.35,-0.350) rectangle ( 3.55 ,-0.750);
\draw ( 3.75 ,-0.300) node[scale =0.750] {0};
\draw (4.35,-0.350) rectangle ( 5.55 ,-0.750);
\draw ( 5.75 ,-0.300) node[scale =0.750] {1};
\draw (0.350,-1.35) rectangle ( 1.55 ,-1.75);
\draw ( 1.75 ,-1.30) node[scale =0.750] {1};
\end{tikzpicture}}
\otimes \raisebox{20.0pt}{\begin{tikzpicture}[scale =0.480,baseline=(current bounding box.north)]
\draw (0,0) rectangle (1,-1);
\draw (1,0) rectangle (2,-1);
\draw (0.350,-0.350) rectangle ( 1.55 ,-0.750);
\draw ( 1.75 ,-0.300) node[scale =0.750] {1};
\end{tikzpicture}}
\quad\Bigg\vert\,
T_2 = 
\raisebox{20.0pt}{\begin{tikzpicture}[scale =0.480,baseline=(current bounding box.north)]
\draw (0,0) rectangle (1,-1);
\draw (1,0) rectangle (2,-1);
\draw (2,0) rectangle (3,-1);
\draw (3,0) rectangle (4,-1);
\draw (4,0) rectangle (5,-1);
\draw (5,0) rectangle (6,-1);
\draw (0,-1) rectangle (1,-2);
\draw (1,-1) rectangle (2,-2);
\draw (0.350,-0.350) rectangle ( 1.55 ,-0.750);
\draw ( 1.75 ,-0.300) node[scale =0.750] {0};
\draw (2.35,-0.350) rectangle ( 3.55 ,-0.750);
\draw ( 3.75 ,-0.300) node[scale =0.750] {1};
\draw (4.35,-0.350) rectangle ( 5.55 ,-0.750);
\draw ( 5.75 ,-0.300) node[scale =0.750] {1};
\draw (0.350,-1.35) rectangle ( 1.55 ,-1.75);
\draw ( 1.75 ,-1.30) node[scale =0.750] {0};
\end{tikzpicture}}
\otimes \raisebox{20.0pt}{\begin{tikzpicture}[scale =0.480,baseline=(current bounding box.north)]
\draw (0,0) rectangle (1,-1);
\draw (1,0) rectangle (2,-1);
\draw (0.350,-0.350) rectangle ( 1.55 ,-0.750);
\draw ( 1.75 ,-0.300) node[scale =0.750] {1};
\end{tikzpicture}}
\quad\Bigg\vert\,
T_3 = 
\raisebox{20.0pt}{\begin{tikzpicture}[scale =0.480,baseline=(current bounding box.north)]
\draw (0,0) rectangle (1,-1);
\draw (1,0) rectangle (2,-1);
\draw (2,0) rectangle (3,-1);
\draw (3,0) rectangle (4,-1);
\draw (4,0) rectangle (5,-1);
\draw (5,0) rectangle (6,-1);
\draw (0,-1) rectangle (1,-2);
\draw (1,-1) rectangle (2,-2);
\draw (0.350,-0.350) rectangle ( 1.55 ,-0.750);
\draw ( 1.75 ,-0.300) node[scale =0.750] {0};
\draw (2.35,-0.350) rectangle ( 5.55 ,-0.750);
\draw ( 5.75 ,-0.300) node[scale =0.750] {2};
\draw (0.350,-1.35) rectangle ( 1.55 ,-1.75);
\draw ( 1.75 ,-1.30) node[scale =0.750] {0};
\end{tikzpicture}}
\otimes \raisebox{20.0pt}{\begin{tikzpicture}[scale =0.480,baseline=(current bounding box.north)]
\draw (0,0) rectangle (1,-1);
\draw (1,0) rectangle (2,-1);
\draw (0.250,-0.350) rectangle ( 0.650 ,-0.750);
\draw ( 0.800 ,-0.300) node[scale =0.750] {1};
\draw (1.25,-0.350) rectangle ( 1.65 ,-0.750);
\draw ( 1.80 ,-0.300) node[scale =0.750] {1};
\end{tikzpicture}}
\]
\end{example}
Call a dyadic polybrick tabloid of shape $\sigma$ and content $\tau$ \textit{distinct} if in each row of each diagram, each marking $k$ appears at most once. Denote this set by $\pt^{\dya}_{\mathrm{dis}}(\sigma, \tau)$.
\begin{example}
\boks{0.45}
The following is a distinct polybrick tabloid of shape $(4,3)^1(6,4)^2(2,2,1)^3$ and content $(2,1)^1(2,1,1)^2(1)^3(2)^4(1,1)^6(1)^8$:
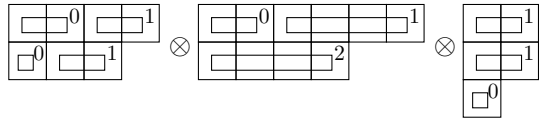
\[
\raisebox{20.0pt}{\begin{tikzpicture}[scale =0.500,baseline=(current bounding box.north)]
\draw (0,0) rectangle (1,-1);
\draw (1,0) rectangle (2,-1);
\draw (2,0) rectangle (3,-1);
\draw (3,0) rectangle (4,-1);
\draw (0,-1) rectangle (1,-2);
\draw (1,-1) rectangle (2,-2);
\draw (2,-1) rectangle (3,-2);
\draw (0.350,-0.350) rectangle ( 1.55 ,-0.750);
\draw ( 1.75 ,-0.300) node[scale =0.750] {0};
\draw (2.35,-0.350) rectangle ( 3.55 ,-0.750);
\draw ( 3.75 ,-0.300) node[scale =0.750] {1};
\draw (0.250,-1.35) rectangle ( 0.650 ,-1.75);
\draw ( 0.800 ,-1.30) node[scale =0.750] {0};
\draw (1.35,-1.35) rectangle ( 2.55 ,-1.75);
\draw ( 2.75 ,-1.30) node[scale =0.750] {1};
\end{tikzpicture}}
\otimes \raisebox{20.0pt}{\begin{tikzpicture}[scale =0.500,baseline=(current bounding box.north)]
\draw (0,0) rectangle (1,-1);
\draw (1,0) rectangle (2,-1);
\draw (2,0) rectangle (3,-1);
\draw (3,0) rectangle (4,-1);
\draw (4,0) rectangle (5,-1);
\draw (5,0) rectangle (6,-1);
\draw (0,-1) rectangle (1,-2);
\draw (1,-1) rectangle (2,-2);
\draw (2,-1) rectangle (3,-2);
\draw (3,-1) rectangle (4,-2);
\draw (0.350,-0.350) rectangle ( 1.55 ,-0.750);
\draw ( 1.75 ,-0.300) node[scale =0.750] {0};
\draw (2.35,-0.350) rectangle ( 5.55 ,-0.750);
\draw ( 5.75 ,-0.300) node[scale =0.750] {1};
\draw (0.350,-1.35) rectangle ( 3.55 ,-1.75);
\draw ( 3.75 ,-1.30) node[scale =0.750] {2};
\end{tikzpicture}}
\otimes \raisebox{20.0pt}{\begin{tikzpicture}[scale =0.500,baseline=(current bounding box.north)]
\draw (0,0) rectangle (1,-1);
\draw (1,0) rectangle (2,-1);
\draw (0,-1) rectangle (1,-2);
\draw (1,-1) rectangle (2,-2);
\draw (0,-2) rectangle (1,-3);
\draw (0.350,-0.350) rectangle ( 1.55 ,-0.750);
\draw ( 1.75 ,-0.300) node[scale =0.750] {1};
\draw (0.350,-1.35) rectangle ( 1.55 ,-1.75);
\draw ( 1.75 ,-1.30) node[scale =0.750] {1};
\draw (0.250,-2.35) rectangle ( 0.650 ,-2.75);
\draw ( 0.800 ,-2.30) node[scale =0.750] {0};
\end{tikzpicture}}
\]
\end{example}
\begin{theorem}
Let $n\geq 0$. For $\tau, \sigma\in \Typ(n)$, the coefficient of $E^+_\tau$ in $H_\sigma$ is $|\pt^{\dya}_{\mathrm{dis}}(\sigma, \tau)|$.
\end{theorem}
\begin{proof}
The proof proceeds similarly to the proof of Theorem \ref{thm:EinE+} and here we may place at most one $k$-brick for each $k$ in each row.
\end{proof}
\begin{example}
The coefficient of $E^+_{(2)^1(2)^2(1,1)^4}$ in $H_{(4,2)^1(2,2)^2}$ is 3 and can be computed using the following polybrick tabloids:
\[
T_1 = 
\raisebox{19.2pt}{\begin{tikzpicture}[scale =0.480,baseline=(current bounding box.north)]
\draw (0,0) rectangle (1,-1);
\draw (1,0) rectangle (2,-1);
\draw (2,0) rectangle (3,-1);
\draw (3,0) rectangle (4,-1);
\draw (0,-1) rectangle (1,-2);
\draw (1,-1) rectangle (2,-2);
\draw (0.350,-0.350) rectangle ( 3.55 ,-0.750);
\draw ( 3.75 ,-0.300) node[scale =0.720] {1};
\draw (0.350,-1.35) rectangle ( 1.55 ,-1.75);
\draw ( 1.75 ,-1.30) node[scale =0.720] {0};
\end{tikzpicture}}
\otimes \raisebox{19.2pt}{\begin{tikzpicture}[scale =0.480,baseline=(current bounding box.north)]
\draw (0,0) rectangle (1,-1);
\draw (1,0) rectangle (2,-1);
\draw (0,-1) rectangle (1,-2);
\draw (1,-1) rectangle (2,-2);
\draw (0.350,-0.350) rectangle ( 1.55 ,-0.750);
\draw ( 1.75 ,-0.300) node[scale =0.720] {1};
\draw (0.350,-1.35) rectangle ( 1.55 ,-1.75);
\draw ( 1.75 ,-1.30) node[scale =0.720] {1};
\end{tikzpicture}}
\quad\Bigg\vert\,
T_2 = 
\raisebox{19.2pt}{\begin{tikzpicture}[scale =0.480,baseline=(current bounding box.north)]
\draw (0,0) rectangle (1,-1);
\draw (1,0) rectangle (2,-1);
\draw (2,0) rectangle (3,-1);
\draw (3,0) rectangle (4,-1);
\draw (0,-1) rectangle (1,-2);
\draw (1,-1) rectangle (2,-2);
\draw (0.350,-0.350) rectangle ( 3.55 ,-0.750);
\draw ( 3.75 ,-0.300) node[scale =0.720] {2};
\draw (0.350,-1.35) rectangle ( 1.55 ,-1.75);
\draw ( 1.75 ,-1.30) node[scale =0.720] {0};
\end{tikzpicture}}
\otimes \raisebox{19.2pt}{\begin{tikzpicture}[scale =0.480,baseline=(current bounding box.north)]
\draw (0,0) rectangle (1,-1);
\draw (1,0) rectangle (2,-1);
\draw (0,-1) rectangle (1,-2);
\draw (1,-1) rectangle (2,-2);
\draw (0.350,-0.350) rectangle ( 1.55 ,-0.750);
\draw ( 1.75 ,-0.300) node[scale =0.720] {0};
\draw (0.350,-1.35) rectangle ( 1.55 ,-1.75);
\draw ( 1.75 ,-1.30) node[scale =0.720] {1};
\end{tikzpicture}}
\quad\Bigg\vert\,
T_3 = 
\raisebox{19.2pt}{\begin{tikzpicture}[scale =0.480,baseline=(current bounding box.north)]
\draw (0,0) rectangle (1,-1);
\draw (1,0) rectangle (2,-1);
\draw (2,0) rectangle (3,-1);
\draw (3,0) rectangle (4,-1);
\draw (0,-1) rectangle (1,-2);
\draw (1,-1) rectangle (2,-2);
\draw (0.350,-0.350) rectangle ( 3.55 ,-0.750);
\draw ( 3.75 ,-0.300) node[scale =0.720] {2};
\draw (0.350,-1.35) rectangle ( 1.55 ,-1.75);
\draw ( 1.75 ,-1.30) node[scale =0.720] {0};
\end{tikzpicture}}
\otimes \raisebox{19.2pt}{\begin{tikzpicture}[scale =0.480,baseline=(current bounding box.north)]
\draw (0,0) rectangle (1,-1);
\draw (1,0) rectangle (2,-1);
\draw (0,-1) rectangle (1,-2);
\draw (1,-1) rectangle (2,-2);
\draw (0.350,-0.350) rectangle ( 1.55 ,-0.750);
\draw ( 1.75 ,-0.300) node[scale =0.720] {1};
\draw (0.350,-1.35) rectangle ( 1.55 ,-1.75);
\draw ( 1.75 ,-1.30) node[scale =0.720] {0};
\end{tikzpicture}}
\]
\end{example}

Let $\sigma, \tau\in \Typ(n)$.
Define a \textit{singular dyadic polybrick tabloid of shape $\sigma$ and content $\tau$} to be an element of $\pt^\dya(\sigma,\tau)$ such that all dyadic polybricks in a row are $k$-bricks for the same $k$ in a row. We denote the set of singular dyadic polybrick tabloid of shape $\sigma$ and content $\tau$ by $\pt^\dya_{\mathrm{sing}}(\sigma, \tau)$. For $T\in \pt^{\dya}_{\mathrm{sing}}(\sigma, \tau)$, define $\L\ast(T)$ to be the product of the lengths of dyadic polybricks that appear at the end of each row
\begin{theorem}
For $\sigma, \tau\in \Typ(n)$, the coefficient of $E^+_\tau$ in $P_\sigma$ is $\sum
\limits_{T\in \pt^{\dya}_{\mathrm{sing}}(\sigma, \tau)} (-1)^{\ell(\tau)-\ell(\sigma)} \L\ast(T)$.
\end{theorem}
\begin{proof}
For each $d^r$ in $\sigma$, we choose a singular polycomposition $\delta = (\alpha)^{2^k}$, and  place $k$-bricks of length $2^k\alpha_1$, $2^k\alpha_2,\ldots$ in the row of length $d$ in the $r$th tensor factor. Such a tiling corresponds to the term $(-1)^{\ell(\delta)-1}L(\delta) E^+_\delta$ in the expansion of $P_{d^r}$. Here the exponent of the sign is one less than the number of $k$-bricks and $L(\delta)$ is the length of the rightmost brick. Tiling in a similar manner for all $d^r\in \sigma$ gives us a singular dyadic polybrick tabloid $T$ on the shape $\sigma$. To compute its content, for each $k$-brick of length $2^k b$ in the $r$th tensor factor, we create a block $b^{2^k r}$ in $\tau$. The associated sign can be computed by accounting for the total number of bricks minus the number of rows in $\sigma$ which is exactly $\ell(\tau)-\ell(\sigma)$, and the weight is computed by multiplying the lengths of the rightmost bricks which gives $\L\ast(T)$.
\end{proof}
\begin{example}
Let $\sigma = (4,2)^1(2,2)^2$ and $\tau = (2,2,1,1,1)^2$. The coefficient of $E^+_\tau$ in $P_\sigma$ is $-48$ and can be computed using the following three polybrick tabloids:
\[
\begin{array}{c|c|c}
T_1 = 
\raisebox{19.2pt}{\begin{tikzpicture}[scale =0.480,baseline=(current bounding box.north)]
\draw (0,0) rectangle (1,-1);
\draw (1,0) rectangle (2,-1);
\draw (2,0) rectangle (3,-1);
\draw (3,0) rectangle (4,-1);
\draw (0,-1) rectangle (1,-2);
\draw (1,-1) rectangle (2,-2);
\draw (0.350,-0.350) rectangle ( 3.55 ,-0.750);
\draw ( 3.75 ,-0.300) node[scale =0.720] {1};
\draw (0.350,-1.35) rectangle ( 1.55 ,-1.75);
\draw ( 1.75 ,-1.30) node[scale =0.720] {1};
\end{tikzpicture}}
\otimes \raisebox{19.2pt}{\begin{tikzpicture}[scale =0.480,baseline=(current bounding box.north)]
\draw (0,0) rectangle (1,-1);
\draw (1,0) rectangle (2,-1);
\draw (0,-1) rectangle (1,-2);
\draw (1,-1) rectangle (2,-2);
\draw (0.350,-0.350) rectangle ( 1.55 ,-0.750);
\draw ( 1.75 ,-0.300) node[scale =0.720] {0};
\draw (0.250,-1.35) rectangle ( 0.650 ,-1.75);
\draw ( 0.800 ,-1.30) node[scale =0.720] {0};
\draw (1.25,-1.35) rectangle ( 1.65 ,-1.75);
\draw ( 1.80 ,-1.30) node[scale =0.720] {0};
\end{tikzpicture}}&  T_2 = 
\raisebox{19.2pt}{\begin{tikzpicture}[scale =0.480,baseline=(current bounding box.north)]
\draw (0,0) rectangle (1,-1);
\draw (1,0) rectangle (2,-1);
\draw (2,0) rectangle (3,-1);
\draw (3,0) rectangle (4,-1);
\draw (0,-1) rectangle (1,-2);
\draw (1,-1) rectangle (2,-2);
\draw (0.350,-0.350) rectangle ( 3.55 ,-0.750);
\draw ( 3.75 ,-0.300) node[scale =0.720] {1};
\draw (0.350,-1.35) rectangle ( 1.55 ,-1.75);
\draw ( 1.75 ,-1.30) node[scale =0.720] {1};
\end{tikzpicture}}
\otimes \raisebox{19.2pt}{\begin{tikzpicture}[scale =0.480,baseline=(current bounding box.north)]
\draw (0,0) rectangle (1,-1);
\draw (1,0) rectangle (2,-1);
\draw (0,-1) rectangle (1,-2);
\draw (1,-1) rectangle (2,-2);
\draw (0.250,-0.350) rectangle ( 0.650 ,-0.750);
\draw ( 0.800 ,-0.300) node[scale =0.720] {0};
\draw (1.25,-0.350) rectangle ( 1.65 ,-0.750);
\draw ( 1.80 ,-0.300) node[scale =0.720] {0};
\draw (0.350,-1.35) rectangle ( 1.55 ,-1.75);
\draw ( 1.75 ,-1.30) node[scale =0.720] {0};
\end{tikzpicture}}&T_3 = 
\raisebox{19.2pt}{\begin{tikzpicture}[scale =0.480,baseline=(current bounding box.north)]
\draw (0,0) rectangle (1,-1);
\draw (1,0) rectangle (2,-1);
\draw (2,0) rectangle (3,-1);
\draw (3,0) rectangle (4,-1);
\draw (0,-1) rectangle (1,-2);
\draw (1,-1) rectangle (2,-2);
\draw (0.350,-0.350) rectangle ( 1.55 ,-0.750);
\draw ( 1.75 ,-0.300) node[scale =0.720] {1};
\draw (2.35,-0.350) rectangle ( 3.55 ,-0.750);
\draw ( 3.75 ,-0.300) node[scale =0.720] {1};
\draw (0.350,-1.35) rectangle ( 1.55 ,-1.75);
\draw ( 1.75 ,-1.30) node[scale =0.720] {1};
\end{tikzpicture}}
\otimes \raisebox{19.2pt}{\begin{tikzpicture}[scale =0.480,baseline=(current bounding box.north)]
\draw (0,0) rectangle (1,-1);
\draw (1,0) rectangle (2,-1);
\draw (0,-1) rectangle (1,-2);
\draw (1,-1) rectangle (2,-2);
\draw (0.350,-0.350) rectangle ( 1.55 ,-0.750);
\draw ( 1.75 ,-0.300) node[scale =0.720] {0};
\draw (0.350,-1.35) rectangle ( 1.55 ,-1.75);
\draw ( 1.75 ,-1.30) node[scale =0.720] {0};
\end{tikzpicture}}\\[0.4cm]
\L\ast(T_1) =4\cdot 2\cdot 2\cdot 1 = 16 & \L\ast(T_2) = 4\cdot 2\cdot 1\cdot 2 = 16 & \L\ast(T_3) = 2\cdot 2\cdot 2\cdot 2 = 16
\end{array}
\]
The values of $\L\ast$ need not all be the same. For instance, when $\sigma = (3,2)^1(1)^2$ and $\tau = (2,1,1,1)^1(1)^2$, the set $\pt^\dya_{\mathrm{sing}}(\sigma, \tau)$ consists of
\[
\begin{array}{c|c|c}
T_1 = 
\raisebox{19.2pt}{\begin{tikzpicture}[scale =0.480,baseline=(current bounding box.north)]
\draw (0,0) rectangle (1,-1);
\draw (1,0) rectangle (2,-1);
\draw (2,0) rectangle (3,-1);
\draw (0,-1) rectangle (1,-2);
\draw (1,-1) rectangle (2,-2);
\draw (0.350,-0.350) rectangle ( 1.55 ,-0.750);
\draw ( 1.75 ,-0.300) node[scale =0.720] {0};
\draw (2.25,-0.350) rectangle ( 2.65 ,-0.750);
\draw ( 2.80 ,-0.300) node[scale =0.720] {0};
\draw (0.250,-1.35) rectangle ( 0.650 ,-1.75);
\draw ( 0.800 ,-1.30) node[scale =0.720] {0};
\draw (1.25,-1.35) rectangle ( 1.65 ,-1.75);
\draw ( 1.80 ,-1.30) node[scale =0.720] {0};
\end{tikzpicture}}
\otimes \raisebox{19.2pt}{\begin{tikzpicture}[scale =0.480,baseline=(current bounding box.north)]
\draw (0,0) rectangle (1,-1);
\draw (0.250,-0.350) rectangle ( 0.650 ,-0.750);
\draw ( 0.800 ,-0.300) node[scale =0.720] {0};
\end{tikzpicture}} & T_2 = 
\raisebox{19.2pt}{\begin{tikzpicture}[scale =0.480,baseline=(current bounding box.north)]
\draw (0,0) rectangle (1,-1);
\draw (1,0) rectangle (2,-1);
\draw (2,0) rectangle (3,-1);
\draw (0,-1) rectangle (1,-2);
\draw (1,-1) rectangle (2,-2);
\draw (0.250,-0.350) rectangle ( 0.650 ,-0.750);
\draw ( 0.800 ,-0.300) node[scale =0.720] {0};
\draw (1.35,-0.350) rectangle ( 2.55 ,-0.750);
\draw ( 2.75 ,-0.300) node[scale =0.720] {0};
\draw (0.250,-1.35) rectangle ( 0.650 ,-1.75);
\draw ( 0.800 ,-1.30) node[scale =0.720] {0};
\draw (1.25,-1.35) rectangle ( 1.65 ,-1.75);
\draw ( 1.80 ,-1.30) node[scale =0.720] {0};
\end{tikzpicture}}
\otimes \raisebox{19.2pt}{\begin{tikzpicture}[scale =0.480,baseline=(current bounding box.north)]
\draw (0,0) rectangle (1,-1);
\draw (0.250,-0.350) rectangle ( 0.650 ,-0.750);
\draw ( 0.800 ,-0.300) node[scale =0.720] {0};
\end{tikzpicture}} & T_3 = 
\raisebox{19.2pt}{\begin{tikzpicture}[scale =0.480,baseline=(current bounding box.north)]
\draw (0,0) rectangle (1,-1);
\draw (1,0) rectangle (2,-1);
\draw (2,0) rectangle (3,-1);
\draw (0,-1) rectangle (1,-2);
\draw (1,-1) rectangle (2,-2);
\draw (0.250,-0.350) rectangle ( 0.650 ,-0.750);
\draw ( 0.800 ,-0.300) node[scale =0.720] {0};
\draw (1.25,-0.350) rectangle ( 1.65 ,-0.750);
\draw ( 1.80 ,-0.300) node[scale =0.720] {0};
\draw (2.25,-0.350) rectangle ( 2.65 ,-0.750);
\draw ( 2.80 ,-0.300) node[scale =0.720] {0};
\draw (0.350,-1.35) rectangle ( 1.55 ,-1.75);
\draw ( 1.75 ,-1.30) node[scale =0.720] {0};
\end{tikzpicture}}
\otimes \raisebox{19.2pt}{\begin{tikzpicture}[scale =0.480,baseline=(current bounding box.north)]
\draw (0,0) rectangle (1,-1);
\draw (0.250,-0.350) rectangle ( 0.650 ,-0.750);
\draw ( 0.800 ,-0.300) node[scale =0.720] {0};
\end{tikzpicture}}\\[0.4cm]
\L\ast(T_1) =1\cdot 1\cdot 1 = 1 & \L\ast(T_2) = 2\cdot 1 \cdot 1 = 2 & \L\ast(T_3) = 1\cdot 2\cdot 1 = 2
\end{array}
\]
Thus, the coefficient of $E^+_{(2,1,1,1)^1(1)^2}$ in $P_{(3,2)^1(1)^2}$ is $(-1)^{5-3}(1+2+2) =5$.
\end{example}

\section{Connections to Sequence in the Online Encyclopedia of Integer Sequences (OEIS)}\label{sec:oeis}
In this section, we describe some wonderful connections of our expansions involving $E^+$ to some OEIS \cite{oeis} entries. We not only prove the connection but in some cases also provide new formulas not listed on the entry's page.

Let $F$ and $G$ be a pair of plethystic bases chosen from $\{H, E, E^+, P\}$. Recall that for a polycompositions $\delta$, $\psort(\delta)$ is the type $\tau$ where each $\tau|^i$ is the partition formed by rearranging the entries of $\delta|^i$.
If we have a $G$-expansion of a plethystic basis element $F_d$ indexed by polycompositions, then we can find a $G$-expansion indexed by types by combining the terms $\delta$ which have the same $\psort(\delta)$.  Let $S^F_G$ be the set of polycompositions which index the non-zero terms in the $G$-expansion of $F$. We omit $E^+$ when it appears in the superscript or the subscript: for $F\in \{H,E,P\}$, let $S^F$ be the subset of polycompositions that appear in the $E^+$ expansion of $F$ and let $S_F$ be the set of polycompositions that appear in the $F$-expansion of $E^+$. For $d\geq 0$, define $\T^F_G(d) = \{\psort(\delta)\mid \delta\in S^F_G\}$ and $T^F_G(d) = |\T^F_G(d)|$. Explicitly, $T^F_G(d)$ counts the number of non-zero terms indexed by types that appear in the $G$-expansion of $F$. As in the notation of $S^F_G$, we omit $E^+$ in the superscript or the subscript in $\T^F_G(d)$ and $T^F_G(d)$. All of the sequences mentioned can be found on the OEIS \cite{oeis} by their sequence numbers such as \texttt{A006951}

Before we talk about the connections of the bases expansions to the OEIS, we discuss the relationship of OEIS to the counting of polycompositions.

\subsection{Number of polycompositions and A006951} The total number of polycompositions of $n$ can be computed using the formula $|\PCom(n)| = \sum_{\lambda\in \Par(n)} 2^{\ell(\lambda) - \text{dis}(\lambda)}$ where $\text{dis}(\lambda)$ is the number of {distinct} parts that occur in $\lambda$. To derive this formula, note that each polycomposition $\delta$ of $n$ can be associated with a partition $\lambda$ of $n$ such that $m_i(\lambda) = |\,\delta|^i\,|$. We may also start with such a partition $\lambda$ and find all possible polycompositions that we can generate from it. For each $i\geq 1$, we can choose composition of $m_i(\lambda)$ in $2^{m_i(\lambda) - 1}$ ways. Thus the partition $\lambda$ generates $\prod_{i\geq 1} 2^{m_i(\lambda) - 1} = 2^{\sum_i (m_i(\lambda) - 1)} = 2^{\ell(\lambda) - \text{dis}(\lambda)}$ polycompositions. To obtain all polycompositions of size $n$, we sum over all partitions of $n$. As an example, for $n = 4$, we have the partitions $(1,1,1,1)$, $(2,1,1)$, $(2,2)$, $(3,1)$ and $(4)$. We compute 
$|\PCom(4)| = 2^{4-1} + 2^{3-2} + 2^{2-1} + 2^{2-2} + 2^1 = 14$.
 The values of $|\PCom(n)|$ for $n\geq 0$ form the OEIS sequence \texttt{A006951} and count the number of conjugacy classes of $GL_n(\mathbb{F}_2)$. The first few values are \texttt{1, 1, 3, 6, 14, 27, 60, 117, 246, 490, 1002, $\ldots$} 
A related concept is known as a \textit{generalized composition} \cite{corteel} which in our terminology is a sequence of blocks.

\subsection{$E$-expansion of $E^+$ and A024786}
Each $\tau\in \T_E(d)$ is of the form $\lambda^1 (b)^2$ for some partition $\lambda$ and some non-negative integer $\lambda$. For each choice of $b\geq 0$, we can choose a partition $\lambda\in \Par(d-2b)$. This gives us 
\[
T_E(d) = |\Par(d)| + |\Par(d-2)| + |\Par(d-4)| + \ldots
\]
Let $a(n)$ for $n\geq 0$ be the $n$th term of the OEIS entry \texttt{A024786}. The description for the entry states that $a(n)$ is the number of copies of the part 2 in all integer partitions of $n$. A formula on the page computes $a(n)$ using $a(n) = \sum_{k\geq 1}|\Par(n-2k)|$.
We observe that $T_E(d) = |\Par(d)| + a(d) = a(d+2)$ for $d\geq 0$. The first few values of $T_E(d)$ starting at $d = 0$ are $\texttt{1, 3, 4, 8, 11, 19, 26, 41, 56, 83, 112, 160,\ldots}$.

\subsection{$H$-expansion of $E^+$ and A025065}
The term $a(n)$ for $n\geq 0$ of the OEIS entry \texttt{A025065} counts the number of palindromic partitions of $n$. These are partitions $\lambda$ for which there exists a composition $\alpha$ with $\sort(\alpha) = \lambda$ such that $\alpha_i = \alpha_{\ell(\alpha)-i+1}$ for all $1\leq i\leq \ell(\alpha)$. The nomenclature ``palindrome'' refers to the fact that one can rearrange the parts of the partition mirrored about a center. For instance, the partition $\lambda = (4,4,3,3,3,3,2)$ has the rearrangement $(3,3,4,2,4,3,3)$ and $\mu = (5,5,5,5,1,1)$ has the rearrangement $(5,5,1,1,5,5)$, making both $\lambda$ and $\mu$ palindromic partitions of 22. We claim that $T_H(d) = a(d)$.
To see this, consider any type $\tau = (a)^1(\lambda)^2\in \T_H(d)$ where $a\geq 0$ and $\lambda$ is a partition. Construct a unique corresponding palindromic partition $(\lambda, a, \lambda)$ where we remove any entries that are zero. This map establishes a bijection between $\T_H(d)$ and the palindromic partitions of $d$ for $d\geq 0$. Thus, we have $T_H(d) = \texttt{A025065}(d)$. The first few values of $T_H(d)$ starting at $d = 0$ are $\texttt{1, 1, 2, 2, 4, 4, 7, 7, 12, 12, 19, 19,\ldots}$.

\subsection{$P$-expansion of $E^+$ and A002513}
For $n\geq 0$, denote the $n$th term of the OEIS entry $\texttt{A002513}$ by $a(n)$. The description of the OEIS entry states that $a(n)$ is the number of \textit{cubic partitions of $n$}. In a cubic partition, even parts are of two types: marked and unmarked. It is implicit in the description that the position of the marked part does not matter, and we write our partitions such that the  marked part $i$ appears before the unmarked part $i$. For instance, $(4',4',3,2',2,1)$, $(4',4,3,2,2,1)$, and $(4',4',4',4)$ are partitions of 16 with marked parts. With a cubic partition $\rho$, associate a type $\tau = \lambda \mu^2$ where $\lambda$ is the partition formed by the unmarked parts of $\rho$ and $\mu$ is formed by halving each marked part (and removing the markings). So, the types obtained from the above examples are $(3,2,1)^1(2,2,1)^2$, $(4,3,2,2,1)^1(2)^2$ and $(4)^1(2,2,2)^2$ respectively. This map describes a bijection between $\T_P(d)$ and cubic partitions of $d$ for all $d\geq 0$. Thus, $T_P(d) = \texttt{A002513}(d)$. The first few values of $T_P(d)$ starting at $d = 0$ are $\texttt{1, 1, 3, 4, 9, 12, 23, 31, 54, 73, 118, 159,\ldots}$.

\subsection{$E^+$-expansion of $H$ and A018819}
Let $a(d)$ be the $d$th term of the OEIS entry \texttt{A018819}.
The term $a(d)$ counts the number of partitions $\lambda$ of $d$ where all the parts are powers of 2. The quantity $T^H(d)$ counts the number of types $\tau$ of $d$ where the multiplicities are powers of 2 and for each $i$, $\tau|^i \leq 1$. Recall $m_i(\lambda)$ is the number of times a part $i$ appears in $\lambda$.
If we have a partition $\lambda$ where all parts are powers of 2, then we construct a corresponding $\tau = (m_1(\lambda))^1(m_2(\lambda))^2(m_4(\lambda))^4\ldots (m_{2^k}(\lambda))^{2^k}\ldots$. This provides a bijection between $\T^H(d)$ and partitions of $d$ where all parts are powers of 2, thus giving us $T^H(d) = \texttt{A018819}(d)$. The first few values of $T_P(d)$ starting at $d = 0$ are $\texttt{1, 1, 2, 2, 4, 4, 6, 6, 10, 10, 14, 14,\ldots}$.

\subsection{$E^+$-expansion of $E$ and A092119}
The set $\T^E(d)$ is the number of dyadic types of size $d$, that is, types of $d$ where all multiplicities are powers of 2.
 We describe a formula to compute $T^E(n)$. To create a dyadic type $\tau $ of $n$, we start with a partition $\lambda$ consisting of parts which are powers of 2. For each part $2^k$, we find a partition $\lambda^{(k)}$ of $m_{2^k}(\lambda)$ (where $m_i(\lambda)$ is the number of times the part $i$ appears in $\lambda$). We then construct $\tau = (\lambda^{(0)})^{1}(\lambda^{(1)})^{2}\ldots (\lambda^{(k)})^{2^k}\ldots$.  This gives us the formula
\[
T^E(n) = \sum\limits_{\lambda} \prod\limits_{k\geq 0} |\Par(m_k(\lambda))|
\]
where the sum is over all partitions $\lambda$ of $n$ with parts equal to powers of 2. For instance, we have the following partitions (in standard as well as exponential notation) of 5 whose parts are powers of 2: $41 = (4,1), 2^21 = (2,2,1),21^3 = (2,1,1,1), 1^5 = (1,1,1,1,1)$. The partition $(2,1,1,1)$ generates the types $\{(3)^1(1)^2, (2,1)^1 (1)^2, (1,1,1)^1(1)^2\}$ which is enumerated by $|\Par(1)|\cdot |\Par(3)|$ owing to 2 appearing once and 1 appearing three times. We compute $a(5) = 13$ by
\[
|\Par(1)|\cdot |\Par(1)| + |\Par(2)|\cdot |\Par(1)| + |\Par(1)|\cdot |\Par(3)| + |\Par(5)| = 1 + 2 + 3 + 7  .
\]
Let $a(n)$ denote the $n$th term of the OEIS entry $ \texttt{A092119}$. If $A(x) = \sum_{d\geq 0 } a(d) x^d$ and $\pi(x) = \sum_{n\geq 0} |\Par(n)|x^n$, then the entry states that $A(x) = \prod_{i\geq 0} \pi(x^{2^i})$. So, $A(x)= \prod_{i\geq 0} \sum_{j\geq 0} |\Par(j)| x^{2^ij}$. If $\lambda$ is a partition with parts that are powers of 2, we interpret the index $j$ in the above sum over $j$ as using the part $2^z$ $j$ times in $\lambda$, namely, $m_{2^z}(\lambda)$. Expanding the product over $i$ using this interpretation recovers the formula for $T^E(n)$. The first few values of $T^E(d)$ starting at $d = 0$ are \texttt{1, 1, 3, 4, 10, 13, 26, 35, 66, 88, 150, 202...}.

\subsection{$E^+$-expansion of $P$ and A305841}
The set $\T^P(d)$ is the number of types with a single multiplicity that is a power of 2 and denotes its count by $A(d) = T^P(d)$. To compute $A(d)$, we choose a multiplicity $2^k$ dividing $d$ and consider all partitions of $d/2^k$ to form the degrees. Let $p(n)$ denote $|\Par(n)|$ and $p(n) = 0$ whenever $n$ is not a non-negative integer. Our discussion yields the formula 
\[
A(d) = p(d) + p(d/2) + p(d/4) + \ldots
\]
Let $a(n)$ be the $n$th term of the OEIS entry \texttt{A305841} which is defined in relation to the generating function of ``partitions of partitions'' as
\[
\prod\limits_{n \geq 1}(1+x^n)^{a(n)} = \prod\limits_{n\geq 1} (1-x^n)^{-p(n)}.
\]
We now prove that this relation also holds true for $A(n)$.
\begin{prop}
We have
\[
\prod\limits_{n \geq 1}(1+x^n)^{A(n)} = \prod\limits_{n\geq 1} (1-x^n)^{-p(n)}.
\]
\end{prop}
\begin{proof}
    We rewrite the left-hand side as \[
    \prod\limits_{n\geq 1}\dfrac{(1-x^{2n})^{A(n)}}{(1-x^n)^{A(n)}} = \prod\limits_{\text{odd } n} \dfrac{1}{(1-x^n)^{A(n)}} \prod\limits_{k\geq 1} \dfrac{(1-x^{2k})^{A(k)}}{(1-x^{2k})^{A(2k)}}.
    \]
From the explicit formula for $A(n)$, we can deduce $A(2k) = p(2k) + A(k)$. So the second factor becomes \[\prod\limits_{k\geq 1} \dfrac{(1-x^{2k})^{A(k)}}{(1-x^{2k})^{A(k) + p(2k)}}=\prod_{\text{even } n}\dfrac{1}{(1-x^n)^{p(n)}}.\] For an odd integer $n$, $A(n) = p(n)$ and so $\prod_{\text{odd } n} \dfrac{1}{(1-x^n)^{A(n)}}$ is equal to $\prod_{ \text{odd } n} \dfrac{1}{(1-x^n)^{p(n)}}$. Combining these two gives us the right hand side in the statement.
\end{proof}
We now show that the sequences given by $A(n)$ and $a(n)$ are the same.
\begin{prop}
For $n\geq 1$, $A(n) = a(n)$.
\end{prop}
\begin{proof}
We have the equality of functional identities
\[
\prod\limits_{n \geq 1}(1+x^n)^{A(n)} = \prod\limits_{n \geq 1}(1+x^n)^{a(n)}
\]
We rewrite this as get $\prod_{n \geq 1}(1+x^n)^{A(n)-a(n)} = 1$.
We take formal $\log$ on both sides to obtain
\[
\sum\limits_{n\geq 1} (A(n)-a(n)) \log(1+x^n) = 0
\]
Using $\log(1+x) = \sum_{i\geq 1} (-1)^{i-1}\frac{x^i}{i}$, we rewrite the equality as
\[
\sum\limits_{n\geq 1}\sum\limits_{i\geq 1} (-1)^{i-1} (A(n)-a(n))\dfrac{x^{ni}}{i} = 0.
\]
We can collect the terms of the same degree and this gives us a formal power series with a double summation
\[
 \sum\limits_{n\geq 1} \left(\sum\limits_{k\mid n}\dfrac{(-1)^{n/k-1}(A(k)-a(k))}{n/k}\right) x^n = 0.
\]
As the formal power series is zero, each coefficient must be zero. This shows that for all $n\geq 1$, $\sum_{k\mid n}(-1)^{n/k-1}k(A(k)-a(k)) = 0$. Define $g(k) = k(A(k) - a(k))$. Let $n = 1$. As 1 is the only divisor of 1, we get $g(1) = 0$. For any prime $p$, we have $g(p) + (-1)^{p-1}g(1) = 0$ which shows $g(p) = 0$. We now induct on the number of prime factors (counted with multiplicity). Assume that $g(n) = 0$ for all $n$ with $l$ prime factors. Let $N$ have $l+1$ prime factors. Then, $\sum_{k\mid N} (-1)^{N/k-1} g(k) = 0$ can be rewritten as $\sum_{\substack{k \mid N\\ k\neq N}} (-1)^{N/k-1}g(k) + g(N) = 0$. The first sum is over all $g(k)$ where $k$ contains at most $l$ prime factors and thus is zero by our induction hypothesis. This shows that $g(N) = 0$ and $g(n) = 0$ for all $n\geq 1$. From this we obtain $A(n) = a(n)$ for all $n\geq 1$.
\end{proof}

A similar manipulation of the functional identities leads to another relation between $a(n) = A(n)$ and $p(n)$.
\begin{prop}
    For all natural numbers $m$,
    \[
    \sum\limits_{k\mid m} k\left( p(k) + (-1)^{m/k}a(k)\right) = 0.
    \]
\end{prop}
\begin{proof}
    We take the log of both sides of the functional identity to get
    \[
    \sum\limits_{n\geq 1} a(n) \log(1+x^n) = \sum\limits_{n\geq 1} -p(n) \log(1-x^n).
    \]
    We use the expansions $\log(1+x) =\sum_{i\geq 1} (-1)^{i-1} x^i/i$ and $-\log(1-x) = \sum_{i\geq 1} x^i/i$ to obtain
    \[
    \sum\limits_{n\geq 1} \sum\limits_{i\geq 1} (-1)^{i-1} a(n)\dfrac{x^{ni}}{i} = \sum\limits_{n\geq 1} \sum\limits_{i\geq 1}p(n)\dfrac{x^{ni}}{i}.
    \] 
    Now collect the coefficients for each $x^m$, we get
    \[
    \sum\limits_{m\geq 1} \left(\sum\limits_{k\mid m}\dfrac{(-1)^{k-1}a(k)}{m/k}\right) x^m = \sum\limits_{m\geq 1} \left(\sum\limits_{k\mid m}\dfrac{p(k)}{m/k}\right) x^m
    \]
and equate the coefficient of $x^m$ on both sides to find
    \[
    \sum\limits_{k\mid m} (-1)^{m/k - 1} ka(k) = \sum\limits_{k \mid m} kp(k)
    \]
    which can be rearranged to give the statement.
\end{proof}
\begin{example}
    Let $m = 12$. We list the divisors and their corresponding $a$ and $p$ values below.
    \begin{table}[H]
        \centering
        \begin{tabular}{c|c|c|c|c|c|c}
            Divisors, $k$ &  1 & 2 & 3 & 4 & 6 & 12\\
            \hline
           $a(k)$ &  1 & 3 & 3 & 8 & 14 & 91\\
           \hline
           $p(k)$ &  1 & 2 & 3 & 5 & 11 & 77
        \end{tabular}
    \end{table}
Plugging this into the above equation gives 
\[
1(1+1) + 2(2+3) + 3(3+3) + 4(5-8) + 6(11+14) + 12(77 - 91) = 180 - 180 = 0.
 \]
\end{example}

\section{Acknowledgments}
I would like to thank my advisor Dr Nicholas Loehr for his extremely detailed and helpful suggestions for this manuscript.
\newpage

\section{Appendix Sample Transition Matrices for $n = 4$}
In this section, we present the transition matrices between the 12 pairs of bases discuss above. We label the matrices by $\mcM(F,G)$ wherein the entry in column $\sigma$ and row $\tau$ is the coefficient of $E_\tau$ in $H_\sigma$. So, for instance, in $\mcM(H,E)$, we find the $E$-expansion of $H_\sigma$ by reading down the column as follows:
\[
H_{(3,1)} = E_{(1,1,1,1)^1} - 2E_{(2,1,1)^1} + E_{(3,1)^1}.
\]
\renewcommand{\quad}{\hspace{0.8em}}

\begin{footnotesize} 
{
\begin{align*}
&\mcM(H,E) &\mcM(E,H)\\[1cm]
&\bbmatrix{
~ & \rot{(1)^4} & \rot{(1)^1(1)^3} & \rot{(1,1)^2} & \rot{(1,1)^1(1)^2} & \rot{(1,1,1,1)^1} & \rot{(2)^1(1)^2} & \rot{(2,1,1)^1} & \rot{(3,1)^1} & \rot{(2)^2} & \rot{(2,2)^1} & \rot{(4)^1} \cr
(1)^4 & -1 & 0 & 0 & 0 & 0 & 0 & 0 & 0 & 0 & 0 & 0 \cr
(1)^1(1)^3 & 0 & 1 & 0 & 0 & 0 & 0 & 0 & 0 & 0 & 0 & 0 \cr
(1,1)^2 & 0 & 0 & 1 & 0 & 0 & 0 & 0 & 0 & 1 & 0 & 0 \cr
(1,1)^1(1)^2 & 0 & 0 & 0 & -1 & 0 & -1 & 0 & 0 & 0 & 0 & 0 \cr
(1,1,1,1)^1 & 0 & 0 & 0 & 0 & 1 & 0 & 1 & 1 & 0 & 1 & 1 \cr
(2)^1(1)^2 & 0 & 0 & 0 & 0 & 0 & 1 & 0 & 0 & 0 & 0 & 0 \cr
(2,1,1)^1 & 0 & 0 & 0 & 0 & 0 & 0 & -1 & -2 & 0 & -2 & -3 \cr
(3,1)^1 & 0 & 0 & 0 & 0 & 0 & 0 & 0 & 1 & 0 & 0 & 2 \cr
(2)^2 & 0 & 0 & 0 & 0 & 0 & 0 & 0 & 0 & -1 & 0 & 0 \cr
(2,2)^1 & 0 & 0 & 0 & 0 & 0 & 0 & 0 & 0 & 0 & 1 & 1 \cr
(4)^1 & 0 & 0 & 0 & 0 & 0 & 0 & 0 & 0 & 0 & 0 & -1 \cr
}
&\qquad
\bbmatrix{
~ & \rot{(1)^4} & \rot{(1)^1(1)^3} & \rot{(1,1)^2} & \rot{(1,1)^1(1)^2} & \rot{(1,1,1,1)^1} & \rot{(2)^1(1)^2} & \rot{(2,1,1)^1} & \rot{(3,1)^1} & \rot{(2)^2} & \rot{(2,2)^1} & \rot{(4)^1} \cr
(1)^4 & -1 & 0 & 0 & 0 & 0 & 0 & 0 & 0 & 0 & 0 & 0 \cr
(1)^1(1)^3 & 0 & 1 & 0 & 0 & 0 & 0 & 0 & 0 & 0 & 0 & 0 \cr
(1,1)^2 & 0 & 0 & 1 & 0 & 0 & 0 & 0 & 0 & 1 & 0 & 0 \cr
(1,1)^1(1)^2 & 0 & 0 & 0 & -1 & 0 & -1 & 0 & 0 & 0 & 0 & 0 \cr
(1,1,1,1)^1 & 0 & 0 & 0 & 0 & 1 & 0 & 1 & 1 & 0 & 1 & 1 \cr
(2)^1(1)^2 & 0 & 0 & 0 & 0 & 0 & 1 & 0 & 0 & 0 & 0 & 0 \cr
(2,1,1)^1 & 0 & 0 & 0 & 0 & 0 & 0 & -1 & -2 & 0 & -2 & -3 \cr
(3,1)^1 & 0 & 0 & 0 & 0 & 0 & 0 & 0 & 1 & 0 & 0 & 2 \cr
(2)^2 & 0 & 0 & 0 & 0 & 0 & 0 & 0 & 0 & -1 & 0 & 0 \cr
(2,2)^1 & 0 & 0 & 0 & 0 & 0 & 0 & 0 & 0 & 0 & 1 & 1 \cr
(4)^1 & 0 & 0 & 0 & 0 & 0 & 0 & 0 & 0 & 0 & 0 & -1 \cr
}
\end{align*} }
\end{footnotesize}

\begin{footnotesize} 
{
\begin{align*}
&\mcM(P,H) &\mcM(P,E)\\[1cm]
&\bbmatrix{
~ & \rot{(1)^4} & \rot{(1)^1(1)^3} & \rot{(1,1)^2} & \rot{(1,1)^1(1)^2} & \rot{(1,1,1,1)^1} & \rot{(2)^1(1)^2} & \rot{(2,1,1)^1} & \rot{(3,1)^1} & \rot{(2)^2} & \rot{(2,2)^1} & \rot{(4)^1} \cr
(1)^4 & 1 & 0 & 0 & 0 & 0 & 0 & 0 & 0 & 0 & 0 & 0 \cr
(1)^1(1)^3 & 0 & 1 & 0 & 0 & 0 & 0 & 0 & 0 & 0 & 0 & 0 \cr
(1,1)^2 & 0 & 0 & 1 & 0 & 0 & 0 & 0 & 0 & -1 & 0 & 0 \cr
(1,1)^1(1)^2 & 0 & 0 & 0 & 1 & 0 & -1 & 0 & 0 & 0 & 0 & 0 \cr
(1,1,1,1)^1 & 0 & 0 & 0 & 0 & 1 & 0 & -1 & 1 & 0 & 1 & -1 \cr
(2)^1(1)^2 & 0 & 0 & 0 & 0 & 0 & 2 & 0 & 0 & 0 & 0 & 0 \cr
(2,1,1)^1 & 0 & 0 & 0 & 0 & 0 & 0 & 2 & -3 & 0 & -4 & 4 \cr
(3,1)^1 & 0 & 0 & 0 & 0 & 0 & 0 & 0 & 3 & 0 & 0 & -4 \cr
(2)^2 & 0 & 0 & 0 & 0 & 0 & 0 & 0 & 0 & 2 & 0 & 0 \cr
(2,2)^1 & 0 & 0 & 0 & 0 & 0 & 0 & 0 & 0 & 0 & 4 & -2 \cr
(4)^1 & 0 & 0 & 0 & 0 & 0 & 0 & 0 & 0 & 0 & 0 & 4 \cr
}
&\qquad
\bbmatrix{
~ & \rot{(1)^4} & \rot{(1)^1(1)^3} & \rot{(1,1)^2} & \rot{(1,1)^1(1)^2} & \rot{(1,1,1,1)^1} & \rot{(2)^1(1)^2} & \rot{(2,1,1)^1} & \rot{(3,1)^1} & \rot{(2)^2} & \rot{(2,2)^1} & \rot{(4)^1} \cr
(1)^4 & -1 & 0 & 0 & 0 & 0 & 0 & 0 & 0 & 0 & 0 & 0 \cr
(1)^1(1)^3 & 0 & 1 & 0 & 0 & 0 & 0 & 0 & 0 & 0 & 0 & 0 \cr
(1,1)^2 & 0 & 0 & 1 & 0 & 0 & 0 & 0 & 0 & 1 & 0 & 0 \cr
(1,1)^1(1)^2 & 0 & 0 & 0 & -1 & 0 & -1 & 0 & 0 & 0 & 0 & 0 \cr
(1,1,1,1)^1 & 0 & 0 & 0 & 0 & 1 & 0 & 1 & 1 & 0 & 1 & 1 \cr
(2)^1(1)^2 & 0 & 0 & 0 & 0 & 0 & 2 & 0 & 0 & 0 & 0 & 0 \cr
(2,1,1)^1 & 0 & 0 & 0 & 0 & 0 & 0 & -2 & -3 & 0 & -4 & -4 \cr
(3,1)^1 & 0 & 0 & 0 & 0 & 0 & 0 & 0 & 3 & 0 & 0 & 4 \cr
(2)^2 & 0 & 0 & 0 & 0 & 0 & 0 & 0 & 0 & -2 & 0 & 0 \cr
(2,2)^1 & 0 & 0 & 0 & 0 & 0 & 0 & 0 & 0 & 0 & 4 & 2 \cr
(4)^1 & 0 & 0 & 0 & 0 & 0 & 0 & 0 & 0 & 0 & 0 & -4 \cr
}
\end{align*} }
\end{footnotesize}

\begin{footnotesize} 
{
\begin{align*}
&\mcM(H,P) &\mcM(E,P)\\[1cm]
&\bbmatrix{
~ & \rot{(1)^4} & \rot{(1)^1(1)^3} & \rot{(1,1)^2} & \rot{(1,1)^1(1)^2} & \rot{(1,1,1,1)^1} & \rot{(2)^1(1)^2} & \rot{(2,1,1)^1} & \rot{(3,1)^1} & \rot{(2)^2} & \rot{(2,2)^1} & \rot{(4)^1} \cr
(1)^4 & 1 & 0 & 0 & 0 & 0 & 0 & 0 & 0 & 0 & 0 & 0 \cr
(1)^1(1)^3 & 0 & 1 & 0 & 0 & 0 & 0 & 0 & 0 & 0 & 0 & 0 \cr
(1,1)^2 & 0 & 0 & 1 & 0 & 0 & 0 & 0 & 0 & \frac{1}{2} & 0 & 0 \cr
(1,1)^1(1)^2 & 0 & 0 & 0 & 1 & 0 & \frac{1}{2} & 0 & 0 & 0 & 0 & 0 \cr
(1,1,1,1)^1 & 0 & 0 & 0 & 0 & 1 & 0 & \frac{1}{2} & \frac{1}{6} & 0 & \frac{1}{4} & \frac{1}{24} \cr
(2)^1(1)^2 & 0 & 0 & 0 & 0 & 0 & \frac{1}{2} & 0 & 0 & 0 & 0 & 0 \cr
(2,1,1)^1 & 0 & 0 & 0 & 0 & 0 & 0 & \frac{1}{2} & \frac{1}{2} & 0 & \frac{1}{2} & \frac{1}{4} \cr
(3,1)^1 & 0 & 0 & 0 & 0 & 0 & 0 & 0 & \frac{1}{3} & 0 & 0 & \frac{1}{3} \cr
(2)^2 & 0 & 0 & 0 & 0 & 0 & 0 & 0 & 0 & \frac{1}{2} & 0 & 0 \cr
(2,2)^1 & 0 & 0 & 0 & 0 & 0 & 0 & 0 & 0 & 0 & \frac{1}{4} & \frac{1}{8} \cr
(4)^1 & 0 & 0 & 0 & 0 & 0 & 0 & 0 & 0 & 0 & 0 & \frac{1}{4} \cr
}
&\qquad
\bbmatrix{
~ & \rot{(1)^4} & \rot{(1)^1(1)^3} & \rot{(1,1)^2} & \rot{(1,1)^1(1)^2} & \rot{(1,1,1,1)^1} & \rot{(2)^1(1)^2} & \rot{(2,1,1)^1} & \rot{(3,1)^1} & \rot{(2)^2} & \rot{(2,2)^1} & \rot{(4)^1} \cr
(1)^4 & -1 & 0 & 0 & 0 & 0 & 0 & 0 & 0 & 0 & 0 & 0 \cr
(1)^1(1)^3 & 0 & 1 & 0 & 0 & 0 & 0 & 0 & 0 & 0 & 0 & 0 \cr
(1,1)^2 & 0 & 0 & 1 & 0 & 0 & 0 & 0 & 0 & \frac{1}{2} & 0 & 0 \cr
(1,1)^1(1)^2 & 0 & 0 & 0 & -1 & 0 & -\frac{1}{2} & 0 & 0 & 0 & 0 & 0 \cr
(1,1,1,1)^1 & 0 & 0 & 0 & 0 & 1 & 0 & \frac{1}{2} & \frac{1}{6} & 0 & \frac{1}{4} & \frac{1}{24} \cr
(2)^1(1)^2 & 0 & 0 & 0 & 0 & 0 & \frac{1}{2} & 0 & 0 & 0 & 0 & 0 \cr
(2,1,1)^1 & 0 & 0 & 0 & 0 & 0 & 0 & -\frac{1}{2} & -\frac{1}{2} & 0 & -\frac{1}{2} & -\frac{1}{4} \cr
(3,1)^1 & 0 & 0 & 0 & 0 & 0 & 0 & 0 & \frac{1}{3} & 0 & 0 & \frac{1}{3} \cr
(2)^2 & 0 & 0 & 0 & 0 & 0 & 0 & 0 & 0 & -\frac{1}{2} & 0 & 0 \cr
(2,2)^1 & 0 & 0 & 0 & 0 & 0 & 0 & 0 & 0 & 0 & \frac{1}{4} & \frac{1}{8} \cr
(4)^1 & 0 & 0 & 0 & 0 & 0 & 0 & 0 & 0 & 0 & 0 & -\frac{1}{4} \cr
}
\end{align*} }
\end{footnotesize}

\begin{footnotesize} 
{
\begin{align*}
&\mcM(E^+,E) &\mcM(E^+,H)\\[1cm]
&\bbmatrix{
~ & \rot{(1)^4} & \rot{(1)^1(1)^3} & \rot{(1,1)^2} & \rot{(1,1)^1(1)^2} & \rot{(1,1,1,1)^1} & \rot{(2)^1(1)^2} & \rot{(2,1,1)^1} & \rot{(3,1)^1} & \rot{(2)^2} & \rot{(2,2)^1} & \rot{(4)^1} \cr
(1)^4 & -1 & 0 & 0 & 0 & 0 & 0 & 0 & 0 & 1 & 0 & 0 \cr
(1)^1(1)^3 & 0 & 1 & 0 & 0 & 0 & 0 & 0 & 0 & 0 & 0 & 0 \cr
(1,1)^2 & 0 & 0 & 1 & 0 & 0 & -1 & 0 & 0 & 1 & 1 & 0 \cr
(1,1)^1(1)^2 & 0 & 0 & 0 & -1 & 0 & -1 & 1 & 1 & 0 & 2 & 1 \cr
(1,1,1,1)^1 & 0 & 0 & 0 & 0 & 1 & 0 & 1 & 1 & 0 & 1 & 1 \cr
(2)^1(1)^2 & 0 & 0 & 0 & 0 & 0 & 1 & 0 & 0 & 0 & -2 & -1 \cr
(2,1,1)^1 & 0 & 0 & 0 & 0 & 0 & 0 & -1 & -2 & 0 & -2 & -3 \cr
(3,1)^1 & 0 & 0 & 0 & 0 & 0 & 0 & 0 & 1 & 0 & 0 & 2 \cr
(2)^2 & 0 & 0 & 0 & 0 & 0 & 0 & 0 & 0 & -1 & 0 & 1 \cr
(2,2)^1 & 0 & 0 & 0 & 0 & 0 & 0 & 0 & 0 & 0 & 1 & 1 \cr
(4)^1 & 0 & 0 & 0 & 0 & 0 & 0 & 0 & 0 & 0 & 0 & -1 \cr
}
&\qquad
\bbmatrix{
~ & \rot{(1)^4} & \rot{(1)^1(1)^3} & \rot{(1,1)^2} & \rot{(1,1)^1(1)^2} & \rot{(1,1,1,1)^1} & \rot{(2)^1(1)^2} & \rot{(2,1,1)^1} & \rot{(3,1)^1} & \rot{(2)^2} & \rot{(2,2)^1} & \rot{(4)^1} \cr
(1)^4 & 1 & 0 & 0 & 0 & 0 & 0 & 0 & 0 & -1 & 0 & 0 \cr
(1)^1(1)^3 & 0 & 1 & 0 & 0 & 0 & 0 & 0 & 0 & 0 & 0 & 0 \cr
(1,1)^2 & 0 & 0 & 1 & 0 & 0 & -1 & 0 & 0 & 0 & 1 & 1 \cr
(1,1)^1(1)^2 & 0 & 0 & 0 & 1 & 0 & 0 & -1 & -1 & 0 & 0 & 0 \cr
(1,1,1,1)^1 & 0 & 0 & 0 & 0 & 1 & 0 & 0 & 0 & 0 & 0 & 0 \cr
(2)^1(1)^2 & 0 & 0 & 0 & 0 & 0 & 1 & 0 & 0 & 0 & -2 & -1 \cr
(2,1,1)^1 & 0 & 0 & 0 & 0 & 0 & 0 & 1 & 0 & 0 & 0 & 0 \cr
(3,1)^1 & 0 & 0 & 0 & 0 & 0 & 0 & 0 & 1 & 0 & 0 & 0 \cr
(2)^2 & 0 & 0 & 0 & 0 & 0 & 0 & 0 & 0 & 1 & 0 & -1 \cr
(2,2)^1 & 0 & 0 & 0 & 0 & 0 & 0 & 0 & 0 & 0 & 1 & 0 \cr
(4)^1 & 0 & 0 & 0 & 0 & 0 & 0 & 0 & 0 & 0 & 0 & 1 \cr
}
\end{align*} }
\end{footnotesize}

\begin{footnotesize} 
{
\begin{align*}
&\mcM(E^+,P) &\mcM(E,E^+)\\[1cm]
&\bbmatrix{
~ & \rot{(1)^4} & \rot{(1)^1(1)^3} & \rot{(1,1)^2} & \rot{(1,1)^1(1)^2} & \rot{(1,1,1,1)^1} & \rot{(2)^1(1)^2} & \rot{(2,1,1)^1} & \rot{(3,1)^1} & \rot{(2)^2} & \rot{(2,2)^1} & \rot{(4)^1} \cr
(1)^4 & 1 & 0 & 0 & 0 & 0 & 0 & 0 & 0 & -1 & 0 & 0 \cr
(1)^1(1)^3 & 0 & 1 & 0 & 0 & 0 & 0 & 0 & 0 & 0 & 0 & 0 \cr
(1,1)^2 & 0 & 0 & 1 & 0 & 0 & -1 & 0 & 0 & \frac{1}{2} & 1 & \frac{1}{2} \cr
(1,1)^1(1)^2 & 0 & 0 & 0 & 1 & 0 & \frac{1}{2} & -1 & -1 & 0 & -1 & -\frac{1}{2} \cr
(1,1,1,1)^1 & 0 & 0 & 0 & 0 & 1 & 0 & \frac{1}{2} & \frac{1}{6} & 0 & \frac{1}{4} & \frac{1}{24} \cr
(2)^1(1)^2 & 0 & 0 & 0 & 0 & 0 & \frac{1}{2} & 0 & 0 & 0 & -1 & -\frac{1}{2} \cr
(2,1,1)^1 & 0 & 0 & 0 & 0 & 0 & 0 & \frac{1}{2} & \frac{1}{2} & 0 & \frac{1}{2} & \frac{1}{4} \cr
(3,1)^1 & 0 & 0 & 0 & 0 & 0 & 0 & 0 & \frac{1}{3} & 0 & 0 & \frac{1}{3} \cr
(2)^2 & 0 & 0 & 0 & 0 & 0 & 0 & 0 & 0 & \frac{1}{2} & 0 & -\frac{1}{2} \cr
(2,2)^1 & 0 & 0 & 0 & 0 & 0 & 0 & 0 & 0 & 0 & \frac{1}{4} & \frac{1}{8} \cr
(4)^1 & 0 & 0 & 0 & 0 & 0 & 0 & 0 & 0 & 0 & 0 & \frac{1}{4} \cr
}
&\qquad
\bbmatrix{
~ & \rot{(1)^4} & \rot{(1)^1(1)^3} & \rot{(1,1)^2} & \rot{(1,1)^1(1)^2} & \rot{(1,1,1,1)^1} & \rot{(2)^1(1)^2} & \rot{(2,1,1)^1} & \rot{(3,1)^1} & \rot{(2)^2} & \rot{(2,2)^1} & \rot{(4)^1} \cr
(1)^4 & -1 & 0 & 0 & 0 & 0 & 0 & 0 & 0 & -1 & 0 & -1 \cr
(1)^1(1)^3 & 0 & 1 & 0 & 0 & 0 & 0 & 0 & 0 & 0 & 0 & 0 \cr
(1,1)^2 & 0 & 0 & 1 & 0 & 0 & 1 & 0 & 0 & 1 & 1 & 1 \cr
(1,1)^1(1)^2 & 0 & 0 & 0 & -1 & 0 & -1 & -1 & -1 & 0 & -2 & -1 \cr
(1,1,1,1)^1 & 0 & 0 & 0 & 0 & 1 & 0 & 1 & 1 & 0 & 1 & 1 \cr
(2)^1(1)^2 & 0 & 0 & 0 & 0 & 0 & 1 & 0 & 0 & 0 & 2 & 1 \cr
(2,1,1)^1 & 0 & 0 & 0 & 0 & 0 & 0 & -1 & -2 & 0 & -2 & -3 \cr
(3,1)^1 & 0 & 0 & 0 & 0 & 0 & 0 & 0 & 1 & 0 & 0 & 2 \cr
(2)^2 & 0 & 0 & 0 & 0 & 0 & 0 & 0 & 0 & -1 & 0 & -1 \cr
(2,2)^1 & 0 & 0 & 0 & 0 & 0 & 0 & 0 & 0 & 0 & 1 & 1 \cr
(4)^1 & 0 & 0 & 0 & 0 & 0 & 0 & 0 & 0 & 0 & 0 & -1 \cr
}
\end{align*} }
\end{footnotesize}

\begin{footnotesize} 
{
\begin{align*}
&\mcM(H,E^+) &\mcM(P,E^+)\\[1cm]
&\bbmatrix{
~ & \rot{(1)^4} & \rot{(1)^1(1)^3} & \rot{(1,1)^2} & \rot{(1,1)^1(1)^2} & \rot{(1,1,1,1)^1} & \rot{(2)^1(1)^2} & \rot{(2,1,1)^1} & \rot{(3,1)^1} & \rot{(2)^2} & \rot{(2,2)^1} & \rot{(4)^1} \cr
(1)^4 & 1 & 0 & 0 & 0 & 0 & 0 & 0 & 0 & 1 & 0 & 1 \cr
(1)^1(1)^3 & 0 & 1 & 0 & 0 & 0 & 0 & 0 & 0 & 0 & 0 & 0 \cr
(1,1)^2 & 0 & 0 & 1 & 0 & 0 & 1 & 0 & 0 & 0 & 1 & 0 \cr
(1,1)^1(1)^2 & 0 & 0 & 0 & 1 & 0 & 0 & 1 & 1 & 0 & 0 & 0 \cr
(1,1,1,1)^1 & 0 & 0 & 0 & 0 & 1 & 0 & 0 & 0 & 0 & 0 & 0 \cr
(2)^1(1)^2 & 0 & 0 & 0 & 0 & 0 & 1 & 0 & 0 & 0 & 2 & 1 \cr
(2,1,1)^1 & 0 & 0 & 0 & 0 & 0 & 0 & 1 & 0 & 0 & 0 & 0 \cr
(3,1)^1 & 0 & 0 & 0 & 0 & 0 & 0 & 0 & 1 & 0 & 0 & 0 \cr
(2)^2 & 0 & 0 & 0 & 0 & 0 & 0 & 0 & 0 & 1 & 0 & 1 \cr
(2,2)^1 & 0 & 0 & 0 & 0 & 0 & 0 & 0 & 0 & 0 & 1 & 0 \cr
(4)^1 & 0 & 0 & 0 & 0 & 0 & 0 & 0 & 0 & 0 & 0 & 1 \cr
}
&\qquad
\bbmatrix{
~ & \rot{(1)^4} & \rot{(1)^1(1)^3} & \rot{(1,1)^2} & \rot{(1,1)^1(1)^2} & \rot{(1,1,1,1)^1} & \rot{(2)^1(1)^2} & \rot{(2,1,1)^1} & \rot{(3,1)^1} & \rot{(2)^2} & \rot{(2,2)^1} & \rot{(4)^1} \cr
(1)^4 & 1 & 0 & 0 & 0 & 0 & 0 & 0 & 0 & 2 & 0 & 4 \cr
(1)^1(1)^3 & 0 & 1 & 0 & 0 & 0 & 0 & 0 & 0 & 0 & 0 & 0 \cr
(1,1)^2 & 0 & 0 & 1 & 0 & 0 & 2 & 0 & 0 & -1 & 4 & -2 \cr
(1,1)^1(1)^2 & 0 & 0 & 0 & 1 & 0 & -1 & 2 & 0 & 0 & -4 & 0 \cr
(1,1,1,1)^1 & 0 & 0 & 0 & 0 & 1 & 0 & -1 & 1 & 0 & 1 & -1 \cr
(2)^1(1)^2 & 0 & 0 & 0 & 0 & 0 & 2 & 0 & 0 & 0 & 8 & 0 \cr
(2,1,1)^1 & 0 & 0 & 0 & 0 & 0 & 0 & 2 & -3 & 0 & -4 & 4 \cr
(3,1)^1 & 0 & 0 & 0 & 0 & 0 & 0 & 0 & 3 & 0 & 0 & -4 \cr
(2)^2 & 0 & 0 & 0 & 0 & 0 & 0 & 0 & 0 & 2 & 0 & 4 \cr
(2,2)^1 & 0 & 0 & 0 & 0 & 0 & 0 & 0 & 0 & 0 & 4 & -2 \cr
(4)^1 & 0 & 0 & 0 & 0 & 0 & 0 & 0 & 0 & 0 & 0 & 4 \cr
}
\end{align*} }
\end{footnotesize}
 

\begin{thebibliography}{20} 

\bibitem{polysymm}
Asvin G and Andrew O'Desky, ``Configuration spaces, graded spaces,
 and polysymmetric functions,'' \textit{Arxiv}(2022 v3) \url{https://arxiv.org/abs/2207.04529}.
 
\bibitem{macd} Ian Macdonald, \emph{Symmetric Functions and Hall Polynomials}
(second ed.), Oxford University Press (1995).

\bibitem{alg-ubp} Rosa Orellana, Franco Saliola, Anne Schilling,
 and Mike Zabrocki, ``Plethysm and the algebra of uniform block permutations,''
 \emph{Algebraic Combinatorics} \textbf{5} (2022) no. 5, 1165--1203.

\bibitem{gelfand}
Israel M. Gelfand, Daniel Krob, Alain Lascoux, Bernard Leclerc, Vladimir S. Retakh, et al..
``Noncommutative symmetric functions``. \textit{Advances in Mathematics}, Elsevier, 1995, 112, pp.218-
348
\bibitem{eg-rem} \"Omer E\u{g}ecio\u{g}lu and Jeffrey B. Remmel,
``Brick tabloids and the connection matrices between bases of 
symmetric functions,'' \emph{Discrete Applied Mathematics} \textbf{34}
(1991), 107--120.

\bibitem{corteel} Sylvie Corteel and Paweł Hitczenko, ``Generalizations of Carlitz Compositions," \textit{Journal of Integer Sequences} \textbf{10} (2007) Article 07.8.8.

\bibitem{KLpsym} Aditya Khanna \& Nicholas Loehr, ``Transition matrices and Pieri-type rules for polysymmetric functions".  \textit{Algebraic Combinatorics}, Volume 8 (2025) no. 4, pp. 1085-1117 

\bibitem{KLlocal} Aditya Khanna \& Nicholas Loehr, ``A local framework for proving combinatorial matrix inversion theorems"  (2025), \url{https://arxiv.org/abs/2505.10783}

\bibitem{loehr-comb} 
Nicholas Loehr, \emph{Combinatorics} (second edition), CRC Press (2017).

\bibitem{sagan} Bruce Sagan, \emph{The Symmetric Group:
Representations, Combinatorial Algorithms, and Symmetric Functions},
Graduate Texts in Mathematics \textbf{203},
Springer Science \& Business Media, New York (2013).

\bibitem{bergeron}
François Bergeron, Gilbert Labelle, and Pierre Leroux. ``Combinatorial Species and Tree-like Structures". Trans. Margaret Readdy. Cambridge: Cambridge University Press, 1997. Print. \textit{Encyclopedia of Mathematics and Its Applications}.

\bibitem{oeis}
OEIS Foundation Inc. (2025),\textit{ The On-Line Encyclopedia of Integer Sequences}, Published electronically at \url{https://oeis.org}

\bibitem{mobius}
Allen (https://math.stackexchange.com/users/518918/allen), Möbius inversion formula?, URL (version: 2018-01-14): https://math.stackexchange.com/q/2604789

\bibitem{BnRemmel}
Desiree Beck, Jeff Remmel, and Tamsen Whitehead. ``The combinatorics of transition matrices between the bases of the symmetric functions and the Bn analogues.'' \textit{Discrete Math.}, 153:3–27, 1996

\bibitem{HuangHecke}
Jia Huang. ``A tableau approach to the representation theory of 0-hecke algebras''.
\textit{Annals of Combinatorics}, 20(4):831–868, 2016

\bibitem{SchurLift}
Chris Berg, Nantel Bergeron, Franco Saliola, Luis Serrano, Mike Zabrocki. "A Lift of the Schur and Hall–Littlewood Bases to Non-commutative Symmetric Functions". \textit{Canadian Journal of Mathematics} 66. 3(2014): 525–565.

\bibitem{noncomm}
Angela Hicks and Robert McCloskey. ''A Combinatorial Perspective on the Noncommutative Symmetric Functions''. \textit{Electronic Journal of Combinatorics}, vol. 32, no. 1

\bibitem{stanley}
Richard Stanley. \textit{Enumerative Combinatorics}, Volume 2. Cambridge Studies in
Advanced Mathematics. Cambridge University Press, 2001

\bibitem{dm-psym}
David Martinez. (2025). ``Decomposition of Polysymmetric Functions and Stack Partitions''.   \textit{Arxiv}(12 Oct 2025 v3) \url{https://arxiv.org/abs/2510.10377} 

\end{thebibliography}
\end{document}